\documentclass[11pt]{article}
\usepackage{amsmath, exscale, epsfig,amssymb, 
amsthm,
 makeidx}

\usepackage{setspace}
\usepackage{caption}

\usepackage{titlesec}
\titlespacing*{\section}{0pt}{0.4\baselineskip}{0.3\baselineskip}
\titlespacing*{\subsection}{0pt}{0.4\baselineskip}{0.3\baselineskip}

\usepackage[pdftex]{hyperref} 
\hypersetup{
    unicode      = false,     
    pdftoolbar   = true,      
    pdfmenubar   = true,      
    pdffitwindow = true,      
    pdfnewwindow = true,      
    colorlinks   = true,      
    linkcolor    = blue,      
    citecolor    = red,      
    filecolor    = blue,      
    urlcolor     = green       
}

\small
\normalsize
\usepackage[utf8]{inputenc}

\usepackage{setspace}

\usepackage[T1]{fontenc}
\usepackage[english]{babel}
\usepackage{enumerate,vmargin}

\usepackage{times}
\usepackage{amsfonts}
\usepackage{amsbsy}
\usepackage{amscd}

\usepackage{multicol}
\allowdisplaybreaks

\usepackage[all]{xy}

\usepackage{stmaryrd}
\usepackage{graphicx}
\usepackage{paralist}
\usepackage{amsfonts}
\usepackage{amssymb,mathrsfs}

\setmarginsrb{2.9cm}{2.6cm}{2.9cm}{1.7cm}{0cm}{0mm}{0cm}{9mm}

\def\build#1_#2^#3{\mathrel{\mathop{\kern 0pt#1}\limits_{#2}^{#3}}}
\def\noi{{\noindent}}

\def\cq{$\hfill \square$}
\def\un{{\bf 1}}

\newcommand{\bE}{{\bf E}}

\newcommand{\bM}{\mathbb{M}}
\newcommand{\bN}{\mathbf{N}}
\newcommand{\bbN}{\mathbb{N}}

\newcommand{\bP}{{\bf P}}

\newcommand{\cJ}{\mathcal{J}}

\newcommand{\cE}{\mathcal{E}}

\newcommand{\cH}{\mathcal{H}}

\newcommand{\cM}{\mathcal{M}}

\newcommand{\cT}{\mathcal{T}}

\def\be{\begin{equation}}

\def\ee{\end{equation}}
\def\ba{\begin{eqnarray*}}
\def\ea{\end{eqnarray*}}

\def\noi{\noindent}

\newcommand{\lgeo}{[\![}
\newcommand{\rgeo}{]\!]}
\def\cqfd{ \hfill $\blacksquare$ }

\newcommand{\eqo}{\! = \! }
\newcommand{\geqo}{\! \geq \! }
\newcommand{\leqo}{\! \leq \! }
\newcommand{\ino}{\! \in \! }

\newcommand{\bss}{\mathbf s}
\newcommand{\bxx}{\mathbf x}

\newcommand{\bbR}{\mathbb{R}}
\newcommand{\leko}{\! < \! }
\newcommand{\geko}{\! > \! }

\def\btt{\mathbf{t}}

\def\bC{\mathbf{C}}

\newcommand{\bc}{\mathbf c}

\newcommand{\bnu}{\boldsymbol \nu}
\newcommand{\bmu}{\boldsymbol \mu}
\newcommand{\bzeta}{\boldsymbol \zeta}
\def\cJ{\mathcal{J}}

\def\cH{\mathcal{H}}

\def\bxx{\mathbf{x}}
\def\byy{\mathbf{y}}

\def\bnn{\mathbf{n}}
\def\bmm{\mathbf{m}}
\def\bgam{{\boldsymbol{\gamma}}}

\def\ee{\mathrm{e}}

\def\sh{\mathrm{sinh}}

\def\bM{\mathbf{M}}
\def\bss{\mathbf{s}}
\def\bxx{\mathbf{x}}
\def\byy{\mathbf{y}}

\def\bcc{\mathbf{c}}
\def\btt{\mathbf{t}}
\def\bnn{\mathbf{n}}
\def\bmm{\mathbf{m}}
\def\tee{\mathtt{e}}
\def\tmm{\mathtt{m}}
\def\tnn{\mathtt{n}}
\def\tZ{\mathtt{Z}}
\def\tC{\mathtt{C}}
\def\tU{\mathtt{Uni}}
\def\tJJ{\mathtt{J}}

\def\grob{\mathtt{G}\overline{\mathtt{ro}}}

\newcommand{\PPi}{{\boldsymbol{\Pi}}}
\newcommand{\bvaro}{{\boldsymbol{\varrho}}}

\newcommand{\bGam}{{\boldsymbol{\Gamma}}}

\newcommand{\pcH}{\mathtt{p}_{ H}} 
\newcommand{\tdd}{\mathtt{d}} 
\newcommand{\tgg}{\mathtt{g}} 
\newcommand{\tss}{\mathtt{s}} 
\newcommand{\trr}{\mathtt{r}} 
\newcommand{\txx}{\mathtt{x}} 
\def\tHH{\mathtt{H}}
\def\bcE{{\boldsymbol{\mathcal{E}}}}

\def\bH{\mathbf{H}}

\newtheoremstyle{thmstyl}
{3.5pt} 
{2.5pt} 
{\em} 
{} 
{\bfseries} 
{.} 
{.5em} 
{} 
\theoremstyle{thmstyl}

\newtheorem{theorem}{Theorem}[section]

\newtheorem{lemma}[theorem]{Lemma}
\newtheorem{proposition}[theorem]{Proposition}

\newtheoremstyle{dfstyl}
{3.5pt} 
{2.5pt} 
{} 
{} 
{\bfseries} 
{.} 
{.5em} 
{} 
\theoremstyle{dfstyl}

\newtheorem{definition}[theorem]{Definition}
\newtheorem{notation}[theorem]{Notation}

\newtheorem{remark}[theorem]{Remark}
\newtheorem{example}[theorem]{Example}

\newcommand{\eqnsection}{
\renewcommand{\theequation}{\arabic{section}.\arabic{equation}}
    \makeatletter
    \csname  @addtoreset\endcsname{equation}{section}
    \makeatother}
\eqnsection

\title{ \textbf{The continuous Derrida-Retaux branching process \\ in the Brownian CRT}}

\author{Thomas \textsc{Duquesne}
\thanks{LPSM, Sorbonne Universit\'e, Campus Pierre et Marie Curie, 4 place Jussieu, F-75252 Paris Cedex 05, France.
Email: thomas.duquesne@sorbonne-universite.fr}
\and Zhan \textsc{Shi}
\thanks{State Key Laboratory of Mathematical Sciences, AMSS, Chinese Academy of Sciences, 100190 Beijing, China.
Email: shizhan@amss.ac.cn}
}

\begin{document}

\maketitle

\vspace{-7mm}

\begin{abstract} 
The Derrida-Retaux continuous branching process, which is conjectured to be
the scaling limit of the corresponding discrete model and of many other critical hierarchical renormalization models, 
is a process of cells evolving inhomogeneously on the time interval $[0, 1)$ via linear growth of each cell and 
independent splitting of cells at a specific rate. In this article, we show that the
Derrida-Retaux continuous branching process is, on the one hand, encoded by a family of processes converging to a Brownian motion and on the other hand, fully obtained from the continuum Brownian tree encoded by the limiting 
Brownian path by time reversal and through length erasure of this tree. 
This direct representation explains the specific laws governing the significant quantities featuring the model 
(total mass, number of cells, law of large numbers) and provides a better understanding of its dynamics.

\smallskip

{\small 

\noi

\smallskip

\noi
}
\end{abstract}

\section{Introduction}

The \emph{Derrida-Retaux continuous branching process} (\emph{D-R branching process} for short) describes the evolution on the time interval $[0, 1)$ of a population of cells which grow and split according to a dynamics that is described informally as follows.  
\begin{compactenum}

\smallskip

\item[$(a)$] At time $0$, there is one cell with mass $\smash{ x\ino \bbR_+}$.

\smallskip

\item[$(b)$] A cell of mass $y$ at time $t\in [0, 1)$ that does not split between times $t$ and $\smash{t^\prime\ino [t, 1)}$ has a mass equal to $\smash{y + t^\prime \! -\! t}$ at time $\smash{t^\prime}$, i.e., \emph{the growth of cells is linear}. 

\smallskip

\item[$(c)$] A cell of mass $y$ at time $t\in [0, 1)$ instantly splits at time $t$ with rate $\smash{\frac{2y}{(1-t)^2}}$.

\smallskip

\item[$(d)$] When a cell division occurs, the mother-cell of mass $y$ instantly produces a dauther-cell of random mass $\smash{y'}$. The new mass of the mother-cell is then $\smash{y''\eqo y\! -\! y'}$, i.e., \emph{there is no loss of mass}. Moreover, the splitting ratio $\smash{y'/y}$ is \emph{uniformly distributed} on $[0, 1]$. 

\smallskip

\item[$(e)$] \emph{Splitting times and splitting ratios are independent} and \emph{cells behave independently} too. 

\smallskip

\end{compactenum}
Even though no rigorous proof has been available yet, D-R branching process is expected to be the scaling limit of the discrete-time renormalization model at criticality studied in \cite{derrida-retaux} and conditioned on survival (see \cite{4authors} and \cite{bz_survey} and  see also a rigorous proof of the convergence for related models in Hu, Mallein and Pain~\cite{HMP} or in Bai and Zhang~\cite{bai-zhang}). 
It is also expected to be the limit of more general hierarchical renormalization models. This is the motivation for studying its main properties in \cite{DerDuqShi24}. 

Let us briefly recall some results from \cite{DerDuqShi24}. To this end, we need to introduce some notation. 
The number of cells in the population at time $t$ is finite and we denote it by $\bnn_t$. We also introduce the following. 
\begin{equation}
\label{DRsimple}
\smash{\bxx (t)\! = \! \big( \bxx_k (t)\big)_{1\leq k \leq \bnn_t} \quad \textrm{and} \quad \bmm_t= \sum_{^{1\leq k \leq \bnn_t}} \bxx_k(t)} 
 \end{equation}

\vspace{5pt}
 
\noi 
resp.~the list of the masses of the cells and the total mass of the cell population at time $t$ (we explain below how to list conveniently the cells of the population). 
Since each cell grows linearly at unit speed between two splitting times, we necessarily get 
\begin{equation}
\label{mnequation}
\smash{\bmm_t = x+ \! \int_0^t \!\! \bnn_s \, \mathrm d s\;  , \quad t\ino [0, 1) \; .}
\end{equation}

\vspace{5pt}

\noi
The splitting rate explodes at time $1$. Therefore, $\smash{\lim_{t\uparrow 1} \bnn_t\eqo \lim_{t\uparrow 1} \bmm_t\eqo \infty}$ and 
as conjectured in \cite{bz_survey} and proved in \cite{DerDuqShi24},  
\begin{equation}
\label{limitduree}
\smash{\big( (1\!-\! t) \bmm_t , (1\!-\! t)^2 \bnn_t\big)\!  \overset{\textrm{proba}}{\underset{t\to 1^-}{-\!\!\!-\!\!\! \longrightarrow}} (2\bzeta, 2\bzeta) \; \,  \textrm{where} \; \,  \bE \big[ \mathrm{e}^{\! -2\lambda \bzeta }\,  \big] \eqo  \frac{\lambda}
{\sh^2 (\sqrt{\lambda})}\mathrm{e}^{-2x(\sqrt{\lambda} \coth (\sqrt{\lambda})-1)  }.}
\end{equation} 
As showed in \cite{DerDuqShi24} the mass of each cell goes to $0$ as $\smash{t\! \uparrow \! 1}$ but they can be renormalized to obtain the following law of large numbers (Theorem 3.13 in \cite{DerDuqShi24}):
\begin{equation}
\label{LLN}
\smash{\frac{1}{\bnn_t} \! \sum_{^{1\leq k\leq \bnn_t}} \!\!\! \delta_{\frac{1}{1-t} \bxx_k (t)} \, \overset{\textrm{proba}}{\underset{t\to 1^-}{-\!\!\!-\!\!\! \longrightarrow}} \,  \boldsymbol{\gamma}_2 (\mathrm dy)\! :=\! 4ye^{-2y} \mathrm d y.}
\end{equation}
Here the limit holds in probability on the space $\mathcal M_1 (\bbR_+)$ of the laws on $\bbR_+$ equipped with the Polish topology of weak convergence. 

The proofs of (\ref{limitduree}) and (\ref{LLN}) rely on the one hand on a Poisson coupling which is recalled below. It explains why (\ref{limitduree}) and (\ref{LLN}) hold in probability. They also rely on the fact that the time-changed process 
${\smash\bxx^*(s)\! :=\! \big(\mathrm{e}^s\mathbf x_k (1\! -\! \mathrm{e}^{-s})\,  ; \, 1\! \leq \! k \! \leq \! \mathbf n_{1-\mathrm{e}^{-s}} \big)}$, $\smash{s\ino \bbR_+}$, is a time-homogeneous Markovian growth-fragmentation process in the sense of Bertoin~\cite{bertoin}. 
 Although it does not fit in the general laws of large numbers obtained in 
 Bertoin and Watson~\cite{bertoin-watson18}, it turns out to be a solvable model and (\ref{limitduree}) and (\ref{LLN}) are obtained thanks to explicit computations using the infinitesimal generator of $\bxx^*$ which is known with sufficient accuracy.  
 
 \smallskip

\noi
\textbf{Goals and main ideas.} This article aims at providing a representation of the D-R branching process in terms of the Brownian CRT. In addition to providing direct probabilistic proofs of (\ref{limitduree}) and (\ref{LLN}), we claim that this representation, which is geometrically simple, explains the main features of the D-R branching process. Before discussing precisely our results, let us first outline the main ideas underlying our work.

 The first idea consists in adding to each cell a \emph{nucleus} from which the mass expands. Indeed we view each cell of the population as an interval whose length is its mass. The nucleus of the cell is simply a distinguished point of the interval (See Figure \ref{1CellNucleus} and Definition \ref{popcelldef} below). 

\begin{figure}[htb]
\centering
 \begin{minipage}{0.9\textwidth}
 \centering 
 \includegraphics[width=0.75\textwidth, height=0.18\textheight]{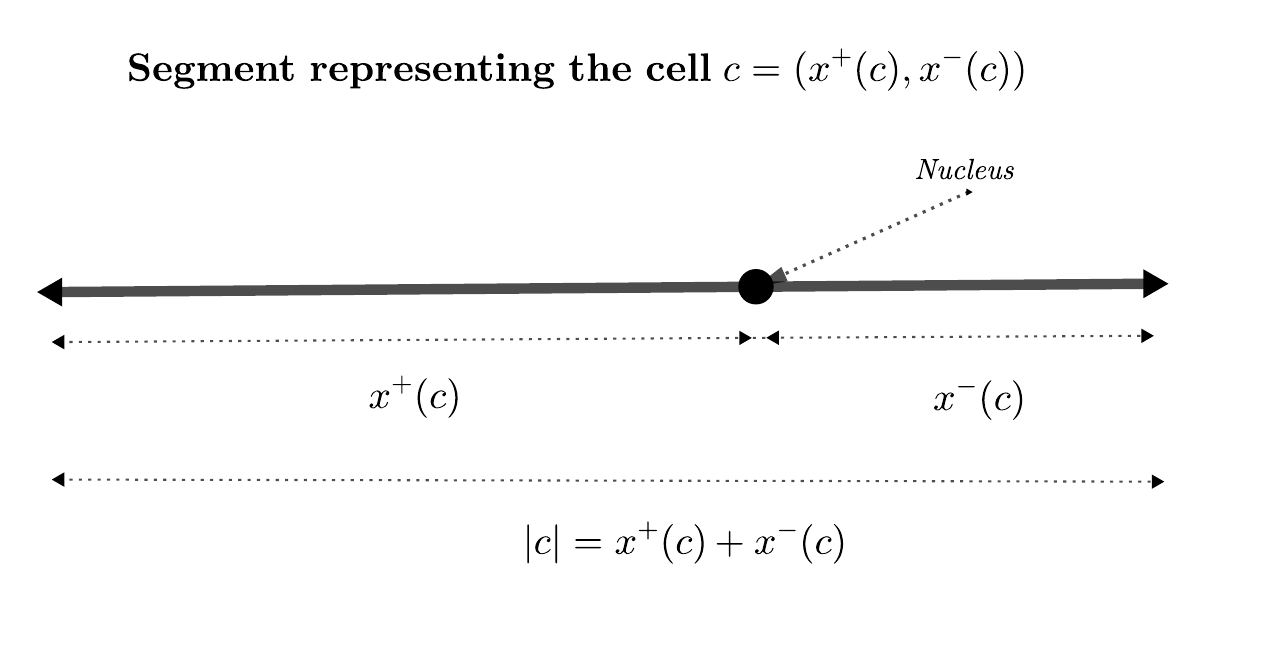}
\vspace{-0.5cm}
\caption{{\footnotesize \emph{A cell $c$ is a segment whose nucleus is a distinguisged point (the black dot) at distance $\smash{x^{_+}\! (c)}$ from the endpoint on the left-hand side of the segment and at distance $\smash{x^{_-}\! (c)}$ from the endpoint on the right-hand side of the segment. The mass of the cell is the total length of the segment: $\! \smash{|c|\eqo  x^{_+}\! (c)+x^{_-}\! (c)}$. \cq } }
\label{1CellNucleus}}
 \end{minipage}
\end{figure}  

\vspace{-10pt}

The cell then grows at speed $\smash{ \frac{_1}{^2}}$ both on the left-hand side and on the right-hand side 
of the nucleus. Namely, let us suppose that at 
time $t$ the cell is represented by the interval $[0, y]$ with nucleus $z\ino [0, y]$ and let us suppose that the cell does not split up to time $\smash{t' \ino [t, 1)}$. Then the cell at time $\smash{t'}$ can be represented by the intervall $\smash{[0, y+t'\! -\! t]}$ and its nucleus is then situated at $\smash{z'\eqo  z+ \frac{_1}{^2} (t'\! -\! t)}$. The overall dynamics of the D-R process equipped with nuclei remains as in $(a$-$e)$ but when a cell splits, we need to specify the following: \emph{the mother-cell is the cell that keeps the nucleus and at the splitting time, the nucleus of the daughter-cell is situated 
at the endpoint corresponding to the place at which the division occurs} and it eventually moves inwards as  the daughter-cell grows. See Figure \ref{SplitGrowth}. More precise definitions are provided below. 
 
The second idea (already introduced in \cite{DerDuqShi24}) consists in putting end-to-end the intervals representing the cell population, in order to produce a cumulated version of them. Since the cells are independent and since the splitting ratios are uniform, we can couple them to derive the splitting times and ratios from a single Poisson point process on $\smash{[0, 1) \! \times \! \bbR_+}$ with intensity $\smash{2\mathrm d t \, \mathrm d y /(1\! -\! t)^2}$. It also induces an indexation of the cell that is consistent with this Poisson coupling. See Figure \ref{PoissCoupl} (and see below for a more formal definition).  

\begin{spacing}{1}
\begin{figure}[htb]
\centering

\begin{minipage}{0.9\textwidth}
 \centering 

\includegraphics[width=0.9\textwidth, height=0.23\textheight]{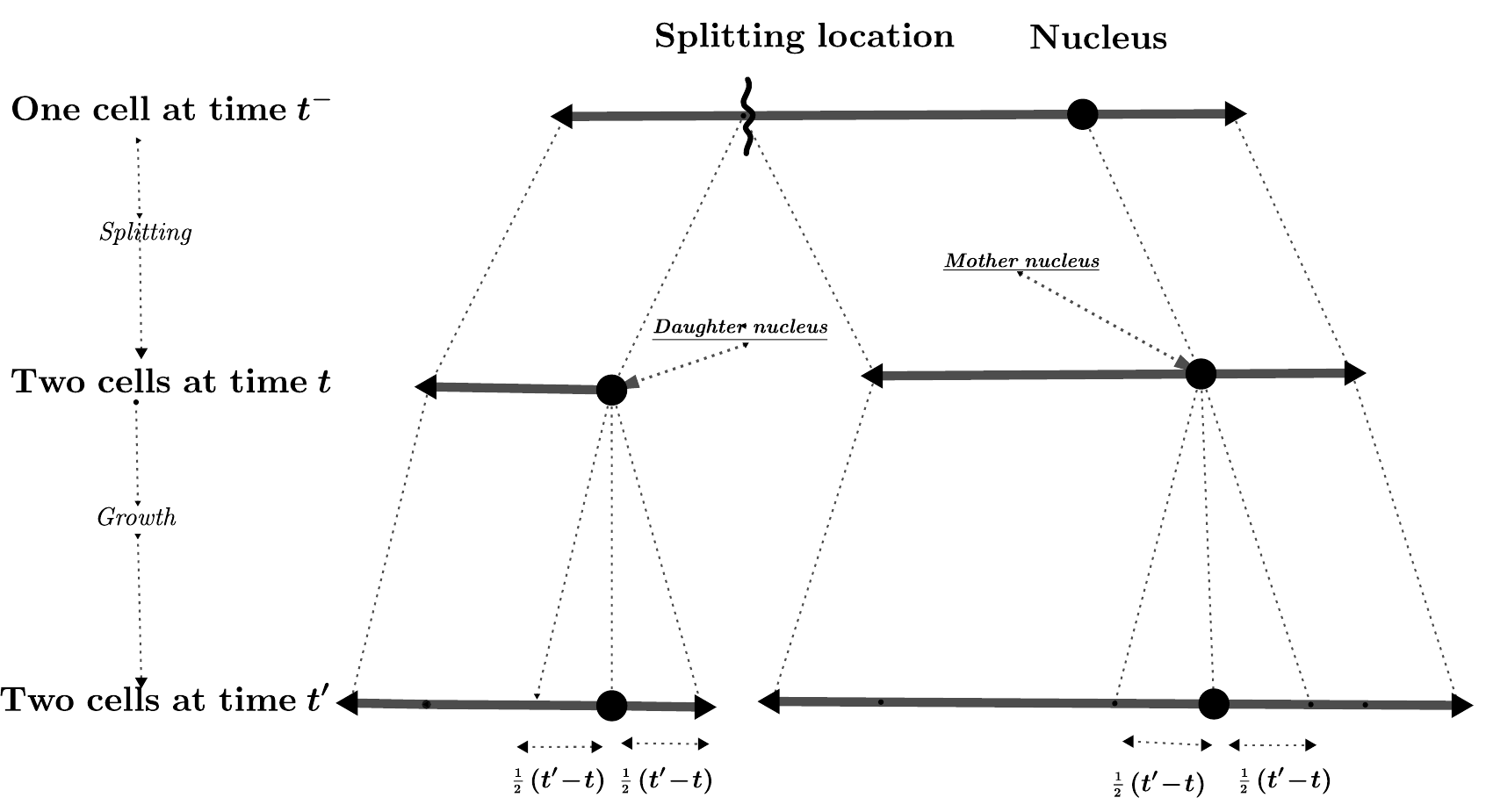}
\caption{
{\footnotesize \emph{One single cell (equipped with a nucleus) splits at time $t^-$. Two cells are produced: 
the mother cell is the one that keeps the nucleus (here, on the right-hand side). The other cell is called the daughter cell: its nucleus is created at the endpoint where the splitting occurs (here, the endpoint on right-hand side of the daughter cell). We assume here that there is no splitting event between time $t$ and time $t'$ and the two cells grow linearly at speed $1/2$ on each side of their respective nuclei.} \cq }
\label{SplitGrowth}}
\end{minipage}
\end{figure}  
\end{spacing}
  
\vspace{-10pt}  
 
  The reason for introducing a nucleus in each cell is that it allows a cell population to be \emph{encoded by a 
zigzag function}, i.e., a continuous and piecewise affine function with slopes equal to $1$ or to $-1$. More precisely, 
a 
cell, which is represented by the interval $[0, y]$ and whose nucleus is $z\ino [0, y]$, is encoded by the function 
$\smash{t\ino [0, y] \! \mapsto \! 
t\! -\! 2(t\! -\! z)_+}$, where $\smash{(\, \cdot \, )_+}$ denotes the positive part function. 
The cell population following the D-R dynamics as discussed above is then encoded by the zigzag function obtained 
by concatenating the coding functions of each cell (see Figure \ref{ZigzagCellPop}). In 
Theorem \ref{cvzigzag} we state that such a zigzag process, when suitably rescaled, converges as $\smash{t\! \uparrow \! 1}$ to a Brownian path whose law 
 is specified below. Moreover, Theorem \ref{recovercell} asserts that the whole D-R branching process can be recovered from the limiting Brownian path of Theorem \ref{cvzigzag} by a procedure called length erasure. or trimming, wich is a simple geometric transform of the Brownian CRT coded by the limiting path.

\smallskip

\noi
\textbf{Detailled statements.} Let us next explain more precisely our main results. To this end, we need to introduce some notation on cells and cell populations. 

\begin{definition} 
\label{popcelldef}(\emph{Cell populations})
\noindent
$(a)$ A \emph{cell} $c$ $\eqo$ $\smash{ (x^+  (c) , x^- (c)) }$ $\ino$ $\smash{\bbR_+^2}$, 
is a pair of nonnegative real numbers: $\smash{x^+(c)}$ is the mass situtated on the left-hand side of nucleus and $\smash{x^-(c)}$ the mass situated on the right-hand side of the nucleus, the mass of $c$ being $\smash{|c| \! :=\! x^+(c)+x^-(c)}$. See Figure \ref{1CellNucleus}.

\smallskip

\noindent
$(b)$ A \emph{population} of $n$ cells is then given by a vector $\smash{\mathbf c \eqo (c_k)_{1\leq k\leq n} \ino (\bbR_+^2)^n\! \equiv\!  \bbR_+^{2n}}$. We denote by $ \mathtt n(\mathbf c)\eqo n$ the 
\emph{size} of the cell population $\bcc$ and we denote by $\smash{\mathtt m (\mathbf c) \eqo \sum_{1\leq k \leq n} |c_k|} $ its \emph{total mass}.

\smallskip

\noindent $(c)$
  The \emph{cumulated version of the population} $\smash{\bcc\eqo (c_k)_{1\leq k\leq n}}$ is denoted and defined 
by $\smash{\overline{\mathbf c}} $ $ \eqo $ $\smash{(z_{j})_{1\leq k\leq 2n}}$ $ \in$ $\smash{ \bbR^{2n}_+}$ where $ \smash{z_{2k} \eqo  |c_1|+ \ldots +|c_k| }$ and 
$\smash{ z_{2k-1}\eqo  z_{2k} \! -\! x^{-}(c_k)}$, $\smash{1\leqo k\leqo n }$. 
Here, it is convenient to set 
$z_0\eqo 0$. Note that $\smash{z_{2n}\eqo \mathtt{m} (\mathbf c)}$. See Figure \ref{ZigzagCellPop}.

\smallskip

\noindent
$(d)$ The space of cell populations is the disjoint union $\smash{\mathtt{Cell} \! :=\! \bigsqcup_{n\in \bbN^*}  \bbR^{2n}_+} $. It is endowed with the sum topology which makes it a complete 
locally compact space with countable basis.  \cq 
\end{definition}
Linear growth and splitting of cells formally correspond to
the following functions on $\mathtt{Cell}$.

\begin{definition}
\label{GroDivdef} (\emph{Growth and division}) 
Let $\smash{\mathbf c \eqo (c_k)_{1\leq k\leq n} \ino \mathtt{Cell}}$ whose cumulated version $\smash{\overline{\mathbf c}}$ is denoted by $\smash{(z_{j})_{1\leq k\leq 2n}}$. 

\smallskip

\noindent
$(a)$ For all $\smash{t\ino \bbR_+}$ we denote by $\smash{\mathtt{Gro}_t (\mathbf c)\eqo \big( \big(\frac{_1}{^2} t + x^{+} (c_k) ,  \frac{_1}{^2} t + x^{-} (c_k)  \big)\big)_{1\leq k\leq n}}$ the cell population  obtained from 
$\smash{\mathbf c}$ after a linear growth of each cells during a timespan $t$. 

\smallskip

\noindent
$(b)$  The cell population resulting from a division of $\mathbf c$ at location $\smash{y\ino \bbR_+}$ is defined as 
follows: if $\smash{y\geko \mathtt{m} (\mathbf c) }$, we simply set $\smash{\mathtt{Div}_y(\mathbf c)\eqo  \mathbf c}$ and 
if $\smash{y\ino [0, \mathtt{m} (\mathbf c) ]}$, 
then we set $\smash{\mathtt{Div}_y(\mathbf c)\eqo \mathbf c'}$ where 
the cumulated version of 
$\smash{\mathbf c'}$ is the vector with $2n+2$ nondecreasing components obtained by inserting twice the value $y$ in 
$\smash{(z_j)_{1\leq j\leq 2n}}$. Namely $\smash{\overline{\mathbf c}'\eqo  (z_1, \ldots , z_j, y, y , z_{j+1}, \ldots, z_{2n})}$ if $\smash{y\ino [z_j ,  z_{j+1}]}$ for some $\smash{1\leqo j\leko 2n}$ and $\smash{\overline{\mathbf c}'\eqo  (y,y, z_1, \ldots , z_{2n})}$ if $\smash{y\ino [0, z_1]}$. \cq 
\end{definition}

We fix an initial cell $\smash{\bcc(0)\eqo (x^+\!\! , x^-)}$ whose mass is $\smash{x\eqo x^+\! + x^-}$. The D-R branching process (with nuclei) is then completely determined once we have fixed the sequence 
of its \emph{splitting times} $\smash{(\btt_n)_{n\in \bbN^*}}$, which satisfy $\smash{0\leko \btt_n \leko \btt_{n+1} \leko \sup_{p\in \bbN^*} \btt_p \eqo 1}$, and the sequence of its \emph{splitting locations} $\smash{(\byy_n)_{n\in \bbN^*}}$ that are derived as follows from a Poisson point process $\smash{\PPi}$ on $\smash{[0, 1) \times \bbR_+}$ with intensity $\smash{2\mathrm d t \, \mathrm d y /(1\! -\! t)^2}$: 
\begin{compactenum}

\smallskip

\item[$(i)$] we easily check that there is a unique $\smash{\bbR_+\! }$-valued 
continuous process $\smash{(\bmm_t)_{t\in [0, 1)}}$ that is solution for all $t\ino [0, 1)$ to 
\begin{equation}
\label{mtntcompati}
\smash{\underline{\texttt{Eq} (\bcc(0), \PPi)}:} \quad \smash{ \bmm_t  =  \mathtt{m} (\bcc (0))+\!  \int_{0}^t \!\!  \big( 1+ \# \big\{ (r,y)\ino \PPi: r\leko s \;  \textrm{and} \; y \leko \bmm_r \big\}\big) \mathrm d s\,  ;}
\end{equation}
\item[$(ii)$] the sequence $\smash{(\btt_n, \byy_n)_{n\in \bbN^*}}$ is such that 
$\smash{\{ (\btt_n, \byy_n); n\ino \bbN^*\}\eqo \{ (t,y)\ino \PPi: y\leko \bmm_t\}}$; 

\smallskip

\item[$(iii)$] the D-R branching process with initial cell $\bcc(0)$ is then defined recursively 
as the process $\smash{t\ino [0, 1) \! \mapsto \! \bcc (t) \ino \mathtt{Cell}}$ such that for all $\smash{t\ino [0, \btt_1]}$, $\smash{\bcc(t)\eqo \mathtt{Gro}_t (\bcc(0)) }$ and such that for all $\smash{n\ino \bbN^*}$ and all $\smash{t\ino (\btt_n, \btt_{n+1} ]}$, 
\end{compactenum} 
\begin{equation}
\label{nuclDRbrdef}
\smash{\bcc(t)= \mathtt{Gro}_{t-\btt_n} \big( \mathtt{Div}_{\byy_n} (\bcc(\btt_n))\big) . }
\end{equation}  
See Figure \ref{PoissCoupl}. We always keep using notation $\smash{\bcc(t)\eqo (\bcc_k(t))_{1\leq k\leq \bnn_t}}$ for the D-R branching process: here, 
$\smash{\bnn_t}$ is the finite size of the cell population and $\smash{\bcc_k(t)}$ stands for the $k$-th cell in the indexation induced by the Poisson coupling. 
First note that each time a division occurs, the size of the population increases by $1$. 
Then clearly $\smash{\bnn_t \eqo 1+\# \big\{ (s,y)\ino \PPi: s\leko t \;  \textrm{and} \; y \leqo \bmm_s \big\}}$ 
and we observe that 
$\smash{\bmm_t}$ is the total mass of $\smash{\bcc(t)}$, namely $\smash{\mathtt{m} (\bcc(t))\eqo \bmm_t}$. Therefore 
$\smash{\texttt{Eq} (\bcc(0), \PPi)}$ is (\ref{mnequation}). The process $\bcc(\cdot)$ clearly follows the dynamics informally described at the beginning of the section and when we forget about nuclei and we only consider the mass of the cells, then $\smash{\bxx(t)\! :=\! (|\bcc_k(t)|)_{1\leq k\leq \bnn_t}}$, $\smash{t\ino [0, 1)}$, is the D-R branching process with inital mass $\smash{x\eqo x^+\! +x^-}$ as introduced genuinely in \cite{DerDuqShi24} (and as mentioned at the beginning of the introduction). 

\begin{figure}[htb]
\centering
\centering
  \begin{minipage}{0.9\textwidth}
 \centering 
\includegraphics[height=0.28\textheight, width=0.75\textwidth]{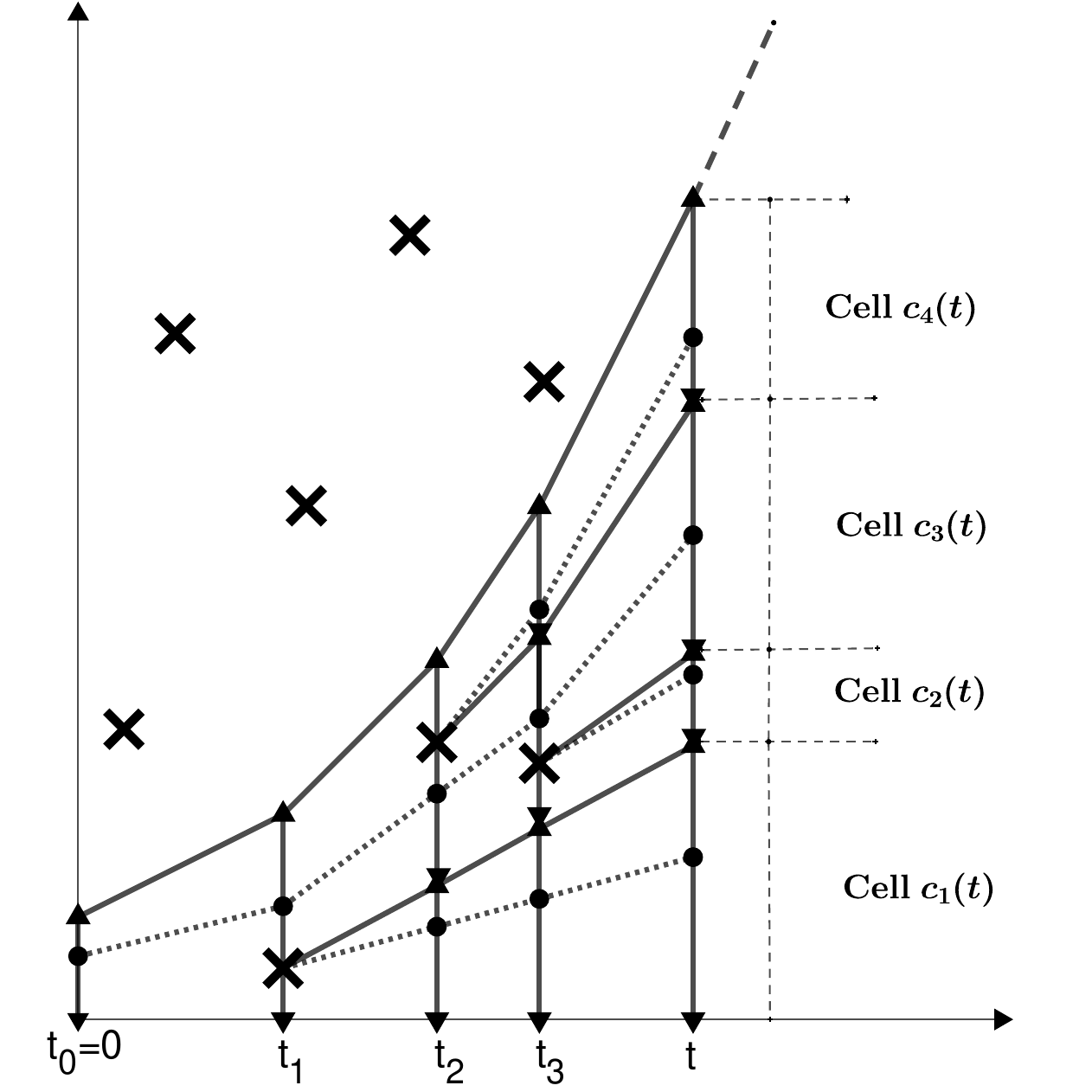}
\caption{{\footnotesize \emph{An example of the branching system associated with a single initial cell $\smash{\mathbf c(0)\eqo (x^+\! , x^-)}$ and coupled with a P.p.p.~$\PPi$ on $\smash{[0, 1) \times \bbR_+}$ whose points are represented by bold crosses.  
Time corresponds to the horizontal axis and cells are put end-to-end vertically: their endpoints are delimited by arrows whose evolution follows continuous lines. Nuclei are represented by thick bullets $\bullet$ and their evolution follows dashed lines. There are exactly three division events between times $0$ and $t$ indicated as $t_1, t_2$ and $t_3$. Thus at time $t$ there are four cells which are listed in increasing order from bottom to top.}  \cq
 }}
\label{PoissCoupl}
\end{minipage}
\end{figure}

 We next explain how to associate with each cell population $\bcc(t)$ an encoding process, how to rescale it and how to pass to the limit. Our framework in terms of processes is the following:  
we consider continuous functions $\smash{H\!  :\!  t\ino \bbR_+ \! \mapsto \! H_t \ino \bbR}$ such that $\smash{H_0\eqo 0}$ and with a (possibly infinite) \emph{lifetime} denote by $\zeta \ino [0, \infty]$. It means here that if $\zeta \leko \infty$, then $H$ is constant on $[\zeta, \infty)$. We simply denote by $\bC$ the space such continuous functions 
$(H, \zeta)$ provided with a lifetime. We endow this space with the product topology of the usual convergence  in $[0, \infty]$ (for the lifetimes) and the uniform convergence on every compact intervals, which makes of $\bC$ a Polish space.  We denote by 
$\partial$ the unique function of $\bC$ with null lifetime and we call it the \emph{null lifetime function}. See Definition  \ref{proclifetimedef} for more details. To any $\smash{\bcc \eqo (c_k)_{1\leq k\leq n}\ino \mathtt{Cell}}$ whose cumulated version $\smash{\overline{\bcc}}$ is denoted by $\smash{(z_j)_{1\leq j\leq 2n}}$, we associate an encoding continuous function $\smash{\mathtt{Z}_\cdot (\bcc)}$ whose lifetime is $\smash{z_{2n}\eqo \mathtt{m} (\bcc)}$ (the total mass of $\bcc$) and that is given by 
\begin{equation} 
\smash{ \forall z\ino [0, \mathtt{m} (\bc) ] , \quad \mathtt{Z}_z(\bcc)= \int_0^z \!\!  \Big(\!\! \!\!  \!\!  \sum_{^{\quad 1\leq j\leq 2n}} \!\! \!\! \!\! (-1)^{j+1} \un_{(z_{j-1}, z_{j} ] } (y) \Big) \mathrm d y , }
  \end{equation}
where we recall that $\smash{z_0\eqo 0}$. In other word $\smash{\mathtt{Z}_\cdot (\bcc) }$ is a continuous piecewise 
affine function whose slopes are equal to $1$ or to $-1$: it starts at $0$, it goes up during a timespan equal to $\smash{x^{+} (c_1)}$, then it goes down during a timespan equal to $\smash{x^{-} (c_1)}$, then it goes up during a timespan equal to $\smash{x^{+} (c_2)}$ and it goes down during a timespan equal to $\smash{x^{-} (c_2)}$, etc. Generically (i.e., if $\smash{x^+(c_k)x^-(c_k) \geko 0}$ for all $1\leqo k\leqo n$), the function $\smash{\mathtt{Z}_\cdot (\bcc)}$ reaches $n$ strict 
local maxima at the times $\smash{(z_{2k-1})_{1\leq k\leq n}}$, and $n+1$ strict local minima at the times $\smash{(z_{2k})_{0\leq k\leq n}}$.  We call $\smash{\mathtt{Z}_\cdot (\bcc)}$ the \emph{zigzag function of  $\bcc$}. See Figure \ref{ZigzagCellPop}.
\begin{figure}[htb]
\vspace{0.3cm}
\centering
\centering
\begin{minipage}{0.9\textwidth}
\centering 
\includegraphics[width=0.8\textwidth, height=0.24\textheight]{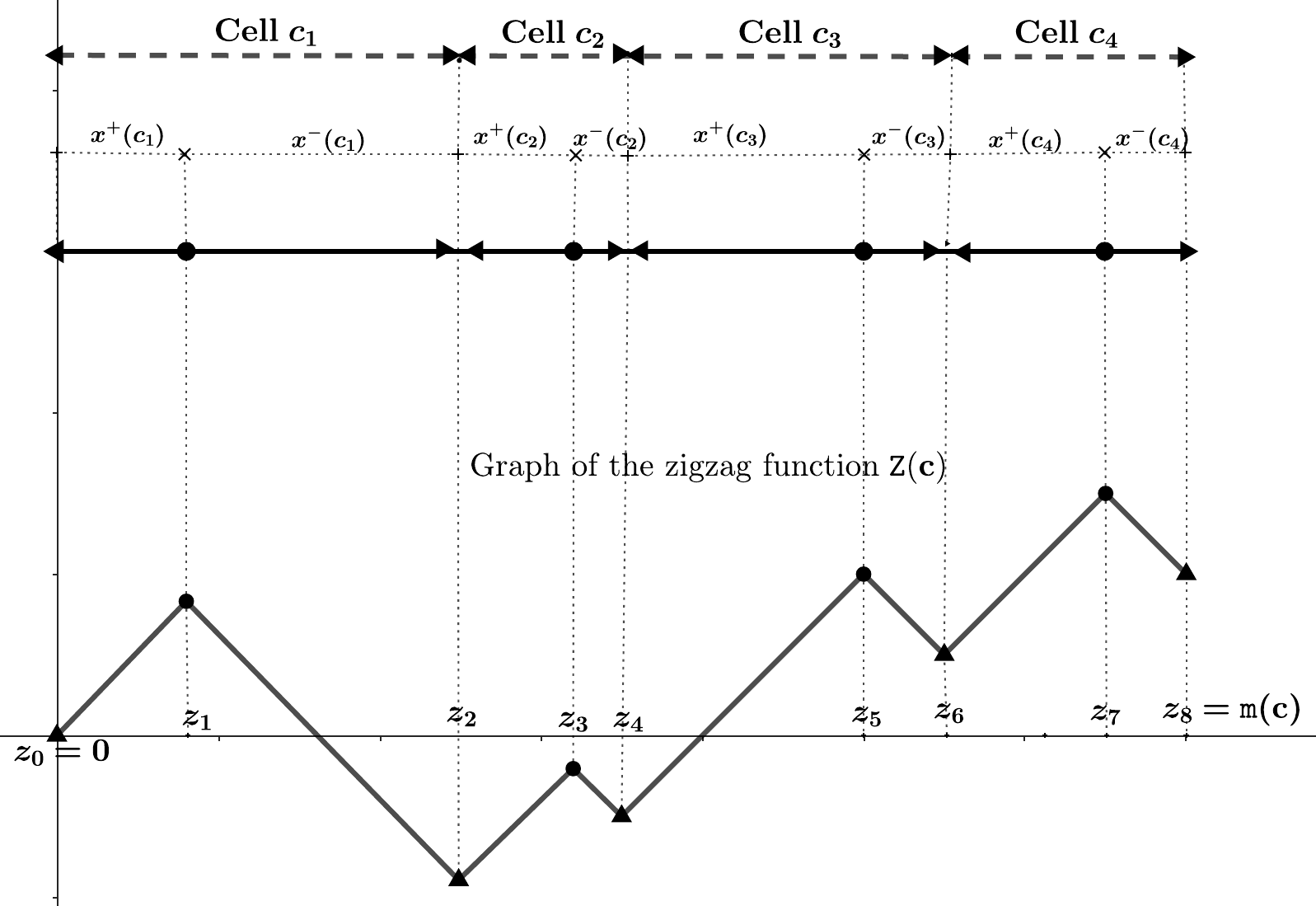}
\caption{{\footnotesize \emph{A cell population $\smash{\mathbf c}$ of $\smash{\mathtt n (\mathbf c)\eqo 4}$ cells: $\smash{c_1\eqo \big( x^+(c_1), x^-(c_1)\big)}$,
$\smash{c_2\eqo \big( x^+(c_2), x^-(c_2)\big)}$, $\smash{c_3\eqo \big( x^+(c_3), x^-(c_3)\big)}$, 
$\smash{c_4\eqo \big( x^+(c_4), x^-(c_4)\big)}$, with its cumulated version 
$\smash{\overline{\mathbf{c}}\eqo (z_j)_{1\leq j\leq 8}}$ and its corresponding zigzag function 
$\smash{\mathtt Z_\cdot (\mathbf c)}$. The lifetime of $\smash{\mathtt Z_\cdot (\mathbf c)}$ is the total mass 
$\smash{z_8\eqo \mathtt m (\mathbf c)}$ of the cell population $\smash{\mathbf c}$. 
Here each nucleus corresponds to a local maximum of $\smash{\mathtt Z_\cdot (\mathbf c)}$. }\cq
\label{ZigzagCellPop}}
}
\end{minipage}
\end{figure}  

 As we see below in Theorem \ref{cvzigzag}, a rescaled version of the $\bC$-valued process 
$\smash{t\ino [0, 1)}$ $\smash{ \! \mapsto\! }$ $\smash{ \big(\mathtt{Z}_\cdot (\bcc(t)), \bmm_t \big)}$ actually 
converges as $\smash{t\! \uparrow \! 1}$ to a limiting random process whose law is derived as follows from a $\bbR$-valued Brownian motion $\smash{B\eqo (B_t)_{t\in \bbR_+}}$ with initial value $\smash{B_0\eqo 0}$. 
To simplify notation, for all $\smash{t\ino \bbR_+}$ we set 
$\smash{\underline{B}_t\eqo \min_{s\in [0, t]} B_s}$ and $\smash{B^{_+}_t \eqo B_t \! -\! \underline{B}_t}$, which is the Brownian motion reflected on its infimum process. Let us first recall basic facts on the excursions of 
$B$ above its infimum. It is possible to take 
$\smash{-\underline{B}_\cdot}$ as local time of $\smash{B^+_\cdot}$ at $0$. We denote by $\smash{(\alpha_j, \beta_j)}$, $\smash{j\ino \mathcal J}$, the excursion intervals of $B$ above its infimum, i.e., the connected components of the open subset $\smash{\{ t\ino \bbR_+ : B^+_t \geko 0\}}$. For all $\smash{j\ino \mathcal J}$,  we denote 
the excursion corresponding to $\smash{(\alpha_j, \beta_j)}$ by 
$\smash{(H_j (\cdot), \zeta_j )} $ $ \eqo$ $\smash{ (B^+_{(\alpha_j + \cdot) \wedge \beta_j} , \beta_j \! -\! \alpha_j )}$.  We then recall that 
\begin{equation}
\label{NItomeadefi}
\smash{\textrm{$ \Big\{( -\! \underline{B}_{\alpha_j} , (H_j, \zeta_j)\Big)\,  ; \, j\ino \mathcal J \Big\}$ is a Poisson point process on $\bbR_+ \! \times \! \bC$} }
\end{equation}
with intensity $\smash{\mathrm dy\,  \bN (\mathrm d H)}$ where $\smash{\bN}$ is \emph{Ito's measure}. Here the normalization of $\smash{\bN}$ is such that 
\begin{equation}
\label{normaIto}
\smash{\forall  h\ino (0, \infty), \quad  \bN \big( \Gamma \geko h  \big) =1/h }
\end{equation}
where for all $\smash{(H, \zeta)\ino \bC}$ we denote by $\smash{\Gamma (H)\eqo \max_{t\in \bbR_+} \! H_t} $, the \emph{total height} of $H$.

\begin{definition}
\label{condheightIto} (\emph{Ito's measure conditioned by its height})
We then denote by $\smash{\bN (\, \cdot\,  | \, \Gamma \eqo h)}$ the law of the Brownian excursion 
conditionned to have height $h$. 
Namely, 
\begin{compactenum}

\smallskip

\item[$(a)$] 
$\smash{h\ino (0, \infty) \! \mapsto \!  \bN (\, \cdot\,  | \, \Gamma \eqo h)}$ is continuous with respect to weak convergence on $\bC$, 

\smallskip

\item[$(b)$]  $\smash{\bN (dH \, | \, \Gamma \eqo h)}$-a.s.~$\smash{\Gamma (H)\eqo h}$, 

\smallskip

\item[$(c)$] $\smash{\bN (\, \cdot \, ) \eqo \int_0^\infty \bN (\, \cdot\,  | \, \Gamma \eqo h) h^{-2} \mathrm d h}$. \cq 
\end{compactenum}
\end{definition}
\noindent
For all $\smash{y\ino \bbR_+}$, we also need to introduce the following hitting time: 
\begin{equation}
\label{bvarodef}
\smash{ \bvaro_{-y} \eqo \inf \big\{ t\ino \bbR_+ : B_t\eqo -y \big\}.}
 \end{equation}
From the previous decomposition of $B$ into excursions above its infimum, we get 
\begin{equation}
\label{ampliBh}
\smash{\forall y,h \ino (0, \infty), \quad  \bP \big( \!\!\!\!\! \!\!\max_{\, \quad s\in [0, \bvaro_{-y} \, ]} \!\!\!\!\!\!\!\!   B^+_{\!s} < h \big)= \mathrm{e}^{-y/h} \; . }
\end{equation} 
The scaling limit of $\smash{\mathtt{Z}_\cdot (\bcc(t))}$ is then defined thanks to the following conditioned Brownian paths. 
\begin{definition}
(\emph{Unicellular Brownian process})
\label{unicelBrodef} Let $\smash{x, x^+ \! ,\,  x^- \ino \bbR_+}$ and let $\smash{h\ino (0, \infty)}$.

\smallskip

\noindent $(a)$ If $x \geko 0$ we denote by $\smash{P^{_\downarrow}_{^{\! x,h}}}$ the law of 
$\smash{( B_{\cdot \, \wedge \bvaro_{-x}} , \bvaro_{-x})}$ under $\smash{\bP ( \, \cdot \, | \, \max_{ s\in [0, \bvaro_{-x} ]} \! B^{_+}_{{ s}} \leko  h\, )}$. If $\smash{x\eqo 0}$ we set 
$\smash{P^{_\downarrow}_{^{\! 0,h}}\eqo \delta_{\partial}}$, where $\partial$ stands for the null lifetime function. We call 
$\smash{P^{_\downarrow}_{^{\! x,h}}}$ the law of the \emph{$x$-Brownian descent with an amplitude less than $h$}.

\smallskip

\noindent $(b)$ If $x \geko 0$ we denote by $\smash{P^{_\uparrow}_{^{\! x,h}}}$ the law of $\smash{( x\! +\! B_{(\bvaro_{-x} -\, \cdot\,  )_+ } \! , \bvaro_{-x} )}$ under $\smash{\bP ( \, \cdot \, | \! \max_{ s\in [0, \bvaro_{-x} ]}  \! B^{_+}_s \leko  h  )}$. If $x\eqo 0$ we set 
$\smash{P^{_\uparrow}_{^{\! 0,h}}\eqo \delta_{\partial}}$. We call 
$\smash{P^{_\uparrow}_{^{\! x,h}}}$ the law of the \emph{$x$-Brownian rise with an amplitude less than $h$}.

\smallskip

\noindent $(c)$ The law $\smash{P_{\! x^+ \! ,\, x^-\! ,\, h}}$ of the \emph{$h$-unicellular Brownian process with initial cell $\smash{(x^+\! , x^-)}$} 
is the law of the $\bC$-valued process $(\bH, \bzeta)$ which is 
the concatenation of three \emph{independent} process $\smash{(\bH^{(i)}\! , \bzeta_i ) }$, $i\ino \{ 1,2,3\}$, i.e., for all $\smash{t\ino \bbR_+}$, 
\begin{equation}
\label{unicelBMdef}
\smash{\bzeta\eqo \bzeta_1\! + \bzeta_2\! + \bzeta_3 \quad \textrm{and} \quad \bH_t\eqo \bH^{(1)}_t\! +\bH^{(2)}_{(t-\bzeta_1)_+}\!\!  +\bH^{(3)}_{(t-\bzeta_1-\bzeta_2)_+},}
\end{equation}
where $\smash{(\bH^{(1)}\! , \bzeta_1) \! \overset{\textrm{law}}{=} \! P^{_\uparrow}_{^{\! x^+\! ,\,  h}}}$,  $\smash{(\bH^{(2)}\! , \bzeta_2)  \! \overset{\textrm{law}}{=} \! \bN (\, \cdot \,  | \,  \Gamma\eqo h)}$ and $\smash{(\bH^{(3)}\! , \bzeta_3) \! \overset{\textrm{law}}{=}\! P^{_\downarrow}_{^{\! x^-\! , \,h}}}$. See Figure \ref{Hunieteraz1}. \cq 
\end{definition}
\begin{theorem}
\label{cvzigzag} Let $\smash{(\bcc(t))_{t\in [0, 1)}}$ be a D-R branching process with nuclei as defined above with initial cell $\smash{\bcc(0)\eqo (x^+\! ,\, x^-)}$. 
Let $\smash{\mathtt{Z}_\cdot (\bcc(t))}$ be the zigzag function encoding $\bcc(t)$. Then there is a $\bC$-valued process 
$(\bH, \bzeta)$ which is distributed as a $\smash{\frac{_1}{^2}}$-unicellular Brownian process with initial cell $\smash{\bcc(0)\eqo (x^+,x^-)}$ (see Definition \ref{unicelBrodef} $(c)$) and such that the following limit holds: 
\begin{equation}
\label{cvzigzagexpli}
\smash{\Big( \big( \mathtt{Z}_{\frac{2s}{1-t}}\big( \bcc(t) \big) \big)_{\! s\in \bbR_+} \, , \, \frac{_1}{^2}(1\! -\! t) \, \bmm_t \Big) \xrightarrow[t\uparrow 1]{\textrm{in probability in $\bC$}}\big( \bH, \bzeta \big) \, .}
\end{equation}
\end{theorem} 
\noi
\textbf{Proof.} See Section \ref{LLNDRsubsec}. Actually, it strongly relies on Theorem \ref{mainth3}.  \cqfd 

\smallskip

We next explain how to obtain $\smash{\tZ_\cdot (\bcc(t))}$, and thus $\bcc(t)$, from the limiting process $(\bH, \bzeta)$ thanks to  \emph{erasure (or trimming)} of the real tree coded by $(\bH, \zeta)$. 
To this end, we first recall what erasure of real trees is, and we next briefly recall 
the encoding of real trees by continuous functions.

\emph{Real trees} are metric spaces which extend graph trees (from the metric point of view). Namely, there are obtained by glueing intervals of the real line without creating loop. More formally a metric space $(T,d)$ is a real tree if any pair of distinct points $\sigma, \sigma' \ino T$ can be joined by a single arc whose image is denoted by $\smash{\lgeo \sigma, \sigma' \rgeo}$ 
and which turns out to be a geodesic (see Definition \ref{realtrdef} in Section \ref{codeRtreesec} for more details). We only consider compact real trees and we equip $T$ with a distinguished point $\rho\ino T$ that is viewed as a \emph{root}. For all $\smash{b\ino \bbR_+}$, the \emph{$b$-erasure} of $(T,d,\rho)$ is the subset 
\begin{equation}
\label{berased1}
\smash{\bcE_bT= \{\rho\} \cup \big\{ \sigma \ino T: \exists\  \sigma' \! \ino T \; \, \textrm{such that}\; \,  \sigma \ino \lgeo \rho, \sigma' \rgeo \; \, \textrm{and} \; \,  d(\sigma, \sigma') \geqo b\big\} \; .}
\end{equation}
Namely, $\smash{\bcE_bT }$ is the set of points in $T$ whose subtree above has  height $\geqo b$ (see Figure \ref{Treeeraz}). 
\begin{figure}[htb]
\vspace{0.3cm}
\centering
\begin{minipage}{0.9\textwidth}
\centering 
\includegraphics[width=0.65\textwidth, height=0.25\textheight]{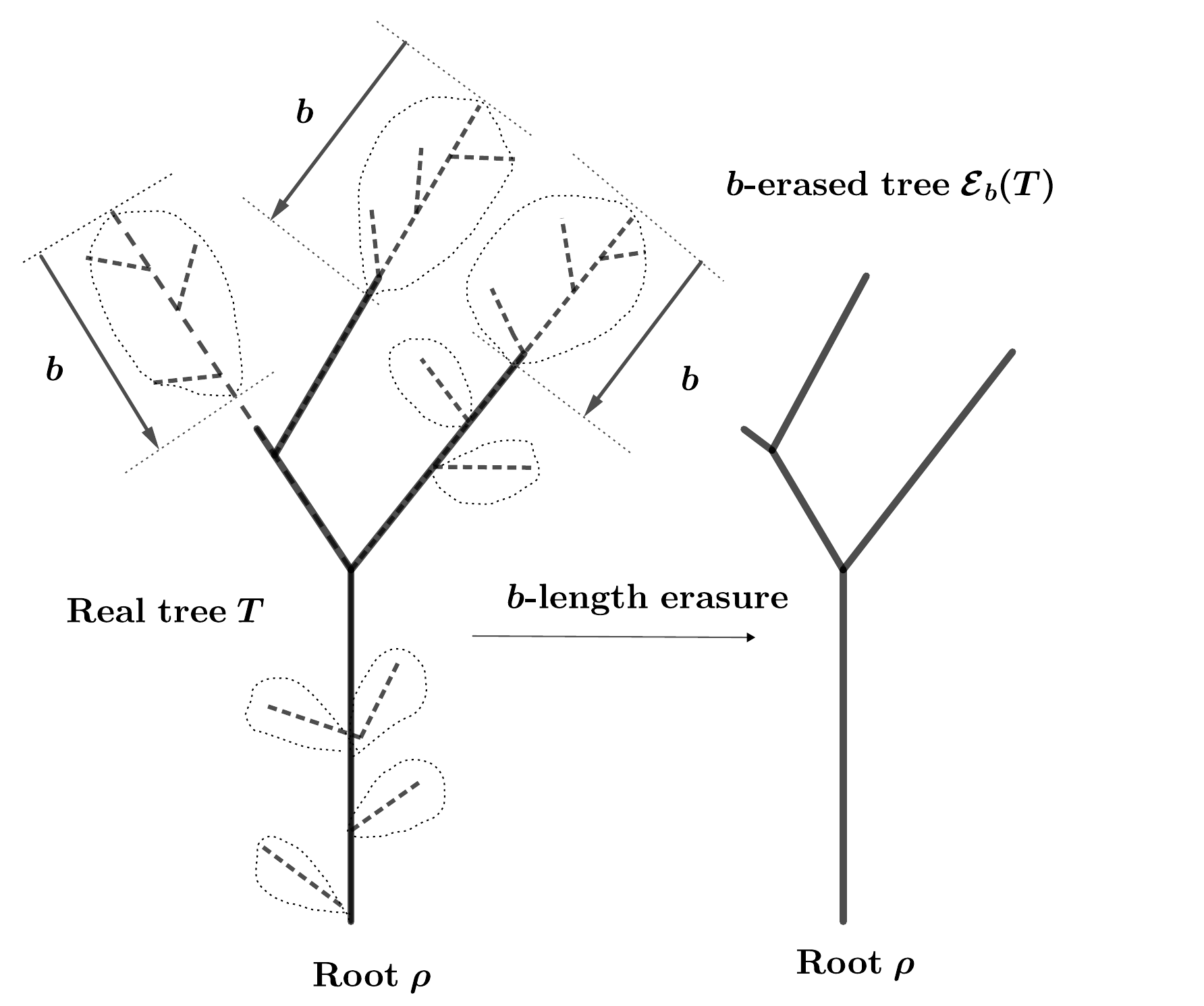}
\caption{{\footnotesize \emph{The parts of the real tree $T$ which are erased are represented with dashed lines. The $b$-erased subtree $\smash{\bcE_b (T)}$ is represented thanks to continuous lines. Each of the three leaves of $\smash{\bcE_b (T)}$ corresponds to a subtree in $T$ whose height is exactly $b$.} \cq}
\label{Treeeraz}}
\end{minipage}
\end{figure}  
We easily check the following properties (see  Evans, Pitman and Winter \cite{EPW06} for proofs). 

\smallskip

\begin{compactenum} 
\item[$(i)$] \emph{$\smash{\bcE_bT}$ is a connected compact subset of $T$ which contains $\rho$ (i.e., it is a subtree of $T$). }

\smallskip

\item[$(ii)$] \emph{$\smash{\bcE_b T}$ has a finite number of leaves (i.e., it is a finite type real tree: see Definition \ref{realtrdef}). }

\smallskip

\item[$(iii)$] \emph{$\smash{d_{\mathtt{Haus}} (\bcE_bT,T) \leqo b}$, where here $\smash{d_{\mathtt{Haus}}}$ stands for the Hausdorff distance on the space of compact subsets of $T$. }

\smallskip

\item[$(iv)$] \emph{Semigroup property: for all $\smash{b, b'\ino \bbR_+}$, } 
\begin{equation}
\label{semigrtreeera}
\smash{\bcE_{b'} (\bcE_b T)\eqo \bcE_{b+b'} T\; .}
\end{equation}
\end{compactenum} 
Erasure (or trimming) has be introduced for discrete trees in Kesten \cite{Kes86} and in Neveu \cite{Ne2}. It has been studied via encoding functions in Neveu and Pitman \cite{NP89I,NP89II} and Le Gall \cite{LG89} (see below) and it has been used to define in a projective way stable CRTs in Le Jan \cite{LJ91}. Approximation of compact rooted $\bbR$-trees via erasure has been introduced and studied systematically in Evans, Pitman and Winter \cite{EPW06}, see also D.~and Winkel \cite{DuWi19} for connections with Lévy trees. 
 
 We now explain why $\smash{\tZ_\cdot (\bcc(t))}$ results from the $\smash{\frac{_1}{^2} (1\!-\! t)}$-erasure of the real tree encoded by $(\bH, \bzeta)$ which is the limiting process appearing in (\ref{cvzigzagexpli}). To this end, for all $(H,\zeta) \ino \bC$ and for all $\smash{s,s'\ino \bbR_+}$, we introduce the following. 
\begin{equation}  
\label{mHdHdef}
\smash{m_H(s,s')= \min_{r\in [s\wedge s'\! , \, s\vee s']} H_r \quad \textrm{and} \quad d_H(s,s')= H_s + H_{s'} -2 m_H(s,s') \; .}
\end{equation} 
We easily check that $\smash{d_H}$ is a pseudo-distance on $[0, \zeta]$. We then denote by $\smash{s\sim_H s'}$ the equivalence relation $\smash{d_H(s,s')\eqo 0}$ and we define the \emph{real tree encoded by $H$} as the  
quotient space 
\begin{equation}
\label{HRtreedef}
\smash{ \cT_{H} := [0,  \zeta] / \sim_H \quad \textrm{equipped with the distance $d_H$.} }
\end{equation}
The metric space $\smash{(\cT_H, d_H)}$ is indeed a compact real tree (see D.~and Le Gall \cite{DuLG05}). We denote the canonical projection by $\smash{\pcH \! : \! [0, \zeta] \! \mapsto \! \cT_H}$ and we root $\smash{\cT_H}$ at the minimum of $H$, i.e., we define the root $\smash{\rho_H}$ of $\smash{\cT_H}$ as follows. 
\begin{equation}
\label{rootdef} 
\smash{\rho_H \! :=\! \pcH (s^* )  \quad \textrm{where $s^*\eqo \min \big\{  t\ino [0, \zeta] : H_t \eqo \min_{s\in \bbR_+} H_s \big\}$, }}
\end{equation}
which makes sense because if $\smash{H_s\eqo H_{s^*}}$, $\smash{d_H(s,s^*)\eqo 0}$ and $\smash{\pcH(s)\eqo \rho_H}$. 
For all $(H,\zeta)\ino \bC$ we use the  notation $\smash{\mathtt{Tree} (H)\eqo (\cT_H, d_H, \rho_H)}$. See Figure \ref{TreeCod}. 
\begin{figure}[htb]
\vspace{0.3cm}
\centering
\begin{minipage}{0.9\textwidth}
\centering 
\includegraphics[width=0.50\textwidth, height=0.20\textheight]{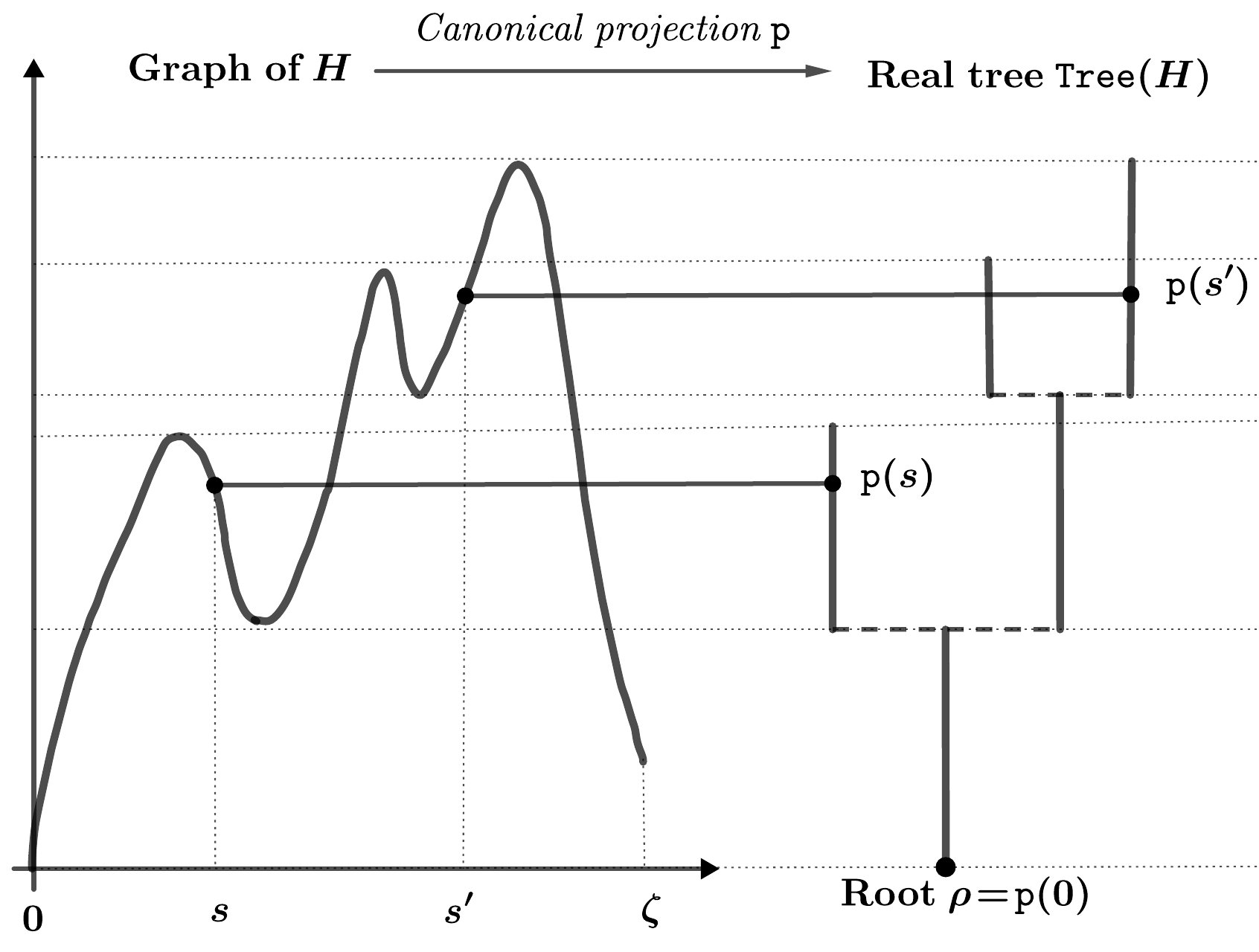}
\caption{{\footnotesize \emph{On the left: the graph of $H$. On the right: the tree encoded by $H$, which is denoted by 
$\smash{\mathtt{Tree} (H)}$. Horizontal dashed lines in $\smash{\mathtt{Tree} (H)}$ are symbolic and their length is null.} \cq }
\label{TreeCod}}
\end{minipage}
\end{figure}  
The process $\smash{\tZ_\cdot (\bcc(t))}$ is then derived from the limiting process $(\bH, \bzeta)$ in (\ref{cvzigzagexpli}) as follows: 
\begin{equation}
\label{eravsgrowtree}
\smash{ \textrm{a.s.~for all $t\ino [0, 1)$,} \quad  \mathtt{Tree} \big( \tZ_\cdot (\bcc(t))\big) \equiv \bcE_{\frac{_1}{^2} (1- t)} \big( \mathtt{Tree} (\bH)\big), }
\end{equation}
where $\equiv$ means that the rooted real trees are isometric. Roughly speaking 
linear growth corresponds, when time is reversed, to erasure. More precisely, a growth of $\mathrm d t$ of mass for each cell corresponds in reversed time to the erasure of a length of $\smash{\frac{_1}{^2} \mathrm dt}$ in the corresponding real tree. 
 
 Note that (\ref{eravsgrowtree}) is ambiguous firstly because many distinct functions encode the same real tree (see e.g.~Remark \ref{manycodings} in Section \ref{codeRtreesec}) and also because the exploration order induced on the real by the given encoding function in (\ref{eravsgrowtree}) plays an important role which is not taken into account if we only consider the real tree alone. 
We therefore need to be more specific and one way to proceed consists in defining erasure directly 
from encoding functions $(H,\zeta) \ino \bC$.
To this end, we first introduce the following notations. 
\begin{compactenum}

\smallskip

\item[$(a)$] For all $\smash{t\ino \bbR_+}$ we set 
\begin{equation}
\label{minreflecHdef}
\smash{ \underline{H}_t\eqo \min_{s\in [0, t]} H_s\quad \textrm{and} \quad H^+_t \eqo H_t \! -\! \underline{H}_t\; .}
\end{equation}

\item[$(b)$] For all $t\ino \bbR_+$, we define the $t$-shifted function (in time and space) by 
\begin{equation}
\label{shiftHdef}
\smash{\theta_t H_\cdot= H(t+ \cdot) \! -\! H_t \; 
\textrm{whose lifetime is defined by $(\zeta\! -\! t)_+$.} }
\end{equation}
\item[$(c)$] For all $b\ino (0, \infty)$ we then set $\smash{s^b_0 (H)\eqo 0}$ and 
\begin{equation}
\label{1erbtps}
\smash{s^b_1(H)\eqo \inf \Big\{ t\ino \bbR_+ : H^+_t + b= \max_{s\in [0, t]} H^+_s \Big\}}
\end{equation}
\end{compactenum}
\noindent
with the convention $\inf \emptyset \! :=\! \infty$; $\smash{s^b_1(H)}$ is introduced in Neveu and Pitman \cite{NP89I,NP89II}. We call it the \emph{first $b$-erasure time of $H$}: it is the first time $H$ completes 
an excursion of height $b$.

\smallskip

\noi
We next define the sequence $\smash{(\bss^b_n)_{n\in \bbN}}$ of successive $b$-erasure times of the limiting process $(\bH, \bzeta)$ in (\ref{cvzigzagexpli}) by setting $\smash{\bss_0^b\eqo 0}$ and for all $\smash{n\ino \bbN}$ 
\begin{equation}   
\label{beratime}
\smash{\bss_{n+1}^b = s^b_1 \big( \theta_{\bss^b_{n}} \bH \big) \; \, \textrm{if $\bss_n^n \leko \infty$} \quad  \textrm{and} \quad  \bss_{n+1}^b= \infty  \; \, \textrm{if $\bss_n^b \eqo \infty$} .}
 \end{equation} 
We also set $\smash{\bN_b \eqo \max \{ n\ino \bbN: \bss_n^b \leko \infty\}}$. Then the \emph{$b$-erasure of $\bH$} is the continuous piecewise affine process whose slopes are either equal to $1$ or to $-1$ and which takes the successive values 
$$ \smash{0\, ,\,  \bH_{\bss^b_1}\, ,\,  m_\bH (\bss^b_1, \bss^b_2) \, , \, \bH_{\bss^b_2} \, , \,   m_\bH (\bss^b_2, \bss^b_3)\, , \, \ldots \, , \,  m_\bH (\bss^b_{\bN_b-1}, \bss^b_{\bN_b})\,  ,\,  \bH_{\bzeta}. }$$  
We denote this zigzag function by $\smash{\cE_b \bH}$. Its lifetime is given by 
$$ \smash{\bM_b  := \!\!\!\!\!\!  \sum_{^{\; \, 1\leq k \leq \bN_b}} \!\!\!\!\! \Big( 2 \bH_{\bss^b_k} \! -\! m_\bH (\bss^b_{k-1}, \bss^b_k)  \! -\! m_\bH (\bss^b_k, \bss^b_{k+1}) \Big)}$$
where we have set conveniently $\smash{s^{b}_{\bN_b+1}\eqo \bzeta}$. See Figure \ref{Hunieteraz1}. 

\begin{figure}[htb]
\centering
\begin{minipage}{0.9\textwidth}
\centering 
\includegraphics[width=1\textwidth, height=0.40\textheight]{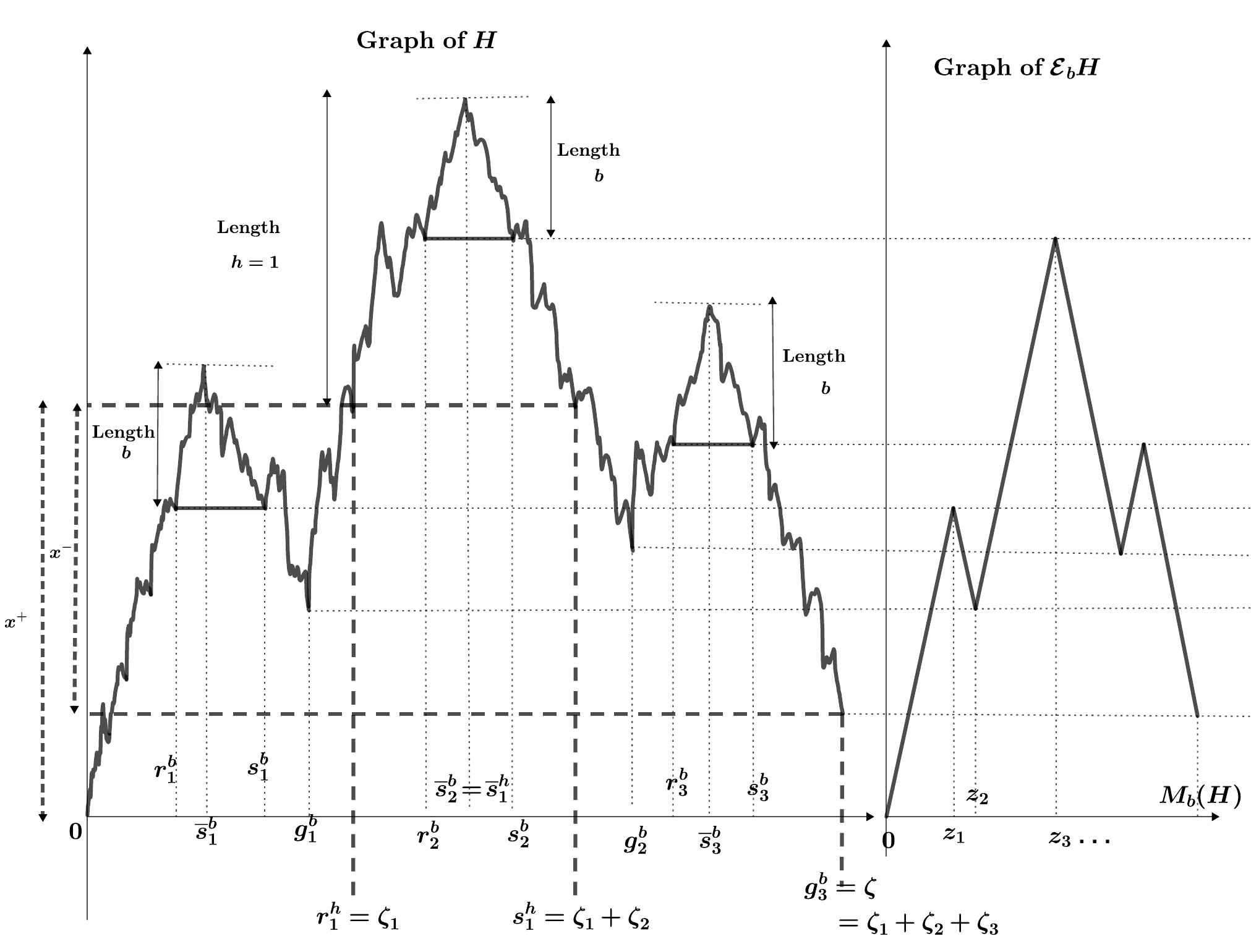}
\caption{{\footnotesize \emph{The function on the left hand side is a $h$-unicellular function $\smash{(\bH,\bzeta)}$ with initial cell 
$\smash{(x^+\!,\,  x^-)}$ (we recall that $h\eqo1$ in the introduction). The zigzag function on the right hand side of the figure is  the corresponding $b$-erased function $\smash{\mathcal E_b \bH}$ whose lifetime is $\smash{\mathbf M_b(\mathbf H)}$. Here, the scale of the horizontal axis is smaller than of the scale of the vertical axis in order to fit in the figure. }\cq}
\label{Hunieteraz1}}
\end{minipage}
\end{figure}

\smallskip

Let us denote $\smash{\mathtt{Tree} (\bH)}$ by $\smash{(\cT_\bH, d_\bH, \rho_\bH)}$. Then 
\begin{equation}
\label{eravsera}
\smash{\bcE_b \cT_\bH \eqo \bigcup_{1\leq k\leq \bN_b} \lgeo \rho_\bH, p_\bH (\bss^b_k) \rgeo  \quad \textrm{and} \quad \mathtt{Tree} (\cE_b \bH) \equiv \bcE_b \cT_\bH \; .}
\end{equation}
We next define the cell population associated with the $b$-erased function $\smash{\cE_b \bH}$ as 
\begin{eqnarray} 
\smash{\tC (\cE_b \bH) \eqo \big( (\bxx^+_k(b), \bxx^-_k(b)) \big)_{1\leq k\leq \bN_b}} \!\!\! & \textrm{where } &\!\!\! \!\!\!  \smash{\textrm{for all $1\leqo k\leqo \bN_b$ we have set}} \nonumber \\
\smash{\bxx^+_k(b)\eqo  \bH_{\bss^b_k} \! -\! m_\bH (\bss^b_{k-1}, \bss^b_k)} & \textrm{and} &  \smash{ \bxx^-_k(b)\eqo  \bH_{\bss^b_k} \! -\! m_\bH (\bss^b_k, \bss^b_{k+1})\; .} \label{bcelldef} 
\end{eqnarray}
Note that $\smash{\bM_b}$ is the total mass of $\smash{\tC(\cE_b \bH)}$. 
The following theorem asserts that the D-R branching process is actually the cell population generated by erasure of a Brownian path.
\begin{theorem}
\label{recovercell} Let $\smash{(\bcc(t))_{t\in [0, 1)}}$ be a D-R branching process with nuclei as defined above. 
Let $(\bH, \bzeta)$ be the Brownian unicellular process with initial cell $\smash{\bcc(0)\eqo (x^+,x^-)}$ which appears in (\ref{cvzigzagexpli}). Let $\smash{\mathtt{Z}_\cdot (\bcc(t))}$ be the zigzag function encoding $\bcc(t)$. Then, 
\begin{equation}
\label{recoverexpli}
\smash{ \textrm{a.s.~for all $t\ino [0, 1)$,} \quad \tZ_\cdot (\bcc(t)) = \cE_{\! \frac{_1}{^2} (1-t)} \bH \quad \textrm{and thus} \quad \bcc(t)\eqo \tC\big( 
\cE_{\! \frac{_1}{^2} (1-t)} \bH \big) \, .}
\end{equation}
In particular, $\smash{ \bmm_t \eqo \bM_{\! \frac{_1}{^2} (1-t)}}$ and $\smash{ \bnn_t \eqo \bN_{\! \frac{_1}{^2} (1-t)-}}$. 
\end{theorem} 
\noi
\textbf{Proof.} See Section \ref{LLNDRsubsec}. Actually it strongly relies on Theorem \ref{mainth3}.   \cqfd 

\smallskip

One nice aspect of the representation provided in Theorem \ref{recovercell} is that the law of the zigzag function $\smash{ \tZ_\cdot (\bcc(t))}$ can be  computed to some extent in terms of simple zigzag processes that alternatively go up and down during independent timespans distributed as exponential r.v.s. This allows in particular to prove the law of large numbers in (\ref{LLN}) as well as (\ref{limitduree}) thanks to simple probabilistic arguments.  
More precisely, let $\smash{ (\tee_n(b))_{n\in \bbN^*}}$ be a sequence of independent exponentially distributed r.v.s with mean $b$. For all $\smash{ z\ino \bbR_+}$ we set $\smash{ S_{b,z} \eqo \sum_{n\in \bbN^*} \un_{[0\, ,\,  \tee_1(b) + \ldots + \tee_n(b) ]} (z)}$, which is the associated homogeneous Poisson process with rate $1/b$. We define two random zigzag processes: the first one $\smash{ (Y_{b, z})_{z\in \bbR_+}}$ has an infinite lifetime and is given by 
\begin{equation}
\label{defYb}
\smash{ \forall z \ino \bbR_+, \quad Y_{b,z} = \int_0^z(-1)^{1+S_{b,y}}\,  \mathrm d y \, . }
\end{equation}
Theorem \ref{Brodecunicelth} in Section \ref{brodecsubsec} first shows that the $b$-erased process of the standard Brownian motion $B$ (that is defined in the same way as the $b$-erased processof $\bH$) has the same law as $\smash{ (Y_{b,z})_{z\in \bbR_+}}$. 
We define a second zigzag process $\smash{ (Z_{b,z})_{z\in \bbR_+}}$: its lifetime is $\smash{ \zeta_b } $ $\! := \! $ 
$\smash{ \inf \big\{ z\ino (0, \infty) : }$ $\smash{ Y_{b,  \tee_1(b)+ z }}$ $ \eqo $ $\smash{ -\tee_1 (b)   \big\}}$ and 
\begin{equation}
\label{defZb}
\smash{ \forall z \ino [0, \zeta_b] , \quad Z_{b,z} =Y_{b,  \tee_1(b)+ z} + \tee_1(b) \; . }
\end{equation}  
In other words $\smash{ (Z_{b,z})_{t\in \bbR_+}}$ is the \emph{first excursion of $\smash{ (Y_{b,z})_{z\in \bbR_+}}$ above its infimum}. Although it is not used in our article, $\smash{ Z_{b,\cdot}}$ has a nice interpretation in terms of trees proved in Le Gall \cite{LG89}. Namely, 
$\smash{ \mathtt{Tree} (Z_{b, \cdot})}$ is distributed as a critical binary Galton-Watson tree (i.e., individuals have 
$0$ or $2$ children with probability $1/2$) with independent and exponentially distributed edge-lengths with mean $b/2$. 

In Lemma \ref{beraazig} (and more generally in Theorem \ref{condi}) we compute the law of $\smash{ \tZ_\cdot (\bcc(t))}$ in terms of the processes $\smash{ Y_{b, \cdot}}$ and $\smash{ Z_{b, \cdot}}$ with $\smash{ b\eqo \frac{_1}{^2} (1\! -\! t)}$. To simplify, here we only state this result in the simpler case where the initial cell is $\smash{ \bcc(0)\eqo (0, 0)}$ and where $(\bH, \bzeta)$ has law $\smash{ \bN (\, \cdot \, | \, \Gamma \eqo \frac{_1}{^2})}$. In this case Theorem \ref{mainth3} asserts that 
\begin{equation}
\label{computZintro}
\smash{ \tZ_\cdot (\bcc(t))= \cE_b \bH  \overset{\textrm{(law)}}{=}Z_{b, \cdot} \; \, \textrm{under} \; \, \bP \big( \, \cdot \, \big|  \max_{s\in [0, \zeta_b]} Z_{b,s} \eqo \frac{_{_1}}{^{^2}} t \, \big) \; .}
\end{equation}
In particular, we see that the cells associated with 
$\smash{ \tZ_\cdot (\bcc(t))}$ are pairs of conditionned exponentially distributed r.v.s with mean 
$\smash{ b\eqo \frac{_1}{^2} (1\! -\! t)}$ that however keep strong independence properties. We prove indeed the following law of large numbers which implies (\ref{LLN}). 
\begin{proposition}
\label{LLNnuc}
Recall that $\smash{ \bcc(t)\eqo \big( \bcc_k(t) \big)_{1\leq k\leq \bnn_t}}$ is the 
D-R branching process with nuclei. Then  $\smash{ \frac{_1}{^2} (1-t)^2 \mathtt{n}_t \to \bzeta}$ in probability on $\smash{ \bbR_+}$ and 
 \begin{equation}
\label{LLNnucexpli}
\smash{ \frac{1}{\bnn_t} \! \sum_{^{1\leq k\leq \bnn_t}} \!\!\! \delta_{\frac{2}{1-t} \bcc_k (t)}
\xrightarrow[t\uparrow 1]{\textrm{in probability in $\mathcal M_1 (\bbR^2_+)$}} \,  \Lambda (\mathrm dy^+\mathrm dy^-)\! :=\! e^{-(y^+\! + y^-) }  \mathrm dy^+\mathrm dy^-.}
\end{equation}
where $\smash{ \mathcal M_1 (\bbR^2_+)}$ is the space of laws on $\smash{ \bbR^2_+}$ 
equipped with the topology of weak convergence. 
\end{proposition}
 \noi
 \textbf{Proof.} See Section \ref{ProofLLNnucsec}. \cqfd

\smallskip

\noi
\textbf{Organization of the paper.} In Section \ref{LinGroDivsec} we explain how to encode by zigzag functions 
cell populations that deterministically evolve through linear growth and division. In Section \ref{lentherasec}, 
we recall how continuous functions encode rooted compact real trees; we recall length erasure of such trees and also 
the corresponding erasure procedure for encoding functions; finally, we discuss how length erasure of a given 
function/tree deterministically yields a linear growth-division process of cells. 
In Section \ref{randDRsec}, we finally prove that length erasure of conditioned Brownian paths corresponds to 
D-R continuous branching process: this is Theorem \ref{mainth3}, which entails 
Theorems \ref{cvzigzag} and \ref{recovercell}.

\smallskip

\noi
\textbf{Acknowledgments.} We would like to thank Elie Aïdékon and Bernard Derrida  for fruitful discussions.
We would also like to express our gratitude to the \emph{Tianyuan Mathematics Research Center} at Kunming 
that provided us with a welcoming environment for two weeks, which allowed us to finalize this article.

\section{Cell populations, encoding functions} 
\label{Cellpopsec}

\subsection{Linear growth-division processes}
\label{LinGroDivsec}
We recall Definitions \ref{popcelldef} and \ref{GroDivdef} and 
the notations therein. Here we fix $h\ino (0, \infty)$. This section focuses on cell populations that evolve on a time intervall $[0, 2h)$ solely through \emph{linear growth} and \emph{divisions}, i.e., functions $\smash{ t\ino [0, 2h) \! \mapsto \! \bcc(t)\ino \mathtt{Cell}}$ that satisfy the following properties. 
\begin{compactenum}

\smallskip

\item[$(i)$] The function $\smash{\bcc(\cdot)}$ is left-continuous and has a right-limit at each $t\ino [0, 2h)$.

\smallskip

\item[$(ii)$] The population explodes at time $2h$: $\smash{\lim_{t\uparrow 2h} \mathtt{n} (\bcc(t))\eqo \infty}$.

\smallskip

\item[$(iii)$] Between two splitting times, each cell grows linearly at unit rate. This necessarily implies that $\smash{t\mapsto \mathtt{n} (\bcc(t))}$ is non-decreasing and that    
\begin{equation}
\label{masssizerel}
\smash{\forall t\ino [0, 2h), \quad \tmm(\bcc(t))=  \tmm(\bcc(0)) + \int_0^t \!\!  \mathtt{n} (\bcc(s)) \, \mathrm d s\; .}
\end{equation}

\noi
\item[$(iv)$] Division of cells corresponds to the functions $\smash{\mathtt{Div}_y}$, $\smash{y\ino \bbR_+}$, as introduced in Definition \ref{GroDivdef}.

\smallskip

\noi
\item[$(v)$] There is one single initial cell, i.e., $\smash{\tnn(\bcc(0))\eqo 1}$.

\smallskip

\end{compactenum}

\noi
We see that such an evolution $\smash{\bcc( \cdot)}$ is completely characterized by its initial cell $\smash{\bcc(0)}$ and its sequence of splitting times and splitting locations $\smash{(t_n,y_n)_{n\in \bbN^*}}$: here, the $t_n $ are the (possibly repeated) jump times of $\smash{t\! \mapsto \! \tnn(\bcc(t))}$. More precisely we introduce the following. 
\begin{definition} 
\label{lingrodivevodef}(\emph{Linear growth-division evolutions}) $\,$ Let $\bcc(0) \ino \bbR^2_+$ be a cell and let $S\eqo ((t_n,y_n))_{n\in \bbN^*}$ be a $[0, 2h) \times \bbR_+$-valued sequence. For all $t\ino [0, 2h)$ we set 
$\bnn_t \eqo 1+\# \{ n\ino \bbN^* : t_n \leqo t \} \ino\bbN \cup \{\infty \}$.

\smallskip

\noindent 
$(a)$ We say that $S$ is a \emph{$\smash{\bcc(0)}$-compatible division sequence} if for all $\smash{n\ino \bbN^*}$ \begin{equation}
\label{divseqexpli}
\smash{0 \leko t_{n} \leqo t_{n+1} \leko \sup_{^{p\in \bbN^*}} t_p\eqo 2h \quad \textrm{and} \quad 0\leqo y_n \leqo \tmm(\bcc(0))+\!  \int_0^{t_n} \!\!\!\!\!  \bnn_t \, \mathrm d t\; .}
\end{equation}
The first condition in (\ref{divseqexpli}) implies that $\smash{\bnn_t \leko \infty}$ for all $t\ino [0,2h)$.

\smallskip

\noindent 
$(b)$ A $\smash{\bcc(0)}$-compatible division sequence $\smash{S\eqo (t_n,y_n)_{n\in \bbN^*}}$ is said to be \emph{nonmultiple} if furthermore $\smash{t_n \leko t_{n+1}}$ for all $\smash{n\ino \bbN^*}$.  

\smallskip

\noindent 
$(c)$ We remind the notations $\smash{\mathtt{Gro}_{t}}$ and $\smash{\mathtt{Div}_y}$ from Definition \ref{GroDivdef}. 
We set $\smash{t_0\eqo 0}$ and we first define recursively $\smash{\bgam_n \ino \mathtt{Cell}}$ for all $\smash{n\ino \bbN}$ by setting $\smash{ \bgam_0\eqo \bcc(0)}$, $\smash{\bgam_1 \eqo \mathtt{Div}_{y_1} (\mathtt{Gro}_{t_1} (\bcc(0)))}$ and 
$\smash{\bgam_{\! n+1}}$ $\eqo$ $\smash{\mathtt{Div}_{y_{n+1}} (\mathtt{Gro}_{t_{n+1}-t_{n}} (\bgam_n))}$. 
Then for all $t\ino (0, 2h)$ there is a unique 
$n\ino \bbN$ such that $\smash{t\ino (t_n, t_{n+1}]}$ and we set $\smash{\bcc(t) \eqo \mathtt{Grow}_{t-t_n} (\bgam_n)}$.
We call $\smash{(\bcc(t))_{t\in [0, 2h)}}$ the \emph{linear growth-division evolution with initial cell $\smash{\bcc(0)}$ and division sequence $\smash{S\eqo (t_n,y_n)_{n\in \bbN^*}}$}. \cq  
 \end{definition}
\begin{remark}
\label{defevorems}
\noi
$(a)$ First observe that $\smash{\mathtt{Gro}_t \circ \mathtt{Gro}_{t'}\eqo  \mathtt{Gro}_{t+t'}}$, 
$\smash{t,t'\ino \bbR_+}$. Moreover 
for all $\smash{\bcc \ino \mathtt{Cell}}$ the function $\smash{t\ino \bbR_+\! \mapsto \! \mathtt{Gro}_t (\bcc)}$ 
is clearly continuous and $\smash{\mathtt{Gro}_0 (\bcc)\eqo \bcc}$. Therefore the evolution $\smash{\bcc(\cdot)}$ in 
Definition \ref{lingrodivevodef} $(c)$ is left-continuous on $(0, \infty)$ and continuous on $\smash{\bigcup_{n\in \bbN} (t_n, t_{n+1})}$. Moreover if $\smash{t_n \leko t_{n+1}}$, then $\smash{\lim_{t\downarrow t_n} \bcc(t)\eqo \bgam_n}$. 
Therefore, $\smash{\bcc(\cdot)}$ has a right limit everywhere on $\smash{\bbR_+}$. 

\smallskip

\noi
$(b)$ Let us discuss what happens when multiple divisions occur at the same time. 
Namely suppose that $\smash{m, n\ino \bbN^*}$ are such that   
$\smash{t_{n-1} \leko  t_{n} \eqo  t_{n+m-1}\leko t_{n+m}}$. If $m\geqo 2$, then $\smash{\bgam_{n+m-1} \eqo (\mathtt{Div}_{y_{n+m-1}}}$ $ \circ$ $ \ldots$ $ \circ$ $\smash{ \mathtt{Div}_{y_{n +1}}) (\bgam_n)}$. 
Here, the order of divisions does not matter since $\smash{\mathtt{Div}_y \circ \mathtt{Div}_{y'}\eqo  \mathtt{Div}_{y'} \circ  \mathtt{Div}_y }$ for all 
$\smash{y,y'\ino \bbR_+}$. Next observe that $\smash{\tnn (\mathtt{Div}_y (\bcc))\eqo 1+ \tnn(\bcc)}$ and 
$\smash{\tnn(\mathtt{Gro}_t (\bcc))\eqo \tnn(\bcc)}$. Thus for all $\smash{s\ino (t_{n-1} , t_n]}$ and all $\smash{t\ino (t_{n+m-1}, t_{n+m}]}$ we get 
$\smash{\tnn(\bgam_{n+m-1}) \eqo 
\tnn(\bcc(t) ) } $ $\eqo$  $m$  $+ $ $\smash{ \tnn(\bcc(s) )} $ $\eqo$ $m$ $+$ $\smash{\tnn (\bgam_{n-1}) }$. Since $\smash{\tnn: \mathtt{Cell} \! \to \! \bbN^* }$ is continuous, $\smash{t\mapsto \tnn(\bcc(t))}$ is left-continuous, 
nondecreasing and its jump at time $\smash{t_n}$ has size $m$. We see that 
the set $\smash{\{ t_n ; n\ino \bbN \}}$ is exactly the set of the jump-times of $\smash{t\mapsto \tnn( \bcc(t))}$ 
and the previous arguments imply that 
$\smash{\tnn(\bcc(t) ) \eqo \bnn_t} $, for all $t\ino [0, 2h)$. By (\ref{divseqexpli}), jump-times accumulate at $2h$ which also entails $\smash{\lim_{t\uparrow 2h} \tnn(\bcc(t))\eqo \infty}$.

\smallskip

\noi
$(c)$ If the division-sequence $\smash{S\eqo (t_n, y_n)_{n\in \bbN^*}}$ is nonmultiple, we simply get 
$\smash{\bcc(t)\eqo \mathtt{Gro}_t (\bcc(0)) }$
for all $\smash{t\ino [0, t_1]}$, and $\smash{\bcc(t)\eqo  \mathtt{Gro}_{t-\btt_n} \big( \mathtt{Div}_{y_n} (\bcc(t_n))\big)}$ for all $\smash{n\ino \bbN^*}$ and all $\smash{t\ino (t_n, t_{n+1}]}$. Then $\smash{t\mapsto \tnn(\bcc(t))}$ has only jumps of size $1$ at each $\smash{t_n}$. \cq 
\end{remark}

The continuous D-R branching process on $[0, 2h)$ is then defined as the linear growth-division process whose 
division-sequence is derived from a single Poisson point process as specified below. 
\begin{definition}
\label{DRnuclei} (\emph{Continuous Derrida-Retaux branching process with nuclei}) $\,$ We fix $h \ino (0, \infty)$ and 
$\smash{\bcc(0)\eqo (x^+,x^-)\ino \bbR_+^2}$. Let $\smash{\PPi_h}$ be a Poisson point process on $\smash{[0, 2h) \times \bbR_+}$ with intensity 
$\smash{2\mathrm d t \, \mathrm d y /(2h\! -\! t)^2}$. Then the D-R branching process on $[0, 2h)$ with initial cell $\bcc(0)$ is the linear growth-division process $\smash{(\bcc(t))_{t\in [0, 2h)}}$ associated with the nonmultiple 
$\bcc(0)$-compatible division-sequence $\smash{(\btt_n, \byy_n)_{n\in \bbN^*}}$ defined by 
$\smash{\{ (\btt_n, \byy_n); \; n\ino \bbN^*\}\eqo \{ (t,y)\ino \PPi_h : y\leqo \bmm_t \}}$ where $\smash{(\bmm_t)_{t\in [0, 2h)}}$ is the unique solution on $[0, 2h)$ to the equation: 
\begin{equation}
\label{mtntcompatibis}
\smash{\underline{\texttt{Eq} (\bcc(0), \PPi_h)}: \quad \bmm_t  =  \mathtt{m} (\bcc (0))+\!  \int_{0}^t \!\!  \Big( 1+ \# \big\{ (r,y)\ino \PPi_h: r\leko s \;  \textrm{and} \; y \leko \bmm_r \big\}\Big) \mathrm d s \, . }
\end{equation}
Since the division-sequence is nonmultiple, we see that 
$\smash{\bcc(t)\eqo \mathtt{Gro}_t (\bcc(0)) }$ for all $\smash{t\ino [0, \btt_1]}$ and 
$\smash{\bcc(t)\eqo \mathtt{Gro}_{t-\btt_n} \big( \mathtt{Div}_{\byy_n} (\bcc(\btt_n))\big) }$ for all  $\smash{ t\ino (\btt_n, \btt_{n+1} ]}$ and all $\smash{n\ino \bbN^*}$. See Figure \ref{PoissCoupl}. \cq 
\end{definition}

We shall need to identify the law of two linear growth-division processes with the same initial cell. One way to proceed consists in showing that their division-sequences have the same law. To this end, we use the following result. 
\begin{lemma}
\label{divseqidentif} Let $h \ino (0, \infty)$, let 
$\smash{\bcc(0)\eqo (x^+,x^-)\ino \bbR_+^2}$, let $\smash{S\eqo (\btt_n, \byy_n)_{n\in \bbN^*}}$ and $\smash{S'\eqo (\btt'_n, \byy'_n)_{n\in \bbN^*}}$ be two random $\smash{[0, 2h) \! \times \! \bbR_+}$-valued sequences. For all $t\ino [0, 2h)$ we set 
$\smash{\bnn_t\eqo 1+\# \{ n\ino \bbN^*: \btt_n \leqo t \}}$ and $\smash{\bnn'_t\eqo 1+ \# \{ n\ino \bbN^*: \btt'_n \leqo t \}}$. We make the following assumptions. 
\begin{compactenum}

\smallskip

\item[$(a)$] A.s.~$S$ and $\smash{S'}$ are nonmultiple $\bcc(0)$-compatible division-sequences as in Definition \ref{lingrodivevodef}.  

\smallskip

\item[$(b)$] $\smash{(\btt_1, \byy_1)}$ and $\smash{(\btt'_1, \byy'_1)}$ have the same law. 

\smallskip

\item[$(c)$] For all $t\ino [0, 2h)$, all $\smash{n\ino \bbN^*}$ and all bounded and measurable $\smash{F \! :\! \bbR^2_+
 \! \to \! \bbR}$, there is a bounded and measurable $\smash{G \! :\! \bbR^{2n}_+ \! \to \! \bbR}$, which depends on $t,n$ and $F$, such that 
\begin{eqnarray*}
\smash{\bE \big[ F(\btt_{n+1}} \!\! \!\! \!\! \!\! &, & \!\! \!\! \!\! \!\! \smash{ \byy_{n+1})\un_{\{ \bnn_t=n+1\}} \big| (\btt_k, \byy_k)_{1\leq k\leq n} \big] \eqo G \big( (\btt_k, \byy_k)_{1\leq k\leq n} \big)} \\
& \textrm{and} &\smash{ \bE \big[ F(\btt'_{n+1} , \byy'_{n+1})\un_{\{ \bnn'_t=n+1\}} \big|  (\btt'_k, \byy'_k)_{1\leq k\leq n} \big] \eqo G \big( (\btt'_k, \byy'_k)_{1\leq k\leq n} \big).}
\end{eqnarray*}
\end{compactenum}
Then $S$ and $\smash{S'}$ have the same law. 
\end{lemma}
\noi
\textbf{Proof.} It is sufficient to prove recursively for all $\smash{n\ino \bbN^*}$ that 
$$\smash{ \texttt{Rec}(n) \qquad   \mathbf q_n \! :=(\btt_k, \byy_k)_{1\leq k\leq n}
  \overset{\textrm{(law)}}{=} ( \btt'_k, \byy'_k)_{1\leq k\leq n}=: \! \mathbf q_n' .}$$
Since $(b)$ is $\smash{\texttt{Rec} (1)}$, we only need to prove that $(a),(b), (c), \smash{\texttt{Rec} (n)}\! \Longrightarrow \! \smash{\texttt{Rec} (n+1)}$; by $(c)$ there is a regularkernel $\smash{q\ino ([0,2h)\times \bbR_+)^{n}\! \mapsto K_q(\mathrm d s ,\mathrm d z) \ino \mathcal M_1 (\bbR^2_+)}$ such that for all bounded and measurable $\smash{\Phi \! :\!  ([0,2h) \! \times \! \bbR_+)^{n+1}\! \to \! \bbR}$, we get 
\begin{eqnarray}
\label{kercalcu}
\smash{ \bE \big[ \Phi\big(  \mathbf q_n, (\btt_{n+1}, \byy_{n+1} )\big) \un_{\{ \bnn_t=n+1\}}  \big] }\!\!\!\!  &= &\!\!\!\! \smash[t]{\bE \Big[ \int_{\bbR^2_+} \!\! \! \Phi \big( \mathbf q_n, (s,z) \big) K_{\mathbf q_n} (\mathrm d s ,\mathrm d z)  \Big] \quad \textrm{(by $(c)$)}} \nonumber \\
\!\!\!\! &= &\!\!\!\! \bE \Big[ \int_{\bbR^2_+} \!\! \! \Phi \big( \mathbf q'_n, (s,z) \big) K_{\mathbf q'_n} (\mathrm d s ,\mathrm d z)  \Big] \quad \textrm{(by $\texttt{Rec}(n) $)} \nonumber \\
\!\!\!\! &= &\!\!\!\! \smash{ \bE \big[ \Phi \big( \mathbf q'_n, (\btt'_{n+1}, \byy'_{n+1} )\big) \un_{\{ \bnn'_t=n+1\}}  \big] \quad \textrm{(by $(c)$)}. }
\end{eqnarray}
Since $S$ and $\smash{S'}$ are nonmultiple, we a.s.~get $\smash{\un_{\{ \bnn_t=n+1\}} \eqo \un_{\{ \btt_n < t \leq \btt_{n+1}\}}}$ and 
$\smash{\un_{\{ \bnn'_t=n+1\}} \eqo \un_{\{ \btt'_n < t \leq \btt'_{n+1}\}}}$. By integrating (\ref{kercalcu}) on $t$, we obtain 
$\smash{\bE [ \Phi(  \mathbf q_n, (\btt_{n+1}, \byy_{n+1} ))} $ $\smash{(\btt_{n+1} \!-\! \btt_{n} ) ] }$ $\smash{\eqo}$ $\smash{\bE [ \Phi ( \mathbf q'_n, (\btt'_{n+1}, \byy'_{n+1} ))} $ $\smash{(\btt'_{n+1} \!-\! \btt'_{n} )]}$, which easily implies $\smash{\texttt{Rec} (n+1)}$. This completes the proof of the lemma. \cqfd

\smallskip
  
The next lemma specifies the computations of conditional laws which are necessary to check Assumptions $(b)$ and $(c)$ in Lemma \ref{divseqidentif} for the division sequence of the D-R branching process as introduced in Definition \ref{DRnuclei}. 
\begin{lemma}
\label{condlawDRcomput} Let $ \smash{ h, \bcc(0), \PPi_h}$, $ \smash{(\btt_n, \byy_n)_{n\in \bbN^*}}$ and $ \smash{(\bcc(t))_{t\in [0, 2h)}}$ 
be as in Definition \ref{DRnuclei}. 
For any $t\ino [0, 2h)$ we set 
$ \smash{\bnn_t \eqo \tnn (\bcc(t))}$ and $ \smash{\bmm_t \eqo \tmm(\bcc(t))}$ and we denote by $ \smash{\mathscr F_{\! t}}$ the sigma field generated by $ \smash{\PPi_h  \cap   ([0,t] \! \times\!  \bbR_+)}$. For all $ \smash{y\ino \bbR_+}$ and 
$ \smash{n\ino \bbN^*}\! $, we define the following Borel probability measure on $ \smash{[0, 2h) \! \times \! \bbR_+}$ 
\begin{equation}
\label{condlawexpli}
 \smash{M_{y,n,t} (\mathrm d s , \mathrm dz) \eqo \frac{  2(2h\! -\! t)^{2n}}{(2h\! -\! s)^{2n+2}}  \mathrm{e}^{-2 (y + n(2h-t)) \big(  \frac{1}{2h-s} -\frac{1}{2h-t}  \big)}  \un_{[t,2h)} (s) \un_{[0, y+ n(s-t) ]} (z) \, \mathrm d s \mathrm dz }
\end{equation}
which is measurable with respect to $(y,n,t)$. Then a regular version of the conditional law of $ \smash{(\btt_{\bnn_t}, \byy_{\bnn_t})}$ given $ \smash{\mathscr F_t}$ is $ \smash{M_{\bmm_t, \bnn_t, t}}$. 
\end{lemma}   
\noi
\textbf{Proof.} We fix $ \smash{n\ino \bbN^*}$, $t\ino [0, 2h)$ and $ \smash{y\ino \bbR_+}$. For all $s\ino [t, 2h)$ we define the following subset of 
$ \smash{[t, 2h)\! \times \!  \bbR_+}$:
$$ \smash{ A^t_{s,n,y} \eqo \big\{  (r,z)\ino [t, s) \! \times \!  \bbR_+ : z\leqo y+n(r\! -\! t)\big\} \; .}$$
We then fix $ \smash{F\! : \!  [0, 2h) \times \bbR_+ \! \to \! \bbR}$, bounded and measurable. 
We recall from Definition \ref{DRnuclei} that $ \smash{\mu (\mathrm ds ,\mathrm dy)\! := \! 2\mathrm d s \, \mathrm d y /(2h\! -\! s)^2}$ is the intensity measure of $ \smash{\PPi_h}$. Since $ \smash{\PPi_h \cap ([t, 2h)\!  \times \! \bbR_+)}$ is a Poisson point process independent of $ \smash{\mathscr F_{\! t}}$, Mecke's formula implies the following. 
\begin{eqnarray}\bE \big[ F(\btt_{\bnn_t} , \byy_{\bnn_t}) \big| \mathscr F_{\! t} \big] \!\! \!\! \!\! & =&\!\! \!\! \!\!   \bE \Big[  \sum_{(s,y) \in \PPi_h} \!\!\! \!\! F(s,y) \, 
\un_{[t, 2h) } (s) \un_{[0, \bmm_t + (s-t)\bnn_t ]} (y)  \un_{\{ \# (\PPi_h \cap A^t_{s, \bnn_t, \bmm_t} ) = 0\}} \Big| \mathscr F_{\! t} \Big] \nonumber  \\
\!\! \!\! \!\! &=&\!\! \!\! \!\!  \int_t^{2h}\!\!  \frac{2\mathrm ds}{(2h \! -\! s)^2} \int_0^\infty \!\! \!\!\!  \mathrm dy \, F(s,y) \un_{[0, \bmm_t + (s-t)\bnn_t]} (y) e^{-\mu (A^t_{s, \bnn_t, \bmm_t})} . \label{Mecke1}
\end{eqnarray}
An elementary computation then implies 
$$ \smash{\mu ( A^t_{s,n,z})  \! = \! \int_t^{s}\!  \frac{2 (y+n(r\! -\! t)) }{(2h \! -\! r)^2} \mathrm dr= 2(y+n(2h \! -\! t)) \Big( \tfrac{1}{2h-s} \! -\! \tfrac{1}{2h-t}\Big) + 2n \log \tfrac{2h-s}{2h-t} }$$
which implies (\ref{condlawexpli}) by (\ref{Mecke1}). \cqfd 

\subsection{Encoding cell populations with zigzag functions}
\label{Encozigzagsec}
The goal of this section is to explain how to encode cell populations 
by broken line functions that we call \emph{zigzag functions}. 
To this end we first provide a suitable framework to work with continuous process endowed with a lifetime. We first denote by 
$\smash{\bC^0 (\bbR_+, \bbR)}$ the space of continuous functions from $\smash{\bbR_+}$ to $\bbR$ equipped with the following metric 
\begin{equation}
\label{distC0polish} 
\smash{\delta (H,H')= \sum_{p\in \bbN^*} 2^{-p} \min \big( 1\, , \, \max_{t \in [0,\,  p]} |H_t \! -\! H'_t|   \big)}
\end{equation}
which induces the topology of uniform convergence on every compact intervals. It makes of $\smash{\bC^0 (\bbR_+, \bbR)}$ a Polish space. 
\begin{definition}
\label{proclifetimedef} (\emph{Continuous functions with lifetime})$\,$ We set 
$$\smash{ \bC = \big\{ (H, \zeta) \ino  \bC^0 (\bbR_+, \bbR) \! \times \! [0, \infty] : H_0 \eqo 0 \; \, \textrm{and} \; \, \forall t\ino (\zeta, \infty) : H_t\eqo H_\zeta \big\} }$$ 
and we equip $\bC$ with the metric $d$ which is defined for all $\smash{(H,\zeta), (H'\! ,\zeta')\ino \bC}$ by 
$$ \smash{d\big((H,\zeta), (H'\! , \zeta')  \big)= \delta (H,H') + \big| \mathrm{arctan}\,  \zeta -\mathrm{arctan} \, \zeta'\big| \, ,}$$
with the convention that $\smash{\mathrm{arctan} (\infty)\eqo \pi/2}$. 
We denote $\smash{\bC_{\! f} \eqo \{ (H, \zeta) \ino \bC : \zeta \leko \infty \}}$ the subspace of finite lifetime functions and we denote by $\partial $ the unique function of $\bC$ whose lifetime is equal to $0$ ($\partial$ is therefore constant to $0$).  \cq 
\end{definition}
\begin{remark}
\label{remlifetime} We check that $(\bC, d)$ is Polish. Moreover $\smash{(H, \zeta) \ino \bC \! \mapsto \! H \ino \bC^0 (\bbR_+, \bbR)}$ and $\smash{(H, \zeta)}$ $ \ino $ $\smash{\bC}$ $ \! \mapsto \! $ $\smash{\zeta \ino [0, \infty]}$ are continuous. Within this framework if $(H, \zeta) \ino \bC$ it does not mean that 
$\zeta$ is a function of $H$ (indeed, $\smash{(H, \zeta') \ino \bC}$ for all $\smash{\zeta' \geqo \zeta}$). 
However, when there is no ambiguity, we sometimes abuse notation and we 
write $\zeta_H$ to denote 'the' lifetime of $H$.   \cq 
\end{remark}
\begin{definition}
\label{concadef}(\emph{Concatenation of functions}) Let  $\smash{(H,\zeta), (H'\! ,\zeta')\ino \bC}$. The concatenation of  $(H,\zeta)$ with $\smash{(H'\! ,\zeta')}$ is the function $\smash{(H\! \centerdot H'\! , \zeta_{H\centerdot H'}) \ino \bC}$ that is 
given by 
$\smash{\zeta_{H \centerdot H'}}$ $ \eqo $ $ \zeta$ $+$  $\smash{\zeta' }$ and $\smash{(H \! \centerdot H')_t }$ $\eqo$  $\smash{H_t}$ $+$ $\smash{ H'_{(t-\zeta)_+}}$, for all $\smash{t\ino \bbR_+ }$. 
Here, we use the convention $x+ \infty\eqo \infty$, which implies that $\smash{(H\! \centerdot H'\! , \zeta_{H \centerdot H'})\eqo (H, \infty)}$ if $\zeta\eqo \infty$. \cq
\end{definition}
\begin{remark}
\label{remconca} $(a)$ Note that $H\! \centerdot \partial\eqo \partial \centerdot H\eqo H$. 

\smallskip

\noi
$(b)$ Since concatenation is associative there is no ambiguity in the notation $\smash{(H^1 \! \centerdot \ldots \centerdot H^n, \zeta_{H^1 \centerdot \ldots \centerdot H^n})}$ for all $\smash{(H^k\! , \zeta_k)\ino \bC}$, $1\leqo k\leqo n$.

\smallskip

\noi
$(c)$ We easily check that concatenation is a continuous function from $\smash{\bC^2}$ to 
$\bC$.  \cq 
\end{remark}
\begin{definition}
\label{zigzagdef} (\emph{Zigzag functions}) $(a)$ $(Z, \zeta_Z)\ino \bC$  is a zigzag function if the following holds true. 

\begin{compactenum} 

\smallskip

\item[$-$] If $\zeta_Z \leko \infty$, there is a sequence of \emph{sign-changes} $\smash{ z_0\eqo 0\leqo z_1 \leqo \ldots \leqo z_{2n-1} \leqo z_{2n}\eqo \zeta_Z}$ such that 
\begin{equation} 
\label{zigzagexpli}
\smash{  \forall z\ino [0, z_{2n} ] , \quad Z_z = \int_0^z \!\!  \Big(\!\! \!\!  \!\!  \sum_{^{\quad 1\leq j\leq 2n}} \!\! \!\! \!\! (-1)^{j+1} \un_{(z_{j-1}, z_{j} ] } (y) \Big) \mathrm d y\,  .} 
\end{equation}
\item[$-$] If $\smash{ \zeta_Z\eqo \infty}$, then for all $\smash{ t\ino \bbR_+}$, $\smash{ (Z_{\cdot \wedge t}, t)}$ is a zigzag function with a finite lifetime as above. 
\end{compactenum}

\smallskip

\noi
$(b)$ If $\smash{ (Z,\zeta_Z)}$ is a zigzag function we denote by $\smash{ \mathscr M (Z)}$ the set of times at which $\smash{ Z_\cdot}$ reaches a strict local maximum. This set has no limit point even when $\smash{ \zeta_Z\eqo \infty}$. 
We also set $\smash{ \tnn(Z)\eqo \# \mathscr M (Z)}$, which is possibly infinite if $\smash{ \zeta_Z\eqo \infty}$. \cq
\end{definition}
Zigzag functions are continuous piecewise affine functions whose slopes are either equal to $1$ or equal to $-1$. Note that distinct sequences of sign-changes may yield the same zigzag function. 

Let $\smash{ (Z, \zeta_Z)}$ be a zigzag function with a finite lifetime 
associated with a sequence of sign-changes $\smash{ (z_j)_{1\leq j\leq 2n}}$. It starts rising (i.e., $\smash{ \dot{Z}(0+)\eqo 1}$) if and only if $\smash{ 0\leko z_1}$ and in this case $\smash{ 0\! \notin \! \mathscr M (Z)}$. It starts falling  (i.e., $\smash{ \dot{Z}(0+)\eqo -1}$) if and only if $\smash{ z_1\eqo 0\leko z_2}$  and in this case $\smash{ 0\! \in \! \mathscr M (Z)}$. Similarly, it ends falling (i.e., $\smash{ \dot{Z}(\zeta_{Z}-)\eqo -1}$)  if and only if $\smash{ z_{2n-1} \leko z_{2n}}$ and in this case $\smash{ \zeta_Z \! \notin \! \mathscr M (Z)}$ and it ends rising (i.e., $\smash{ \dot{Z}(\zeta_{Z}-)\eqo 1}$)  if and only if $\smash{ z_{2n-2} \leko z_{2n-1} \eqo z_{2n}}$ and in this case $\smash{ \zeta_Z \! \in \! \mathscr M (Z)}$. 

\begin{definition}
\label{concadef} (\emph{Zigzag functions associated with cell populations}) Let $\smash{ \bcc\eqo (c_k)_{1\leq k \leq n} \ino \mathtt{Cell}}$. We associate with $\bcc$ the zigzag function $\smash{ \mathtt{Z}_\cdot (\bcc)}$ whose sequence of sign-changes is $\smash{ \overline{\bcc}\eqo (z_{j})_{1\leq j\leq 2n}}$, the cumulated version of $\bcc$. See Figure \ref{ZigzagCellPop}.\cq 
\end{definition}
\begin{remark}
\label{remlifetime} $(a)$ Let $\smash{ x\ino \bbR}$. We define $\smash{(\phi_x, |x|)\ino \bC}$ by setting 
$\smash{\phi_x(t) \eqo \varepsilon (t \wedge |x|)}$, $\smash{t\ino \bbR_+}$, where $\varepsilon$ is the sign of 
$x$ (and if $x\eqo 0$, $(\phi_0, 0)\eqo \partial$). Let $\smash{\bcc\eqo (c_k)_{1\leq k \leq n}  \ino \mathtt{Cell}}$. 
Then we observe that the lifetime of $\smash{\mathtt{Z}_\cdot (\bcc)}$ is $\smash{\tmm (\bcc)\eqo |c_1|+ \ldots |c_n|}$, the total mass of $\bcc$, and we see that $\smash{\mathtt{Z}_\cdot (\bcc)}$ $ \eqo$ $\smash{\phi_{x^+(c_1)} }$ $\centerdot$ $ \smash{\phi_{-x^-(c_1)}}$ $\centerdot$ $\ldots$ $\centerdot$ $\smash{\phi_{x^+(c_n)} }$ $\centerdot$ $ \smash{\phi_{-x^-(c_n)}}$,  
where we recall that $\smash{c_k\eqo (x^+(c_k),x^-(c_k))}$, $1\leqo k\leqo n$.

\smallskip

\noi
$(b)$ Any zigzag function with a finite lifetime can be seen as the zigzag function of at least one cell population. However 
$\smash{\mathtt{Z}: \mathtt{Cell} \! \to \! \bC}$ is not injective: see the following lemma.  \cq 
\end{remark}
\begin{lemma}
\label{mincellzigzag} Let $\smash{(Z, \zeta_Z)}$ be a zigzag function such that $\smash{\zeta_Z\leko \infty}$. Let $\smash{\bcc\eqo (c_k)_{1\leq k \leq n}\ino \mathtt{Cell}}$ with cumulated version 
$\smash{\overline{c}\eqo (z_j)_{1\leq k\leq 2n}}$. We assume $\smash{(Z, \zeta_Z)\eqo (\mathtt{Z}_\cdot (\bcc), \tmm(\bcc))}$. Then the following is true.

\begin{compactenum}

\smallskip

\item[$(i)$] $\smash{\tnn(\bcc) \geqo \tnn (Z)}$ and there is a unique cell population $\smash{\bcc'\ino \mathtt{Cell}}$ such that $\smash{\tnn(\bcc')\eqo \tnn (Z)}$ and $\smash{(Z, \zeta_Z)}$ $\eqo$  $\smash{(\mathtt{Z}(\bcc'), \tmm(\bcc)')}$ which is called the \emph{minimal cell population of $\smash{(Z, \zeta_Z)}$}. Moreover, $\smash{\mathscr M (Z)}$ $\eqo$  $\smash{\{ z'_{2l-1}; 1\leqo l\leqo \tnn(\bcc') \}}$, where $\smash{\overline{\bcc}'\eqo (z'_{j})_{1\leq j \leq 2\tnn(\bcc')}}$ is the cumulated version of $\smash{\bcc'}$. 

\smallskip

\item[$(ii)$] $\smash{\bcc\eqo (c_k)_{1\leq k\leq n}}$ is the minimal cell population of $\smash{(Z, \zeta_Z)}$ if and only if 
$\smash{ x^-(c_1) }$ $\smash{ x^+(c_2)}$ $ \smash{x^-(c_2)}$ $ \ldots$ $ \smash{x^+(c_{n-1})}$ $\smash{x^-(c_{n-1})}$ $ \smash{x^+(c_n)}$ $ \geko 0$ 
which is equivalent to $\smash{z_1 \leko z_2 \leko  \ldots \leko z_{2n-2} \leko z_{2n-1}}$. 
\end{compactenum}
\end{lemma}
\noi
\textbf{Proof.} We leave the details to the reader. \cqfd
\begin{definition}
\label{cellfromzigzag} (\emph{Cell populations associated with  zigzag functions}) Let $\smash{(Z, \zeta_Z)}$ be a zigzag function with a finite lifetime. We denote by $\smash{\mathtt C(Z)}$ its minimal cell population as defined in Lemma \ref{mincellzigzag} $(i)$, and we simply call it the cell population of $Z$. Its total mass is $\smash{\zeta_Z}$ and its size is 
$\smash{\tnn (Z)}$, which is the number of strict local maxima of $Z$. \cq  
\end{definition}

\section{Erasure of trees and of functions}
\label{lentherasec}
A linear growth-divisions evolution yields a dynamics on the encoding zigzag functions that is  understood in a simple way in terms of \emph{erasure}, a geometric procedure on the real tree encoded by the zigzag functions. To explain this we first explain what erasure of compact real trees is. We next recall relevant properties of encoding of real tree by continuous functions and we discuss specific related results for zigzag functions and subtree decompositions. Finally we define erasure of functions and we discuss  connections between linear growth-division evolutions and erasure processes of functions. This section is completely deterministic. 
\subsection{Real trees encoded by continuous functions}
\label{codeRtreesec} 
\begin{definition}
\label{realtrdef} $(a)$ A metric space. $(T, d)$ is a \textit{real tree} if the following hold true.
\begin{compactenum} 

\smallskip

\item[$-$] For all distinct $\smash{ \sigma_1, \sigma_2 \! \in\!  T}$, there is a unique isometry 
$\smash{ f\! : \! [0,d(\sigma_1,\sigma_2)] \! \rightarrow \! T}$ such
that $\smash{ f(0)\!=\! \sigma_1}$ and $\smash{ f(d(\sigma_1,\sigma_2))\! =\! \sigma_2}$. We use the notation  
$\smash{ \lgeo \sigma_1,\sigma_2\rgeo \! :=\! f([0,d (\sigma_1,\sigma_2)])}$.

\smallskip

\item[$-$] For any continuous injective function 
$\smash{ q: [0, 1] \! \rightarrow \! T}$,  $\smash{ q([0,1]) \! = \! \lgeo q(0), q(1)\rgeo}$. 
\end{compactenum}

\smallskip

\noi
We use also the notation $\smash{ \lgeo\sigma, \sigma'\lgeo \, =\! \lgeo \sigma , \sigma' \rgeo \backslash \{ \sigma'\}}$ and similarly $\smash{ \, \rgeo \sigma, \sigma' \rgeo}$ and $\smash{ \, \rgeo \sigma, \sigma' \lgeo\, }$. 
We distinguish a point $\smash{ \rho\! \in \! T}$ that is viewed as the root and we speak of $(T, d, \rho)$ as a \textit{rooted real tree}.

\smallskip

\noi
$(b)$ Let $(T,d,\rho)$ be a compact rooted  real tree. A \emph{subtree} of $T$ is a compact connected subset $\smash{ T'\! \subset \! T}$ such that $\smash{ \rho \ino T'}$. We observe that $\smash{ (T',d,\rho)}$ is a compact rooted real tree too.

\smallskip

\noi
$(c)$ Let $(T,d,\rho)$ be a compact rooted real tree. 
We next define the \emph{set of leaves} of $(T,d,\rho)$ by 
$\smash{ \mathtt{Lf} (T)\eqo \{ \sigma \ino T\backslash \{ \rho \}  : T\backslash \{ \sigma \} \; \textrm{is connected} \}}$. 
The root is never be considered as a leaf of $T$, even if $T\backslash \{ \rho \}$ is connected (and therefore, $\smash{ \mathtt{Lf} (T)}$ depends on the choice of a root).

\smallskip

\noi
$(d)$ Let $(T,d,\rho)$ be a compact rooted real tree. It is a \emph{finite type real tree} if $\smash{ \# \mathtt{Lf} (T) \leko \infty}$. \cq 
\end{definition}

Let us introduce additional notations. Let $(T,d, \rho)$ be a compact rooted real tree. We define its \emph{total} height by 
$\smash{\bGam (T)\eqo \max_{\sigma \in T} d(\rho, \sigma) }$
which is finite since $T$ is compact. Note that $\smash{\bGam (T)}$ depends on the choice of a root. 

For all pair of points $\smash{\sigma, \sigma'\ino T}$, we also define their \emph{branch point} (or their \emph{most recent common ancestor} if $T$ is viewed as a family tree whose $\rho$ is the oldest ancestor). Namely it is the unique point of $T$, which is denoted by $\smash{\sigma \wedge \sigma'}$, such that 
\begin{equation}  
\label{brachpoint}
\smash{\lgeo \rho, \sigma \wedge \sigma' \rgeo \eqo  \lgeo \rho, \sigma \rgeo\cap \lgeo \rho, \sigma' \rgeo\; .}
\end{equation}
The branch point $\smash{\sigma \wedge \sigma'}$ is also the closest point to $\rho$ in $\smash{\lgeo \sigma, \sigma'\rgeo}$ and we get 
$ \smash{2d(\rho, \sigma \wedge \sigma') \eqo d(\rho, \sigma) } $ $+$ $\smash{ d(\rho, \sigma')}$ $\! -\! $ $\smash{d(\sigma, \sigma')}$. Note that the definition of branch points depends on the choice of a root.

\smallskip

We next discuss basic properties of the encoding of real trees by functions, as defined in the introduction. 
Let $\smash{ (H, \zeta) \ino \bC_{\! f}}$. We recall from (\ref{mHdHdef}) the definition of $\smash{m_H (\cdot, \cdot)}$ and of $\smash{d_H(\cdot, \cdot)}$. We also recall from (\ref{HRtreedef}) the definition of the real tree $\smash{(\mathcal T_{\! H}, d_H)}$ and the notation $\smash{\pcH\! :\! [0, \zeta] \! \to \! \mathcal  T_{\! H}}$ for the canonical projection. We remind that 
$\smash{\mathcal T_{\! H}}$ is rooted at the minimum of $H$: see (\ref{rootdef}). The following lemma uses the encoding function to express the fact of being a leaf or a branch point.
\begin{lemma}
\label{leafcoding} We keep the previous notations. We assume that $(H,\zeta)\! \neq \! \partial$. Then the following holds.  
\begin{compactenum}

\smallskip

\item[$(i)$]  Let $\smash{ s\ino [0, \zeta]}$. Then $\smash{ \pcH(s) \! \notin \!  \mathtt{Lf} ( \mathcal T_{\! H} ) }$ if and only if there exists $\smash{ s'\ino [0, \zeta]}$ such that $\smash{H_s} $ $\eqo$ $\smash{ m_H(s,s')\leko H_{s'}}$.

\smallskip

\item[$(ii)$]  Let $\smash{r\ino [s,s'] \! \subset \! [0, \zeta]}$. Then $\smash{\pcH(r)\eqo \pcH(s) \wedge \pcH(s')}$ if and only if $\smash{H_r \eqo m_H(s,s')}$. In this case $\smash{d_H\big( \rho_H, \pcH(s) \wedge \pcH(s')\big) \eqo m_H(s,s') -\min_{u\in [0, \zeta]} H_u}$. 
\end{compactenum}
\end{lemma}
\noi
\textbf{Proof.} We leave the details to the reader. \cqfd 
\begin{definition}
\label{Treedef} $(a)$ Let $\smash{(H, \zeta ) \ino \bC_{\! f}}$. For all $\smash{ \sigma,\sigma'\ino \mathcal T_{\! H}}$, we write 
$\smash{\sigma\! <_H \! \sigma' }$ if $\sigma$ is visited for the first time before $\smash{\sigma'}$, i.e., $\smash{\min 
\pcH^{-1} (\{ \sigma \})  \leko   \min \pcH^{-1} (\{ \sigma' \}) }$. It defines a linear order on $\smash{\mathcal T_{\! H}}$  that we call the \emph{$H$-exploration order}.

\smallskip

\noi
$(b)$ Let $\smash{(T,d,\rho, <), (T',d',\rho', <')}$ be ordered rooted compact real trees. We write $\smash{(T,d,\rho, <)}$ $ \equiv$ $ \smash{(T',d',\rho', <')}$ to mean that there is a bijective isometry $\smash{\phi\! : \! T \! \to \! T'}$ such that $\smash{\phi(\rho)\eqo \rho'}$ and such that $\smash{\sigma \! <\! \gamma}$ if and only if $\smash{\phi(\sigma) \! <'\! \phi(\gamma)}$, for all $\smash{\sigma, \gamma\ino T}$. \cq 
\end{definition}

 Let us discuss here basic properties of the real tree encoded by a zigzag function $(Z, \zeta_Z)$. We focus on functions  starting with a rise and ending with a fall. Namely we assume 
\begin{equation} 
\label{upanddown}
\dot{Z} (0+)\eqo 1 \quad \textrm{and} \quad \dot{Z} (\zeta_Z-)\eqo -1  \; .
\end{equation}
We denote by $(z_j)_{1\leq j\leq 2\tnn (Z)}$ the cumulated version of the minimal cell associated with $(Z, \zeta_Z)$, as defined in Lemma \ref{mincellzigzag} $(i)$: namely $\tnn (Z)\eqo \# \mathscr M (Z)$ where $\mathscr M (Z)\eqo \{ z_{2k-1} \, ; 1\leqo k\leqo \tnn (Z)\}$ is the set of times at which $Z$ reaches its strict local maxima. We set $z_0\eqo 0$ and we recall that $z_{2n}\eqo \zeta_Z$. We note that (\ref{upanddown}) entails $\mathscr M (Z) \! \subset \! (0, \zeta_Z)$. We use the notation $\mathtt{Tree} (Z)\eqo (\mathcal T_{\! Z}, d_Z, \rho_Z , <_Z)$ and we denote as usual by $\mathtt p_{Z}\! : \! [0, \zeta_Z] \! \to \! \mathcal T_{\! Z}$ the canonical projection. 
\begin{lemma}
\label{zigzagtree} Let $\smash{(Z, \zeta_Z)}$ be a zigzag function satisfying (\ref{upanddown}). 
We keep the previous notations and set $\smash{\sigma_0\eqo   \mathtt{p}_Z (0)}$, 
$\smash{\sigma_{\! \tnn (Z)}\eqo  \mathtt{p}_Z(\zeta_Z)}$ and 
$\smash{\{ \sigma_1 \! <_Z \! \ldots  \! <_Z  \sigma_{\! \tnn (Z)} \}} $ $\! :=\!$ $\smash{ \mathtt{Lf} (\mathcal T_{\! Z})}$
(recall that $\smash{ \tnn (Z)\eqo \# \mathtt{Lf} (\mathcal T_{\! Z})}$). 
Then the following holds true. 
\begin{compactenum}

\smallskip

\item[$(i)$] For all $\smash{ 1\leqo k \leqo \tnn (Z)}$, $\smash{ \mathtt p^{-1}_Z (\{ \sigma_k \}) \eqo \{ z_{2k-1} \}}$ 
and $\smash{ [z_{2k-1}, z_{2k+1}]\cap \, \mathtt p_Z^{-1} (\{ \sigma_k \wedge \sigma_{k+1}\} )\eqo \{ z_{2k} \}}$. 
Moreover $\smash{ \mathtt p_Z^{-1} (\{ \sigma_0 \}) \cap [0, z_1]\eqo \{ 0\}}$.

\smallskip

\item[$(ii)$]  Let $\smash{ (Z'\! , \zeta_{Z'})}$ be a zigzag function such that $\smash{ Z'_{\zeta_{Z'}}\! \! \eqo 
Z_{\zeta_Z}}$ and $\smash{ \mathtt{Tree} (Z) \! \equiv \! \mathtt{Tree} (Z')}$. Then, $\smash{ (Z'\!,\zeta_{Z'})} $ $\eqo$ $\smash{  (Z, \zeta_{Z})}$.  
\end{compactenum}
\end{lemma}
\noi
\textbf{Proof.} $(i)$ is an elementary consequence of Lemma \ref{leafcoding}: we leave the details to the reader. 
To prove $(ii)$, we first show that $\smash{ (Z'\!,\zeta_{Z'})}$ starts with a rise 
and that $Z$ and $\smash{ Z'}$ have the same minimal value. By definition of the linear order $\smash{ <_Z}$, 
$\smash{ \sigma_0\eqo \mathtt p_Z(0)}$ is the $\smash{ <_Z}$-least point of $\smash{ \mathcal T_{\! Z}}$ 
and by (\ref{upanddown}) it cannot belong to $\smash{ \mathtt{Lf} (\mathcal T_{\! Z})}$. Thus, 
the $\smash{ <_{Z'}}$-least point of $\smash{ \mathcal T_{\! Z'}}$, which is $\smash{ \sigma'_{\! \bullet}\! :=\! \mathtt p_{Z'}(0)}$ cannot belong to $\smash{ \mathtt{Lf}(\mathcal T_{\! Z'})}$, which implies $\smash{ \dot{Z}'(0+)\eqo 1}$. 
To simplify, we set $\smash{ \underline{m}\eqo \min_{z\in [0, \zeta_Z]} Z_z}$ and $\smash{ \underline{m}'\eqo \min_{z\in [0, \zeta_{Z'}]} Z'_s}$. Then, note that 
$\smash{ -\underline{m}\eqo d_Z(\rho_Z, \sigma_0)\eqo d_{Z'}(\rho_{Z'}, \sigma'_\bullet)\eqo -\underline{m}'}$. 

To prove next that $\smash{ \dot{Z}'(\zeta_{Z'}-)\eqo -1}$, we denote by $\smash{ \sigma'_{\! *}}$ the 
$\smash{ <_{Z'}}$-maximal point of $\smash{ \mathcal T_{\! Z'}}$. By (\ref{upanddown}), the leaf 
$\smash{ \sigma_{\! \tnn (Z)}}$ of $\smash{ \cT_{\! Z}}$ is the $\smash{ <_{Z}}$-maximal point of 
$\smash{ \mathcal T_{\! Z}}$ 
and by $(i)$ we get $\smash{ d_Z (\rho_Z, \sigma_{\! \tnn (Z)})}$ $ \geko$ $\smash{  d_Z (\rho_Z, \mathtt p_Z(\zeta_Z))}$ $
\eqo$ $\smash{  Z_{\zeta_{Z}}\! -\!  \underline{m}}$. Since $\smash{ \underline{m}\eqo \underline{m}'}$ and 
since $\smash{ Z_{\zeta_{Z}}\eqo Z'_{\zeta_{Z'}}}$, we therefore get $\smash{ d_{Z'} (\rho_{Z'}, \sigma'_{\! *})} $ 
$ \geko$ $\smash{  d_{Z'} (\rho_{Z'}, p_{Z'}(\zeta_{Z'})) }$ $\eqo$ $\smash{  Z'_{\zeta_{Z'}}\! -\!  \underline{m}'}$. 
But if $\smash{ \dot{Z}'(\zeta_{Z'}-)\eqo 1}$, then $\smash{ \mathtt{p}_{Z'} (\zeta_{Z'})\eqo \sigma'_*}$, which contradicts the previous strict inequality. Therefore,  
$\smash{ \dot{Z}'(\zeta_{Z'}-)\eqo -1}$ and $\smash{ (Z'\! , \zeta_{Z'})}$ satisfies (\ref{upanddown}). 

We set $\smash{ \mathtt{Lf} (\mathcal T_{\! Z'}) \eqo \{ \sigma'_1\! <_{Z'}\!  \ldots <_{Z'} \! \sigma'_{\! \mathtt n(Z')}\}}$. 
We first note that $\smash{ \tnn (Z) \eqo \tnn(Z')}$ and 
we apply $(i)$ to $Z$ and $\smash{ Z'}$ to see that 
$Z$ (resp.~$\smash{ Z'}$) is the unique zigzag function taking successively the values $0$, $\smash{ d_Z (\rho_Z, \sigma_1)+ \underline{m}\, }$,  $\smash{ d_Z (\rho_Z, \sigma_1\wedge \sigma_2) + \underline{m}\, }$, $\, \ldots\, $, $\smash{ d_Z (\rho_Z, \sigma_{\tnn (Z)-1}\wedge \sigma_{\tnn (Z)}) + \underline{m}\, }$,  
$\smash{ d_Z (\rho_Z, \sigma_{\tnn (Z)}) + \underline{m}\, }$ and $\smash{ Z_{\zeta_Z}}$ 
(resp.~the values $0$, $\smash{ d_{Z'} (\rho_{Z'}, \sigma'_1)+ 
\underline{m}'}$,  $\smash{ d_{Z'} (\rho_{Z'}, \sigma'_1\wedge' \sigma'_2) +
 \underline{m}'}$, $\, \ldots\, $, $\smash{ d_{Z'} (\rho_{Z'}, \sigma'_{\! \tnn(Z')-1} } $ $\!\wedge' \! $ $\smash{  \sigma'_{\! \tnn(Z')}) + \underline{m}'}$,  
$\smash{ d_{Z'} (\rho_{Z'}, \sigma'_{\! \tnn(Z')}) + \underline{m}'}$ and $\smash{ Z'_{\zeta_{Z'}}}$). This implies 
$\smash{ (Z'\!,\zeta_{Z'})\eqo (Z, \zeta_{Z})}$, which completes the proof of the lemma. 
\cqfd

\begin{remark}
\label{manycodings} Lemma \ref{zigzagtree} $(ii)$ is not trivial because except for the point tree, each ordered rooted compact real is coded by uncountably many  continuous functions. Namely, let $\smash{ (H,\zeta), (H'\! , \zeta') \ino \bC_{\! f}}$. Let us assume that there is a nondecreasing continuous and surjective function $\smash{ J\! :\! [0, \zeta]\! \to \! [0, \zeta']}$ such that 
$\smash{ H'_{\! J(t)}\eqo H_t}$, $\smash{ t\ino [0, \zeta]}$. Then $\smash{ \mathtt{Tree} (H)\equiv \mathtt{Tree}(H')}$.  \cq 
\end{remark}

\subsection{Subtree decompositions}
\label{subtreesec}
In this section we provide a decomposition of the real tree encoded by a continuous function along the subtree spanned by a finite set of times (i.e., a finite dimensional \emph{marginal tree}) by providing a local time and a reflected process 
on this subtree. More precisely, we introduce the following. 
\begin{definition}
\label{margtreedef} 
Let $\smash{ (H, \zeta)\ino \bC_{\! f}}$ and let $\smash{\bss \eqo (s_k)_{1\leq k\leq n}}$ be real numbers such that $\smash{0\leqo s_1 \! \leq \ldots \leq \! s_n \leqo \zeta}$. We set $\smash{s_0\eqo 0}$ and $\smash{s_{n+1} \eqo \zeta}$ and we suppose that 
\begin{equation}
\label{Hascend}
\smash{m_H(0,s_1) \eqo 0 \quad \textrm{and} \quad 
m_H(s_n, \zeta)\eqo H_\zeta \; .}
\end{equation}
We use the notation $\smash{\mathtt{Tree} (H)\eqo (\mathcal T_{\! H}, d_H, \rho_H, <_H)}$ and we denote the canonical projection by $\smash{\pcH}$, as usual. We introduce the following.

\smallskip

\noi
$(a)$ The \emph{$\bss$-marginal subtree of $\smash{\mathtt{Tree} (H)}$} is the finite type subtree of $\smash{\mathcal T_{\! H}}$ given by 
\begin{equation}
\label{margtreeexpli}
\smash{T=\bigcup_{^{1\leq k\leq n}} \lgeo \rho_H, \pcH(s_k)\rgeo \; .}
\end{equation}

\noi
$(b)$ The cell population $\smash{ \bcc \eqo ((x^+_k, x^-_k))_{1\leq k\leq n }}$ 
associated with $\smash{(H, \zeta, \bss)}$ is given by 
\begin{equation}
\label{celpopmargtree}
\smash{\forall k\ino \{ 1, \ldots, n\}, \qquad x^+_k \eqo H_{s_k}\! -\! m_H(s_{k-1}, s_k) \quad \textrm{and} \quad x^-_k \eqo H_{s_k}\! -\! m_H(s_k,s_{k+1} )  .}
\end{equation}

\noi
$(c)$ For all $\smash{\sigma \ino \mathcal T_{\! H}}$, we define 
$\smash{\mathtt{proj}_T(\sigma) \ino T}$ as the closest point of $T$ to $\sigma$, i.e.,  
$\smash{\lgeo \rho_H,  \mathtt{proj}_T (\sigma) \rgeo} $ $\eqo$ $\smash{ \lgeo \rho_H, \sigma \rgeo \cap T}$. We call $\smash{\mathtt{proj}_T\! : \!  \mathcal T_{\! H}\! \to \! T}$ the \emph{projection on} $T$, which is $1$-Lipschitz.  \cq 

\end{definition}
\begin{remark}
\label{subtreerem} $(a)$ By (\ref{Hascend}) we get
$\smash{ \pcH(0) \ino \lgeo \rho_H, \pcH(s_1)\rgeo}$ and $\smash{ \pcH(\zeta) \ino \lgeo \rho_H, \pcH(s_n)\rgeo}$.

\smallskip

\noi
$(b)$ Let us denote by $\smash{ (Z, \zeta_Z)\eqo (\mathtt{Z}_\cdot (\bcc) , \tmm(\bcc))}$ the zigzag function encoding the cell population $\bcc$ associated with $\smash{ (H, \zeta, \bss)}$ as defined above (see also Definition \ref{concadef} for 
$\smash{ \mathtt{Z}_\cdot (\bcc)}$).
We get $\smash{ Z_{\zeta_{Z}}\eqo H_\zeta}$ and $\smash{ \underline{Z}_{\zeta_Z}\eqo \underline{H}_\zeta}$, and we also check that $\smash{ \mathtt{Tree} (Z)\equiv \big( T, d_H, \rho_H, <_H)}$. \cq 
\end{remark}

The following lemma provides a decomposition of the function $(H,\zeta)$ along the subtree $T$ via a local time at $T$ and a reflected process on $T$. 
\begin{lemma}
\label{loctimlem} Let $\smash{ (H, \zeta)\ino \bC_{\! f}}$ and let $\smash{ \bss \eqo (s_k)_{1\leq k\leq n}}$ be real numbers such that $\smash{ 0\leqo s_1 \! \leq \ldots \leq \! s_n \leqo \zeta}$. We set $\smash{ s_0\eqo 0}$ and $\smash{ s_{n+1} \eqo \zeta}$ and we assume (\ref{Hascend}). We denote by $\smash{ (Z, \zeta_Z)}$ the zigzag function of the cell population associated with $\smash{ (H, \zeta, \bss)}$ as in Definition \ref{margtreedef} $(b)$. We denote the tree coded by $(H, \zeta)$ by  
$\smash{ \mathtt{Tree} (H)\eqo (\mathcal T_{\! H}, d_H, \rho_H, <_H)}$ and we recal that $\smash{ \pcH \! :\!  [0, \zeta] \! \to \! \mathcal T_{\! H}}$ is the canonical projection. We denote by $T$ the $\bss$-marginal subtree of $\smash{ \mathcal T_{\! H}}$ as in (\ref{margtreeexpli}). Then the following holds true. 

\begin{compactenum}

\smallskip

\item[$(i)$] For all $\smash{ t\ino [0, \zeta]}$, we set $\smash{\mathcal H_t \eqo d_H \big( \pcH(t) , \mathtt{proj}_T(\pcH(t)) \big)}$. Then $\smash{ (\mathcal H, \zeta) \ino \bC_{\! f}}$ and $\smash{\mathcal H_0\eqo \mathcal H_\zeta\eqo 0}$. We call 
$\smash{(\mathcal H, \zeta)}$ the \emph{reflected process on $T$.} 
 
\smallskip

\item[$(ii)$] There is a \emph{unique} continuous nondecreasing and surjective function $\smash{ J\! :\! [0, \zeta] \! \to \! [0, \zeta_Z]}$ such that $\smash{ d_H (\rho_H, \mathtt{proj}_T (\pcH(t)))\eqo Z_{\! J_t}\! -\! \underline{Z}_{\zeta_Z}}$. We call $J$ the \emph{local time of $H$ at $T$}. We thus get $\smash{H_t\eqo \mathcal H_t + Z_{\! J_t}}$, for all $\smash{t\ino [0, \zeta]}$. 

\smallskip

\item[$(iii)$] For all $\smash{ k\ino \{ 0, \ldots, n\}}$ and all $\smash{ t\ino [s_k, s_{k+1}]}$, we get $\smash{Z_{J_t}\eqo m_H(t,s_k) \vee m_H(t,s_{k+1}) }$ and 
\end{compactenum}

\vspace{-5mm}

\begin{equation}
\label{decompexpli}
\smash{J_t =   H_{s_k} \! -\! m_H(s_k,s_{k+1}) + m_H(t, s_{k+1}) \! -\! m_H(t,s_k)+ \!\! \! \sum_{^{0\leq j\leq k}} \!\! d_H(\pcH(s_{j-1}), \pcH( s_j)). }
\end{equation}
\end{lemma}
\begin{remark} 
\label{connTT}
\emph{(Connected components of $\smash{ \mathcal T_{\! H} \backslash T}$)} We keep the same notations as in Lemma \ref{loctimlem} and we denote by $\smash{ (\alpha_j, \beta_j)}$, $\smash{ j\ino \mathcal J}$, the connected components of 
$\smash{ \{ t\ino [0, \zeta] : \mathcal H_t \geko 0\}}$. 
We then set $\smash{ \mathscr T^o_j\eqo \pcH((\alpha_j, \beta_j)) \! \subset \! \cT_{\! H}}$. Then, 
the $\smash{ \mathscr T^o_j}$, $\smash{ j\ino \mathcal J}$, are the connected components 
of $\smash{ \cT_{\! H}\backslash T}$. Moreover $\smash{ \gamma_j\! :=\! \pcH(\alpha_j) \eqo \pcH (\beta_j)}$ is such that 
$\smash{ \{ \gamma_j\} \cup \mathscr T^o_j}$ is the closure of $\smash{ \mathscr T^o_j}$. \cq  
\end{remark}
\noi
\textbf{Proof of Lemma \ref{loctimlem}.} Since $\smash{ \mathtt{proj}_T}$ and $\smash{ \pcH}$ are continuous, so is 
$\smash{\mathcal H}$. 
By Remark \ref{subtreerem} $(a)$, $\smash{\pcH(0)}$ and $\smash{\pcH(\zeta)}$ belong to $T$. Thus, $\smash{\mathcal H_0\eqo \cH_\zeta\eqo 0}$ and it proves $(i)$.  
We next prove $(ii)$ and $(iii)$ in the same time. Let $\smash{ \bcc}$ be the cell population associated with $\smash{(H, \zeta, \bss)}$ as in Definition \ref{margtreedef} $(b)$. We denote by 
$\smash{\overline{\bcc}\eqo (z_j)_{1\leqo j\leq 2n}}$ the cumulated version of $\smash{\bcc}$. We set $\smash{z_0\eqo 0}$ and we recall that $\smash{z_{2n}\eqo \tmm (\bcc)\eqo \zeta_Z}$.  
Let $\smash{0\leqo k\leqo n}$ and $\smash{t_k\ino [s_k, s_{k+1}]}$ be such that $\smash{H_{t_k}\eqo m_H(s_k,s_{k+1})}$. By definition of $\smash{Z_\cdot \eqo \mathtt{Z}_\cdot (\bcc)}$, we get 
\begin{equation}
\label{expliZ}
\smash{H_{s_k}\eqo Z_{z_{2k-1}}\, , \; 1\leqo k\leqo n \quad \textrm{and} \quad m_H(s_k, s_{k+1}) \eqo H_{t_k}\eqo Z_{z_{2k}}\, , \; 0\leqo k\leqo n\; .}
\end{equation} 
Next, for all $\smash{t\ino [0, \zeta]}$ we observe that $\smash{d_H \big(\pcH(t), \lgeo \rho_H, \pcH(s_k) \rgeo \big) } $ $\eqo$ $\smash{ d_H(\pcH (t), \pcH (t) \!  \wedge \! \pcH(s_k))} $ $\eqo$ $\smash{ H_t \! -\! m_H(t,s_k)}$. 
Then, by definition of $T$, 
$\smash{d_H \big( \pcH(t), T\big)\eqo \min_{1\leq k\leq n}   d_H  \big( \pcH(t),  \lgeo \rho_H, \pcH(s_k) \rgeo \big)}$. 
By (\ref{Hascend}) and by Remark \ref{subtreerem} $(a)$ we furthermore get 
$$\smash{\mathcal H_t\eqo  d_H \big( \pcH(t) , \mathtt{proj}_T(\pcH(t)) \big) \eqo d_H \big( \pcH(t), T\big)\eqo  \min_{0\leq k\leq n+1}   d_H  \big( \pcH(t),  \lgeo \rho_H, \pcH(s_k) \rgeo \big).}$$
Therefore
$\smash{d_H \big( \pcH(t), \mathtt{proj}_T(\pcH(t) )\big) \eqo H_t \! -\! H'_t}$ where we have set $\smash{H'_t 
\eqo \max_{0\leq k\leq n+1} m_H(t,s_k)}$ for all $\smash{t\ino [0, \zeta]}$. 
Thus, $\smash{d_H \big(\rho_H ,  \mathtt{proj}_T(\pcH(t) )\big)\eqo H'_t \! -\!   \underline{H}_\zeta}$. 
For $0\leqo k\leqo n$ and $t\ino [s_k, s_{k+1}]$, we next observe that $\smash{ H'_t\eqo m_H(t,s_k) \vee m_H(t, s_{k+1})}$ and the following holds. 

\noi
$1^o)$ $\smash{ H'}$ is nonincreasing on $\smash{ [s_k, t_k]}$ from the value $\smash{ H'_{s_k}\eqo H_{s_k}\eqo Z_{z_{2k-1}}}$ to the value $\smash{ H'_{t_k}\eqo H_{t_k}\eqo Z_{z_{2k}}}$. Assume that $J$ is well-defined at time $\smash{ s_k}$ and such that $\smash{ J_{s_k}\eqo z_{2k-1}}$. We define $J$ for all $\smash{ t\ino [s_k, t_k]}$ by setting 
$\smash{ J_t\eqo J_{s_k} + H_{s_k} \! -\! m_H (t,s_k)} $. Then we get $\smash{ H'_t\eqo Z_{\! J_t}}$ and $\smash{ J_{t_k}\eqo z_{2k}}$.

\noi
$2^o)$ $\smash{ H'}$ is nondecreasing on $\smash{ [t_k, s_{k+1}]}$ from the value $\smash{ H'_{t_k}\eqo H_{t_k}\eqo Z_{z_{2k-1}}}$ to the value $\smash{ H'_{s_{k+1}} } $ $\eqo$ $\smash{  H_{s_{k+1}}} $ $\eqo$ $\smash{  Z_{z_{2k+1}}}$. Assume that $J$ is well-defined at time $\smash{ t_k}$ and such that $\smash{ J_{t_k}\eqo z_{2k}}$. We define $J$ for all $\smash{ t\ino [t_k, s_{k+1}]}$ by setting 
$\smash{ J_t\eqo J_{t_k} + m_H (t,s_{k+1})\! -\! m_H(s_k, s_{k+1}) }$. Thus, $\smash{ H'_t\eqo Z_{\! J_t}}$ and $\smash{ J_{s_{k+1}}\eqo z_{2k+1}}$. 

\noi
This provides a recursive construction of $J$ which satisfies $(ii)$ and $(iii)$. \cqfd

\subsection{Erasure of functions}
\label{erafunsec}
In this section we introduce erasure of functions, which corresponds to erasure of the encoded real trees. 
A key point point is to define $b$-unicellular functions and to provide for any path a unique decomposition into $b$-unicellular functions.

\smallskip 

\noi
\textbf{Decomposition into unicellular functions.}  For all  $\smash{ (H, \zeta) \ino \bC}$ we recall from (\ref{minreflecHdef}) the definition of 
$\smash{ \underline{H}}$ and of $\smash{ H^+\! \eqo H \! -\! \underline{H}}$, and we denote by $\smash{ (\overline{H}, \zeta)\ino \mathbf C}$ the maximal function: namely, for all $\smash{ t\ino \bbR_+}$ we  set $\smash{ \, \overline{\! H}_t\eqo \max_{s\in [0, t]} H_s}$. We easily check that 
$\smash{ (H, \zeta) \! \mapsto \! (\underline{H}, \zeta)}$, $\smash{ (H, \zeta) \! \mapsto\!  (H^+, \zeta)}$ and $\smash{ (H, \zeta) \! \mapsto \!  \, (\overline{\! H}, \zeta)}$ are continuous functions from $\bC$ to $\bC$. It also implies that 
\begin{equation}
\label{contimH}
\smash{ \big(s,s', (H, \zeta) \big) \ino \bbR_+^2 \! \times \! \bC \mapsto m_H(s,s') \ino \bbR \; \, \textrm{and} \; \,  (H, \zeta) \ino  \bC_{\! f} \mapsto \, \overline{\! H}_{\! \zeta} \ino \bbR  \; \,  \textrm{are continuous.} }
\end{equation}
Let $\smash{ t\ino \bbR_+}$. We next recall from (\ref{shiftHdef}) the definition of $\smash{ \theta_t H}$, whose lifetime is $\smash{ (\zeta \! -\! t)_+}$. We also define the function stopped at $t$, $\smash{ H_{\cdot \wedge t}}$ and the function reversed at $t$ $\smash{ \overleftarrow{H}^t\! :=\! H_{(t- \, \cdot)_+}}$, both with lifetime $t$. If $\smash{ \zeta\leko \infty}$, we simply denote $\smash{ \overleftarrow{H}^\zeta}$ by $\smash{ \overleftarrow{H}}$. 
We easily check that $\smash{ (t, (H, \zeta)) \! \mapsto \! (\theta_t H, (\zeta \! -\! t)_+)}$, 
$\smash{  (t, (H, \zeta)) \! \mapsto \! ((H_{\cdot \wedge t}, t)}$ and $ \smash{ (t, (H, \zeta)) \! \mapsto \! ( \overleftarrow{H}^t,t)}$ are continuous functions and we also check that $\smash{ (H, \zeta)\! \mapsto \! (\overleftarrow{H}, \zeta)}$ is continuous on $\bC_{\! f}$. Then we introduce the following notation for hitting times: for all $\smash{ x\ino \bbR}$ we set 
\begin{equation}
\label{hittim}
\smash{ \varrho_{x} (H) \eqo \inf \big\{ t\ino (0, \infty): H_s \eqo x \big\}\; , }
\end{equation}
with the convention $\inf \emptyset \eqo \infty$. Since $\smash{ H_0\eqo 0}$, $\smash{ \varrho_0 (H)}$ is a special case. 
\begin{lemma}
\label{regprophit} The following assertions holds true. 
\begin{compactenum}

\smallskip

\item[$\!\!\! (i)$] For all $\smash{ x\ino \bbR}$, $\smash{ (H, \zeta)\ino \bC \! \mapsto \! \varrho_x(H)\ino [0, \infty]}$ is measurable.  

\smallskip

\item[$\!\!\! (ii)$] $\! \smash{ (y, (H',\zeta'))\! \ino \bbR\! \times \! \bC \! \mapsto \!  \varrho_y (H')\ino [0, \infty]}$ is continuous at $\smash{ (x, (H, \zeta))\ino (-\infty, 0] \! \times \! \bC}$ if
\begin{equation}
\label{hittcontcrit}
\smash{ \varrho_x(H) \leko \infty \quad \textrm{and} \quad \forall \varepsilon \ino (0, \infty), \quad H_{\varrho_x(H)} \eqo x >
m_H(\varrho_{x} (H), \varrho_x(H)+ \varepsilon) }
\end{equation}
\end{compactenum}
\end{lemma}
\noi
\textbf{Proof.} It is left to the reader. \cqfd

\smallskip

Let $b\ino (0, \infty)$. We now introduce functions that encode real trees whose $b$-erasure reduces to a single branch, i.e., whose associated cell population has a single cell. 
\begin{definition}
\label{elemfundef} (\emph{Descents, rises, excursions and unicellular functions}) 
Let $\smash{ x, b\ino (0, \infty)}$, let $\smash{ x^+}$, $\smash{x^-}$ $\ino$ $\smash{\bbR_+}$ and let $\smash{ (H, \zeta)\ino \bC_{\! f}}$.

\smallskip

\noi
$(a)$  $\smash{ (H, \zeta)}$ is a \emph{descent by $x$ with amplitude less than $b$} if 
$\smash{H_\zeta\eqo \underline{H}_\zeta\eqo -x}$ and $ \smash{\max_{t\in [0, \zeta] }H^+_t \leko  b}$. We denote by $\smash{\bC^{_\downarrow}_{^{x,b}}}$ the space of such functions. By convention we set $\smash{\bC^{_\downarrow}_{^{0,b}}\eqo \{ \partial \}}$, the null lifetime function. We also set $\smash{\bC^{_\downarrow}_{^{b}} \eqo \bigcup_{y\in \bbR_+} \bC^{_\downarrow}_{^{y,b}}}$, which is the space of \emph{descents with amplitude less than $b$.} 

\smallskip

\noi
$(b)$ $\smash{(H, \zeta)}$ is a \emph{rise by $x$ with amplitude less than $b$} if for all $\smash{ t\ino (0, \zeta]}$, $\smash{0 \leko H_t \leko m_H(t,\zeta) + b }$ and $\smash{H_\zeta \eqo x }$.
We denote by $\smash{\bC^{_\uparrow}_{^{x,b}}}$ the space of such functions. By convention we set  $\smash{\bC^{_\uparrow}_{^{0,b}}\eqo \{ \partial \}}$ and we also set $\smash{\bC^{_\uparrow}_{^{b}} \eqo \bigcup_{y\in \bbR_+} \bC^{_\uparrow}_{^{y,b}}}$, which is the space of \emph{rises with amplitude less than $b$.}

\smallskip

\noi
$(c)$ $\smash{(H,\zeta)}$ is an \emph{excursion} if $\smash{H_t \geko H_0\eqo H_\zeta\eqo 0 }$ for all $\smash{t\ino (0, \zeta)}$. We denote by $\smash{\mathtt{Exc}}$ the space of excursions. We then recall the notation $\smash{\Gamma (H) \eqo \, \overline{\! H}_{\! \zeta}}$ for the height of $\smash{(H, \zeta)}$ and we introduce the space $\smash{\mathtt{Exc}(b)\eqo \big\{ (H,\zeta) \ino \mathtt{Exc} : \Gamma (H)\eqo b \big\}}$ of excursions whose height is equal to $b$ (and by convention 
$\smash{\mathtt{Exc}(0)}$ reduces to the null lifetime function).  

\smallskip

\noi
$(d)$ $\smash{(H,\zeta)}$ is \emph{$b$-unicellular with initial cell $\smash{(x^+,x^-)}$} if $\smash{H\eqo H^{{1}}  \! \centerdot   H^{{2}}\!  \centerdot H^{{3}}}$ where $\smash{(H^{{1}}, \zeta_1)  \ino \bC^{_\uparrow}_{^{\!x^+\!\!, \,  b}}}$, $\smash{(H^{{2}}, \zeta_2)  \ino \mathtt{Exc} (b)}$ and $\smash{(H^{{3}} , \zeta_3) \ino \bC^{_\downarrow}_{^{\! x^-\!\! , \, b}}}$. 
We denote by $\smash{\bC_{x^+\!\! ,\, x^-\!\! ,\,  b}}$ the space of $b$-unicellular functions with initial cell 
$\smash{(x^+\! ,\, x^-)}$.  
See Figure \ref{Hunieteraz1}.  \cq 
\end{definition}
\begin{remark}
\label{elemfunrem} $(a)$ Note that $\smash{\bC_{0,0, b}\eqo \mathtt{Exc} (b)}$ and that $\smash{\bC_{x,x,b} \! \subset \! \mathtt{Exc} (b+x)}$, the inclusion being strict as soon as $\smash{x\geko 0}$. 

\smallskip

\noi
$(b)$ If $\smash{(H, \zeta) \ino \bC^{_{\uparrow}}_{^{x,b}}}$, then $\smash{(\overleftarrow{H}, \zeta) \ino  \bC^{_{\downarrow}}_{^{x,b}}}$. If $\smash{(H, \zeta) \ino \bC^{_{\downarrow}}_{^{x,b}}}$ and if $\smash{\varrho_{-x}(H)\eqo \zeta}$, then 
$\smash{(\overleftarrow{H}, \zeta) \ino \bC^{_{\uparrow}}_{^{x,b}}}$.

\smallskip

\noi
$(c)$ Let $\smash{(H,\zeta)\ino \bC_{x^+\!, x^-\! , b}}$ for some $b\ino (0, \infty)$ and for some $\smash{x^+\!, x^-\ino \bbR_+}$. Then 
\begin{equation}
\label{ptcenter}
\smash{m_H \big( 0,s^*(H)\big)\eqo 0 \; \textrm{and} \; m_H \big( s^*(H), \zeta \big)\eqo H_\zeta\, , \; \textrm{where} \;  s^*(H) \eqo \inf \big\{ s\ino [0, \zeta] : H_s\eqo \, \overline{\! H}_{\! \zeta}  \big\}. }
\end{equation}
\emph{Indeed}, let $\smash{(H^i, \zeta_i)}$, $i\ino \{ 1,2,3\}$, be as in Definition \ref{elemfundef} $(d)$.  Then observe that $\smash{s^*(H)  \ino [\zeta_1, \zeta_1+ \zeta_2]}$ and $\smash{s^*(H) \! -\! \zeta_1}$ is the time at which $\smash{H^2}$ reaches its maximal value. Then (\ref{ptcenter}) is an immediate consequence of the definition $\smash{H\eqo H^1\centerdot H^2\centerdot H^3 }$. \cq 
\end{remark}

We now explain how to derive from a $b$-unicellular function $(H, \zeta)$ the unique triplet of functions $\smash{(H^i\! , \zeta_i)}$, $i\ino \{ 1,2,3\}$ as in Definition \ref{elemfundef} $(d)$. To this end we recall from (\ref{1erbtps}) the definition of the first $b$-erasure time $\smash{s^b_1(H)}$. We then set 
$\smash{r^b_1(H)\eqo 0}$ if $\smash{s^b_1(H)\eqo \infty }$ and otherwise 
\begin{equation}
\label{r1bHdef}
\smash{r^b_1(H)\eqo \sup \big\{ t\ino [0, s^b_1(H)) : H^+_t \eqo H^+_{\! s^b_1(H)}\big\}\; .}
\end{equation} 
In other words $\smash{r^b_1(H)}$ is the \emph{beginning of the first complete excursion of $H$ of height $b$}. 
When there is no ambiguity, we write $\smash{s^b_1}$ and $\smash{r^b_1}$ instead of 
$\smash{s^b_1(H)}$ and $\smash{r^b_1(H)}$.

\begin{lemma}
\label{uniqextract} Let $\smash{ (H, \zeta) \ino \bC_{x^+\!, x^-\! , b}}$ where $\smash{ x^+\! , x^-\ino \bbR_+}$ and $b\ino (0, \infty)$. Let $\smash{ (H^i\! , \zeta_i)}$, $i\ino \{ 1,2,3\}$ as in Definition \ref{elemfundef} $(d)$. Then 
$\smash{ s^{_b}_1 \leko \infty}$ and 
$$ \smash{ (H^1\! , \zeta_1) \eqo (H_{\cdot \, \wedge r^b_1}, r^b_1 ), \quad  
 (H^2\! , \zeta_2) \eqo
(H_{(r^b_1+\,  \cdot ) \wedge s^b_1}, s^b_1 \! -\! r^b_1) , \quad  (H^3\! , \zeta_3) \eqo ( H_{s^b_1+ \, \cdot }, \zeta \! -\! s^b_1) }
$$
and in particular $\smash{ x^+\eqo H_{r^b_1}\eqo H_{s^b_1}\, }$, $\smash{ x^-\eqo H_{s^b_1}\! -\! H_\zeta}$ and $\smash{ b\eqo \overline{H}_\zeta \! -\! H_{s^b_1}}$. 
\end{lemma}
\noi
\textbf{Proof.} By definition $\smash{ H_t\eqo H^+_t} $ for all $\smash{ t\ino [0, \zeta_1+ \zeta_2]}$. Then for all 
$\smash{ t\ino (0, \zeta_1]}$, $\smash{ 0\leko H_t \leko b+m_H (t,\zeta_1)}$. Since $\smash{ t \! \mapsto \! m_H(t,\zeta_1)}$ is nondecreasing on 
$\smash{ [0, \zeta_1]}$, we easily get $\smash{ 0\leko \overline{H}_t \leko b+ m_H(t, \zeta_1) \leqo b+ H_t}$. This implies 
$\smash{ s^b_1 \geko \zeta_1}$. Next, for all $\smash{ t\ino [0, \zeta_2]}$ we observe that 
$\smash{ \overline{H}_{\zeta_1+t}\eqo \overline{H}_{\zeta_1}\vee \max_{s\in [0, t]} 
H_{\zeta_1+s}\eqo H_{\zeta_1}} $ $+$ $\smash{  (\overline{H}_{\zeta_1} \! -\!  H_{\zeta_1} ) \vee \overline{H}^{_2}_t }$ 
and $\smash{ H_{\zeta_1+t}\eqo H_{\zeta_1}+ H^{_2}_t}$.  This implies that $\smash{ \, \overline{\! H}_{\! \zeta_1+t}\! -\! H_{\zeta_1+t} 
\eqo (\overline{H}_{\zeta_1} \! -\!  H_{\zeta_1} )  \vee  \overline{H}^{_2}_t} $  $ \! -\!$ $\smash{  H^{_2}_t}$. Since 
$\smash{ \, \overline{\! H}_{\! \zeta_1}\! -\! H_{\zeta_1}  \leko b}$, we get $\smash{ s^b_1\eqo \zeta_1+ \zeta_2}$ and $\smash{ r^b_1\eqo \zeta_1}$, which implies the desired result.  \cqfd

\smallskip

In Proposition \ref{buniceldec} we show for all $b\ino (0, \infty)$ that any function with finite lifetime can be \emph{uniquely written as the concatenation of a descent with amplitude less than $b$, a certain number (possibly null) of $b$-unicellular functions and a rise of amplitude less than $b$}. To this end, for all $(H, \zeta) \ino \bC$ we introduce the following.
\begin{equation}
\label{g0bdef}
\smash{ g^b_0(H)\eqo \sup \big\{ t\ino [0, s^b_1(H) ): H^+_t \eqo 0 \big\} \, .}
\end{equation}
Recall the notation $\smash{ \theta_t H\eqo H_{t+\, \cdot }\! -\! H_t}$, whose lifetime is $\smash{ (\zeta \! -\! t)_+}$. We recursively define the times $\smash{ s_n^b(H)}$, $\smash{ r_n^b(H)}$ and $\smash{ g_n^b(H)}$, $\smash{ n\ino \bbN^*}$, as follows: if $\smash{ s^b_n(H) \leko \infty}$, then we set 
\begin{eqnarray}
\label{snbdef}
\smash{ s^b_{n+1} (H)\eqo s^b_{n} (H)+ s^b_1\big( \theta_{s^b_n(H)}H\big) }\, ,   \!\!\!\!  \!\!\!\! & & \!\!\!\!  \!\!\!\! \smash{ r^b_{\! n+1} (H)\eqo s^b_{n} (H)+ r^b_1\big( \theta_{s^b_n(H)}H\big) } \nonumber \\
\label{snbdef} & \textrm{and} &  \smash{ g^b_{n} (H)\eqo s^b_{n} (H)+ g^b_0\big( \theta_{s^b_n(H)}H\big), }
\end{eqnarray} 
and if $\smash{ s^b_n(H) \eqo \infty}$, we set $\smash{ s^b_{n+1} (H)\eqo r^b_{\! n+1} (H)\eqo  g^b_{n} (H)\eqo \infty}$. 
It is also convenient to set $\smash{ s^b_{0} (H)} $ $\eqo$ $\smash{  r^b_{0} (H)}$ $\eqo  0$. We also define 
\begin{equation}
\label{NbHdef}
\smash{ N_b(H) \eqo \max \big\{ n\ino \bbN : s^b_n(H) \leko \infty \big\} \; .}
\end{equation}
If $\smash{ N_b(H)\eqo 0}$, we observe that 
$\smash{ s^b_{0} (H)\eqo r^b_{0} (H)\eqo  0}$ but $\smash{ g^b_0(H)\eqo \sup \{ t\ino [0, s^b_1(H) ): H^+_t \eqo 0 \}}$ is possibly $>0$. When there is no ambiguity, we simply write $\smash{ s^b_n}$, $\smash{ r^b_n}$, $\smash{ g^b_n}$ and $\smash{ N_b}$.

\begin{remark}
\label{sbnrems} $(a)$ $\smash{ (H,\zeta)\ino \bC_{\! f}}$ is $b$-unicellular if and only if 
$\smash{ N_b(H)\eqo 1}$, $\smash{ g^b_0(H)\eqo 0}$ and $\smash{ g^b_1(H)\eqo \zeta}$. Furthermore 
$\smash{ (H, \zeta) \ino \bC_{x^+\!, x^-\! , b}}$ where $\smash{ x^+\eqo  H_{s^b_1}\, }$, $\smash{ x^-\eqo H_{s^b_1}\! -\! H_\zeta}$.

\smallskip

\noi
$(b)$ Let $\smash{ (H,\zeta)\ino \bC_{\! f}}$ and let $\smash{ 0\leqo t_0\leko  t_1\leqo \zeta}$ be such that 
$\smash{ H_t \geko H_{t_0}\eqo H_{t_1}}$ for all $\smash{ t\ino (t_0, t_1)}$. We set 
$\smash{ b_0 \eqo \max_{t\in [t_0, t_1]} H_t \! -\! H_{t_0}}$ and $\smash{ \overline{t}\eqo 
\inf \{ t\ino [t_0, t_1] : H_t \eqo \max_{s\in [t_0, t_1]} H_s \}}$.
 Then for all $\smash{ b\ino (0, b_0]}$, it is possible to define $\smash{ \gamma_b\eqo 
\max \{ t\ino [t_0, \overline{t} \, ]: H_t \eqo H_{\overline{t}  } \! -\! b \}}$ and $\smash{ \delta_b\eqo 
\min \{ t\ino [\,  \overline{t}, t_1 ]: H_t \eqo H_{\overline{t} } \! -\! b \}}$ and we see  that 
\begin{equation}
\label{excubspot}
\smash{ N_b(H) \geqo 1 \quad \textrm{and} \quad \exists \, k_b\ino \{ 1, \ldots, N_b(H) \} : \gamma_b\eqo r^b_{k_b}(H) \; \textrm{and} \; \delta_b \eqo s^b_{k_b}(H) \, .}
\end{equation}

\noi
$(c)$ Let $\smash{ b\ino (0, \infty)}$, $\smash{ (H, \zeta) \ino \bC_{\! f}}$ and $\smash{ 0\leqo s \leko t\leqo \zeta}$ be such that $\smash{
\big\{ k\ino \{ 1, \ldots, N_b(H) \} : g_k^b(H) \ino [s,t] \big\}} $ $\! =:\! L$ $ \! \neq \! $ $\emptyset $. Then we observe that $\smash{ \min_{r\in [s,t]} H_r \eqo \min_{k\in L} H_{\! g^b_k} }$. \cq 
\end{remark}
\begin{proposition} ($b$-unicellular decomposition)  
\label{buniceldec} Let $\smash{ b\ino (0, \infty)}$ and $\smash{ (H, \zeta)\ino \bC_{\! f}}$. 
\begin{compactenum}

\smallskip

\item[$\mathbf(i)$] There is a unique $\smash{ \bC_{\! f}}$-valued sequence $\smash{ (H^k\! ,\zeta_k)_{0\leq k\leq n+1} }$ such that 
\begin{equation}
\label{buniceldecexpli} 
\smash{ H\eqo H^0\! \centerdot  H^1\! \centerdot \ldots  \centerdot  H^n \! \centerdot  H^{n+1} }
\end{equation}
where $\smash{ (H^0\! ,\zeta_0)}$ is a descent with amplitude $<\!  b$, 
$\smash{ (H^{n+1}\! ,\zeta_{n+1})}$ is a rise with amplitude $<\! b$ and where the $\smash{ (H^k\! ,\zeta_k)_{1\leq k\leq n}}$ are $b$-unicellular, if $n\geqo 1$. 

\smallskip

\item[$\mathbf(ii)$] Moreover 
\begin{equation}
\label{bordure}
\smash{n\eqo N_b (H), \quad  (H^0\! ,\zeta_0)\eqo \big( H_{\cdot \, \wedge g^b_0}\,  , g^b_0 \big)  , \; \, 
(H^{n+1}\! ,\zeta_{n+1})\eqo \big( \theta_{\! g^b_n} H \, , (\zeta \! -\! g^b_n)_+ \big)  }
\end{equation}
and for all $1\leqo k\leqo n$, if $n\geqo 1$, 
\begin{equation}
\label{coredec}
\smash{(H^k\! ,\zeta_k)\eqo \big( H_{(g^b_{k-1} + \, \cdot)\wedge g^b_k }\, , g^b_k\! -\! g^b_{k-1} \big) \; . }
\end{equation}
For all $\smash{k\ino \{ 1, \ldots , N_b (H)\} }$ we set $\smash{\tU^b_k(H)\eqo (H^k\! ,\zeta_k)}$ which is the \emph{$k$-th $b$-unicellular part of $(H, \zeta)$}. For all $\smash{k\geko N_b(H)}$ we simply set $\smash{\tU_k^b (H)\eqo \partial}$. 
\end{compactenum}
\end{proposition}
\noi
\textbf{Proof.} We first prove that the decomposition exists and we start by proving that $\smash{( H (\cdot  \wedge g^{_b}_{^0}) \,  ,g^{_b}_{^0})}$ is a descent with amplitude $<\!  b$, which is trivial if 
$\smash{g^{_b}_{^0}\eqo 0}$. So let us assume that $\smash{g^{_b}_{^0}\geko 0}$. By definition $\smash{\underline{H}(g^{_b}_{^0})\eqo H (g^{_b}_{^0})}$, i.e., $\smash{H^{+}({g^{_b}_{^0}})\eqo 0}$. We then set 
$\smash{b_*\! \eqo \max \{ H^+(t);  t\ino [0, g^{_b}_{^0}] \}}$. Since $\smash{H^{+} ({g^{_b}_{^0}}) \eqo 0}$, we  get $\smash{s^{_{b_*}}_{^1} \leqo g^{_b}_{^0}}$. 
It implies $\smash{b_*\leko b}$ since $\smash{s^{_b}_{^1} \geko g^{_b}_{^0}}$, by definition. This proves the desired result. 

We next take $n$ and $\smash{(H^k\! ,\zeta_k)_{1\leq k\leq n+1} }$ as in (\ref{bordure}) and (\ref{coredec}). 
Clearly (\ref{buniceldecexpli}) holds true. It is also clear from the definition and from Lemma \ref{uniqextract} 
that 
$\smash{(H^k\! ,\zeta_k)_{1\leq k\leq n}}$ are $b$-unicellular. To complete the proof of the existence of the decomposition, 
it remains to prove that $\smash{(H^{n+1}\! ,\zeta_{n+1})}$ is a rise with an amplitude $<\! b$.  To simplify 
the notation, 
we set $\smash{H'(t)\eqo H^{n+1}(t)\! -\! m_{H^{n+1}} (t, \zeta_{n+1})}$ and $\smash{b_* \eqo \max \{ H'(t) ;t\ino 
[0, \zeta_{n+1}] \}}$. By definition, $\smash{H^{n+1}(t) \geko 0}$ for all $\smash{t\ino (0, \zeta_{n+1}]}$ and 
$\smash{H'(0)}$ $\eqo$ $\smash{ H' (\zeta_{n+1})} $ $\eqo 0$. Therefore $\smash{s^{_{b_*}}_{^1}(H^{n+1}) \leqo \zeta_{n+1}}$ and 
we get $\smash{b_* \leko b}$ because $\smash{s^{_b}_{^1}(H^{n+1}) \eqo \infty}$ since $n$ $\eqo$ $\smash{ N_b(H)}$. 
Thus, $\smash{(H^{n+1}\! ,\zeta_{n+1})}$ is a rise with amplitude $<\! b$. It completes the proof of the existence. 

  We now prove the uniqueness of the decomposition by showing recursively 
 for all $n\ino \bbN$ the following property $\smash{\mathtt{Prop} (n)}$:  
\emph{if $\smash{(H^k\! ,\zeta_k)_{0\leq k\leq n+1} }$ are as in $(i)$, then  (\ref{bordure}) and (\ref{coredec}) hold true}. 
To prove $\smash{\mathtt{Prop} (0)}$ we immediately see that if $n\eqo 0$ in $(i)$ 
then $\smash{s^{_b}_{^1} \eqo \infty}$ and $\smash{\zeta_0 \eqo g^{_b}_{^0}}$, which implies (\ref{bordure}) (note that (\ref{coredec}) has no meaning when $n\eqo 0$). 

We next assume  $n\eqo 1$ in $(i)$ and we prove $\smash{\mathtt{Prop} (1)}$. By definition $\smash{H^+(t) \leko b}$ 
for all $\smash{t\ino [0, \zeta_0]}$ and $\smash{H^+ (\zeta_0) \eqo 0}$. For all $\smash{t\ino (0, s^{_b}_{^1}(H^1))}$, 
we then see that $\smash{H^1(t) \geko 0}$, so we get $\smash{H^+ \! (\zeta_0+t) }$ $\eqo$ $\smash{ (H^1)^+ (t)}$ $ \geko 0$.
Since $\smash{\max \{H^+(t)  ; t\ino [0, \zeta_0] \} \leko b}$, we get $\smash{s^{_b}_{^1}(H)\eqo 
\zeta_0+ s^{_b}_{^1}(H^1)}$ and $\smash{r^{_b}_{^1} (H)\eqo \zeta_0+ r^{_b}_{^1}(H^1)}$. Since 
$\smash{H^+ (\zeta_0) \eqo 0}$ 
and $\smash{H^+ (\zeta_0+t)  \geko 0}$ for all $\smash{t\ino (0, s^{_b}_{^1}(H^1))}$, we next obtain 
$\smash{g^{_b}_{^0} (H)\eqo \zeta_0}$. We then see that $\smash{g^{_b}_{^1}(H)\eqo \zeta_0+\zeta_1}$, 
which implies (\ref{bordure}) and (\ref{coredec}). 

 Let $n\ino \bbN$ be such that $n\geqo 2$. We assume for all $0\leqo k\leqo n\! -\! 1$ that 
$\smash{\mathtt{Prop} (k)}$ holds true and we prove that it implies $\smash{\mathtt{Prop} (n)}$. We assume that (\ref{buniceldecexpli}) holds true and we first apply $\smash{\mathtt{Prop} (1)}$ to $\smash{H^0\! \centerdot H^1}$ to see that 
$\smash{(H^0\!, \zeta_0)\eqo (H( \cdot \wedge g^{_b}_{^0}), g^{_b}_{^0})}$ and $\smash{(H^1\! , \zeta_1)\eqo 
(H ( (g^{_b}_{^0}+ \cdot)\wedge g^{_b}_{^1})), g^{_b}_{^1}\! -\! g^{_b}_{^0})}$. We then set 
$\smash{(H'^{0}, \zeta_0')\eqo (H ( (s^{_b}_{^1}+ \cdot)\wedge g^{_b}_{^1})),g^{_b}_{^1}\! -\! s^{_b}_{^1})}$. Then we see that 
$\smash{\theta_{\! s^b_1} H\eqo H'^0\! \centerdot H^2 \! \centerdot \cdots \centerdot H^{n+1}}$. 
We then apply $\smash{\mathtt{Prop} (n\! -\! 1)}$ to $\smash{\theta_{\! s^b_1} H}$ which easily implies (\ref{bordure}) and (\ref{coredec}). It completes the proof. \cqfd

\begin{remark}
\label{infinidec} Lemma \ref{buniceldec} extends to functions $\smash{(H,\zeta) \ino \bC}$ such that $\smash{N_b(H)\eqo \infty}$, which implies that $\smash{\zeta \eqo \lim_{n\to \infty} s^{_b}_{^n} (H) \eqo \infty}$. 
Namely, in this case there exists a unique sequence $\smash{(H^n,\zeta_n)_{n\in \bbN}}$ such that $\smash{(H^0, \zeta_0)}$ is a descent with amplitude $<\! b$, such that the functions 
$\smash{(H^n, \zeta_n)_{n\in \bbN^*}}$ are $b$-unicellular, such that $\smash{H\eqo H^0\! \centerdot H^1 \! \centerdot \ldots \centerdot H^n \! \centerdot \ldots }$ (this infinite concatenation makes sense), such that $\smash{(H^0\!, \zeta_0)\eqo (H( \cdot \wedge g^{_b}_{^0}), g^{_b}_{^0})}$ and such that (\ref{coredec}) hold true.  \cq 
\end{remark}

\begin{definition} (\emph{Erasure of functions})
\label{erafundef} 
Let $\smash{ b\ino (0, \infty)}$ and $\smash{ (H,\zeta)\ino \bC_{\! f}}$. The $b$-erased function $H$ is the zigzag function 
$\smash{ \cE_bH} $ whose lifetime is denoted by $\smash{ M_b(H)}$ and which takes the successive values 
$ \smash{ 0, H_{\! g^b_0}, H_{\! s^b_1}, H_{\! g^b_1}, \ldots,  H_{\! g^b_{N_b-1} }, H_{\! s^b_{N_b}}, H_{\! g^b_{N_b}} , H_{\! \zeta} }$. 
More precisely, $\smash{ \cE_bH }$ is the zigzag function associated with the sign-changes sequence $\smash{ (z_j)_{1\leq j \leq 2N_b+4}}$ which is defined by $\smash{ z_0\eqo z_1\eqo 0}$ and 
\begin{equation}
\label{signseqera}
\smash{  \forall k\ino \{ 1, \ldots, N_b +2\}, \quad 
 z_{2k-1}\eqo z_{2k-2} +  H_{\! s^b_{k-1}} \!\!\!  - H_{\! g^b_{k-2}} \; \,  \textrm{and} \; \,   z_{2k} \eqo z_{2k-1}+ H_{\! s^b_{k-1}}\!\! \! - H_{\! g^b_{k-1}} }
 \end{equation}
with the conventions $\smash{ g_{^{-1}}^{_b}\eqo s^{_b}_{^0}\eqo 0}$ and 
$\smash{ g_{^{N_b+1}}^b\eqo s^{_b}_{^{N_b+1}} \eqo \zeta}$.  \cq 
\end{definition}
\begin{example} $(a)$ Let us recall from Remark \ref{remlifetime} $(a)$ the definition of the elementary functions 
$\smash{ (\phi_x, |x|)\ino \bC_{\! f}}$, for all $\smash{ x\ino \bbR}$. We set $\smash{ H\eqo \phi_{-2} \centerdot \phi_{2}}$, whose lifetime is $\smash{ \zeta\eqo  4}$. Then, for all $b\ino (0, \infty)$, $\smash{ N_b (H)\eqo 0}$, $\smash{ \cE_b H\eqo H}$ and $\smash{ M_b (H)\eqo \zeta\eqo 4}$. 

\smallskip

\noi
$(b)$ Let us define $\smash{ H\eqo \phi_{1} \centerdot \phi_{-2}\centerdot \phi_1\centerdot \phi_{-2}\centerdot \phi_{3} \centerdot \phi_{-2}} \ino \bC_{\! f}$, whose lifetime is $\smash{ \zeta\eqo 11}$. Then $\smash{ N_1(H)\eqo 3}$, $\smash{ \cE_1H\eqo \phi_{-2} \centerdot  \phi_{2}\centerdot \phi_{-1}}$ 
and $\smash{ M_1(H)\eqo 5}$. \cq 
\end{example}

\begin{remark}
\label{erafunrem} $(a)$ $b$-erasure extends obviously to continuous functions $H$  such that $\smash{ N_b(H)\eqo \infty}$ and $\smash{ \lim_{n\to \infty} s^{_b}_{n} (H)\eqo \infty}$.

\smallskip

\noi
$(b)$ $b$-erasure does not depend on a specific parametrization of $\smash{ (H, \zeta) \ino \bC_{\! f}}$. Indeed let $\smash{ (H'\! , \zeta')\ino \bC_{\! f}}$ be such that $\smash{ H_t\eqo H'_{\! J_t}}$ where $\smash{ J\! :\! [0, \zeta] \! \to \! [0, \zeta']}$ is continuous nondecreasing and surjective. Then 
$\smash{ (\cE_b H, M_b(H))\eqo (\cE_bH' \! , M_b(H'))}$. 
 
\smallskip

\noi
$(c)$ The sign-changes sequence $\smash{ (z_j)_{1\leq j \leq 2N_b +4}}$ in (\ref{signseqera}) may not correspond to a minimal cell: it is for instance possible that $\smash{ H_{s^b_n}\eqo H_{g^b_n}}$ and therefore 
$\smash{ N_b}$ may be larger than the number of strict local maxima of $\smash{ \cE_b H}$.

\smallskip

\noi
$(d)$ Let $\smash{ h\ino (0, \infty)}$ and $\smash{ x^+\! , \, x^-\ino \bbR_+}$. If $\smash{ (H, \zeta) \ino \bC_{x^+\! , \, x^-\! ,\,  h}}$, then 
$\smash{ (\cE_b H, M_b (H))\ino \bC_{x^+\! , \, x^-\! ,\,  h-b}}$ for all $b\ino (0, h]$. \cq 
\end{remark}

Let $\smash{ (T,d,\rho, <)}$ be an ordered rooted compact real tree and let $\smash{ b\ino \bbR_+}$. 
We recall from (\ref{berased1}) the definition of the $b$-erased tree $\smash{ \bcE_bT}$. We root $\smash{ \bcE_bT}$ 
at $\rho$ and we equip it with the restrictions of the metric $d$ and of the order $<$. Thus $\smash{ (\bcE_b T, d, \rho, <)}$ 
is an ordered rooted compact real tree. Moreover if $\smash{ (T',d',\rho', <')}$ is an ordered rooted compact real tree 
such that $\smash{ (T',d',\rho', <')\! \equiv \! (T,d,\rho, <)}$, then 
$\smash{ (\bcE_b T',d',\rho', <')\! \equiv \! (\bcE_b T,d,\rho, <)}$. 
We also recall from (\ref{semigrtreeera}) semigroup property and we easily see,  
for all $\smash{ b, b' \ino \bbR_+}$, that $\smash{ \bcE_b (\bcE_{b'} T)\eqo \bcE_{b+b'} T}$, as ordered 
rooted compact real tree.   
The following example shows that $\smash{ \bcE_b (\mathtt{Tree} (H))}$ is not necessary equal to 
$\smash{ \mathtt{Tree} (\cE_b H)}$. However if we restrict to $h$-unicellular functions with $h\geqo b$, the next 
lemma shows that it is the case.  
\begin{example}
\label{eracounterex} Let $\smash{ H\eqo \phi_{-1} \centerdot \phi_{1} \centerdot \phi_{-1} \centerdot \phi_{1} }$ whose lifetime is 
$\zeta\eqo 4$. Then 
$\smash{ \mathtt{Tree} (H)}$ is a star-tree consisting of three segments of unit length sharing one 
endpoint, which is the root. Therefore, 
$\smash{ \bcE_1 (\mathtt{Tree} (H))}$ reduces to the point tree. 
However $\smash{ \cE_1 H\eqo \phi_{-1} \centerdot \phi_{1} }$. Its lifetime is therefore $\smash{ M_1(H)\eqo 2}$ and 
$\smash{ \mathtt{Tree} (\cE_1H)}$ is the star-tree consisting of two segments of unit length sharing one endpoint, which is the root. \cq
\end{example}
\begin{lemma}
\label{unicelprop} Let $\smash{ h\geqo b \geko 0}$ and let $\smash{ (H, \zeta) \ino \bC_{\! f}}$ be $h$-unicellular. 
We recall from Definition \ref{zigzagdef} $(b)$ the notation notations $\smash{ \mathscr M (\cE_bH)}$ and $\smash{ \tnn (\cE_b H)\eqo \# \mathscr M (\cE_bH)}$, we recall from (\ref{NbHdef}) the notation $\smash{ N_b(H)}$ and we recall $\smash{ (\cE_b H , M_b(H))}$ from Definition \ref{erafundef}. Then the following holds true. 
\begin{compactenum}

\smallskip

\item[$(i)$] $\smash{ m_H(0, s^{_b}_{^1})\eqo 0}$ and $\smash{  m_H(s^{_b}_{^{N_b}}, \zeta)\eqo H_\zeta}$

\smallskip

\item[$(ii)$] $\smash{ b\ino (0, h] \! \mapsto \! N_b(H)}$ is nonincreasing left-continuous, $\smash{ b\ino (0, h] \! \mapsto \! \tnn (\cE_b H)}$ is nonincreasing rigth-continuous and $\smash{ N_{b+} (H)\eqo\tnn (\cE_b H)}$. 

\smallskip

\item[$(iii)$] We recall that $\smash{ \mathtt{Tree} (H)\eqo (\cT_{\! H}, d_H, \rho_H , <_H)}$. We also recall from Definition \ref{margtreedef} $(a)$ that
$\smash{ T\eqo \bigcup_{1\leq k\leq N_b } \lgeo \rho_H, \pcH(s^{_b}_{^k}) \rgeo}$ is the marginal tree associated with 
$\smash{ (H, \zeta, (s^{_b}_{^k})_{1\leq k\leq N_b})}$. Then, 
\begin{equation}
\label{eraera}
\smash{ \mathtt{Tree} (\cE_b H) \equiv (T, d_H, \rho_H, <_H) = \bcE_b \big( \mathtt{Tree} (H)\big). }
\end{equation} 
In particular it implies $\smash{ b+ 
\max \big\{ (\cE_bH )(z) \, ; \, z\ino [0, M_b] \big\} \eqo \max \big\{ H_t \, ; \, t\in [0, \zeta] \big\}}$.  
\end{compactenum}
\end{lemma}
\noi
\textbf{Proof.} $(i)$ is a direct consequence of the fact that $H$ is $h$-unicellular and of the very definition of the times $\smash{ s^{_b}_{^1}}$ and $\smash{ s^{_b}_{^{N_b}}}$
We next prove $(ii)$. For all $\smash{ 1\leqo k\leqo N_b}$, for all $\smash{ h\geqo b\geko  a \geko 0}$ we set 
\begin{equation}
\label{oversbk}
\smash{ \overline{s}^{_b}_{^k} \eqo  \min \Big\{ t\ino [r^{_b}_{^k}, s^{_b}_{^k}] :  H_t \eqo \!\! \!\! \!\!  \max_{\quad s\in  [r^{_b}_{^k}, s^{_b}_{^k}]} \!\! \!\! \!\! \!  H_s \,  \Big\}, }
\end{equation}   
$$\smash{  \gamma^{_{b,a}}_{^k}\eqo \max \{ t\ino [r^{_b}_{^k} , \overline{s}^{_b}_{^k}] : H_t \eqo H_{\overline{s}^{_b}_{^k}} \! -\!   a \} \quad \textrm{and} \quad  \delta^{_{b,a}}_{^k}\eqo \max \{ t\ino [r^{_b}_{^k} , \overline{s}^{_b}_{^k}] : H_t \eqo H_{\overline{s}^{_b}_{^k}} \! -\! a \}.}$$
By Remark \ref{sbnrems} $(b)$, $\smash{ \{  (\gamma^{_{b,a}}_{^k}, \delta^{_{b,a}}_{^k}) ; 1\leqo k\leqo N_b \} \! \subset \! \{ (r^{_a}_{^l}, s^{_a}_{^l}); 1\leqo l\leqo N_a \}}$. Thus $\smash{ b\ino (0, h] \! \mapsto \! N_b}$ is nonincreasing. Moreover, since $H$ is continuous the function $\smash{ b\mapsto N_b}$ has to be left-continuous.  
To complete the proof of $(ii)$, we fix $0\leko a \leko b$ such that $\smash{ N_a\eqo N_b}$ (there is necessarily such a 
real number). Necessarily, we get $\smash{ (\gamma^{_{b,a}}_{^k}, \delta^{_{b,a}}_{^k}) \eqo (r^{_a}_{^k}, s^{_a}_{^k})}$ 
for all $\smash{ 1\leqo k\leqo N_b}$. Thus, $\smash{ H_{\! s^{_a}_{^k}} \! \wedge \!  H_{\! s^{_a}_{^{k+1}}} \!\!\!  -  
m_H  (s^{_a}_{^k}, s^{_a}_{^{k+1}}) }$ $\geqo$ $ b \! -\!  a$,  for all $\smash{ 1\leqo k \leko N_a}$. It implies that 
$\smash{ \cE_aH}$ has $\smash{ N_a}$ local maximal. Namely, $\smash{ \tnn (\cE_aH)\eqo N_a}$. Therefore, it shows that 
$\smash{ N_b\eqo \tnn(\cE_b H)}$, for all $b\ino (0, h] $ which are not jump-times of the function 
$\smash{ b'\mapsto N_{b'}}$. We now complete the proof of $(ii)$ by observing that 
$\smash{ b\mapsto \tnn(\cE_b H)}$ is necessarily right-continuous since the $\smash{ \cE_bH}$ 
are zigzag functions.  
 
We next prove $(iii)$. Remark \ref{subtreerem} $(c)$ asserts that $\mathtt{Tree} (\cE_b H) \equiv (T,d_H,\rho_H, <_H)$. So we only need to prove that $T\eqo \bcE_b \cT_{\! H}$. To this end we first check easily for all $1\leqo k\leqo N_b$ that 
$\pcH(s^{_b}_{^k}) \ino \lgeo \rho_H , \pcH (\overline{s}^{_b}_{^k})  \rgeo $ and that $d_H( s^{_b}_{^k},\overline{s}^{_b}_{^k})\eqo b$, which implies $T \! \subset \!  \bcE_b \cT_{\! H}$. 

  It remains to prove that $\smash{ \bcE_b \cT_{\! H}\! \subset \! T} $. Let us fix 
$\smash{ \sigma \ino \mathtt{Lf} (\bcE_b \cT_{\! H})}$. By definition, there is 
$\smash{ \overline{\sigma}\ino \cT_{\! H}}$ such that $\smash{ \sigma \ino \lgeo  \rho_H, \overline{\sigma} \rgeo }$ and 
$\smash{ d_H(\sigma,  \overline{\sigma} )\eqo b}$, and such that for all $\smash{ \gamma \ino \mathcal T_{\! H}}$, 
if $\smash{ \sigma \ino \lgeo  \rho_H, \gamma \rgeo }$, then $\smash{ d_H(\sigma, \gamma) \leqo b}$. 
Thus we can find $ \smash{ r , \overline{s} \ino [0, \zeta]}$ such that $\smash{ \pcH(r) \eqo \sigma}$, 
$\smash{ \pcH( \overline{s} )\eqo  \overline{\sigma} }$. Therefore $\smash{ H_r\eqo m_H(r,\overline{s})\eqo 
H_{\overline{s}}\! -\! b}$ and for all $\smash{ s\ino [0, \zeta]}$ such that $\smash{ H_r \eqo m_H(r,s)}$, we get 
$\smash{ H_s \leq H_r+b}$. 
In particular for all $\smash{ t\ino [r\wedge \overline{s}, r\vee \overline{s}]}$, we get $\smash{ H_r \leqo H_t \leqo 
H_r+ b\eqo H_{\overline{s}}}$. Without loss of generality, we can always assume that $\smash{ H_r} $ $\leko$ $\smash{  H_t 
} $ $\leqo$ $\smash{  H_r+ b} $ $\eqo$ $\smash{  H_{\overline{s}}}$ for all $\smash{ t\ino (r\wedge \overline{s}, r\vee \overline{s})}$. Then, 
we prove that 
\begin{equation}
\label{closeexcu}
\smash{ m_H(0, \overline{s})\vee m_H(\overline{s}, \zeta) \leqo H_r}
\end{equation}

Before proving (\ref{closeexcu}), let us first show that it implies $(iii)$. \emph{Indeed}, thanks to (\ref{closeexcu}), it 
makes sense to define $s\eqo \min \{ u\ino [\overline{s}, \zeta] : H_u\eqo H_r \}$ and $s'\eqo \max \{ u\ino [0,\overline{s}] : H_u \eqo H_r \}$. Then $r\ino \{ s,s'\}$ and by Remark \ref{sbnrems} $(b)$ there is $k\ino \{ 1, \ldots, N_b\}$ such that $r^{_b}_{^k}\eqo s' $ and $s^{_b}_{^k}\eqo s$. Consequently,  $\sigma \eqo \pcH(r)\eqo \pcH(r^{_b}_{^k})\eqo  \pcH(s^{_b}_{^k})  \ino T$. We thus have proved that $\mathtt{Lf} (\bcE_b \cT_{\! H}) \!\subset \! T$, which easily entails  that $\bcE_b \cT_{\! H}\! \subset \! T $. It completes the proof of $(iii)$.

\emph{It remains to prove (\ref{closeexcu})}. There are three cases to consider: either $\smash{ r\ino [0, r^{_h}_{^1} ]}$ or $\smash{ r\ino (r^{_h}_{^1} ,s^{_h}_{^1} )}$ or $\smash{ r\ino [s^{_h}_{^1} , \zeta ]}$. 
First we suppose that $\smash{ r\ino [0, r^{_h}_{^1} ]}$. Since $H$ on $ \smash{ [0, r^{_h}_{^1} ]}$ is a rise 
$\smash{ m_H(0, r)\eqo 0\leqo H_r }$
and $\smash{ m_H(r,  r^{_h}_{^1})\leqo H_r}$. Note that $\smash{ m_H (0, \overline{s} ) \leqo 0 \leqo H_r}$ and 
if $\smash{ \overline{s} \leko r}$, then 
$\smash{ m_H( \overline{s}, \zeta) \leqo m_h( r, r^{_h}_{^1}) \leqo H_r}$. If $\smash{ \overline{s} \geko r}$, then 
$\smash{ \overline{s}\leko  r^{_h}_{^1}}$ and $\smash{ m_H(r,  r^{_h}_{^1})\eqo m_H(\overline{s},  r^{_h}_{^1})}$ 
and thus $\smash{ m_H(\overline{s}, \zeta) \leqo 
m_H(\overline{s},  r^{_h}_{^1})\leqo H_r}$. Similarly, if $\smash{ r\ino [s^{_h}_{^1}, \zeta ]}$, since $H$ is a descent 
we can argue in the same way, but  backwards in time, to show that $\smash{H_r \geqo m_H(0, \overline{s})\vee m_H(\overline{s}, \zeta) 
}$. Let us suppose that $\smash{ r\ino (r^{_h}_{^1} ,s^{_h}_{^1} )}$. Necessarily 
$\smash{ \overline{s} \ino (r^{_h}_{^1} ,s^{_h}_{^1} )}$ too and consequently $\smash{ H_r \geqo 
H (r^{_h}_{^1})\eqo m_H(r,r^{_h}_{^1})\eqo m_H(r, s^{_h}_{^1})\eqo m_H(\overline{s}, r^{_h}_{^1})  \eqo 
m_H(\overline{s}, s^{_h}_{^1})}$, which implies that $\smash{ m_H(\overline{s}, 0)\vee m_H(\overline{s}, \zeta) 
\leqo H_r}$. This completes the proof of (\ref{closeexcu}) and thus of the lemma. \cqfd

\begin{lemma}
\label{semigrera} For all $(H, \zeta)\ino \bC_{\! f}$ and all $b, b' \ino \bbR_+$, $\cE_b (\cE_{b'} H)\eqo \cE_{b+b'} H$.  
\end{lemma}
\noi
\textbf{Proof.} We first suppose that $(H, \zeta)$ is $h$-unicellular with $h\geqo b+b'\geko 0$. 
We set $(Z, \zeta_Z)\eqo (\cE_{b'} H, M_{b'} (H))$, which is $(h\! -\! b')$-unicellular by Remark \ref{erafunrem} $(d)$. 
Since $h\! -\! b'\geko b$, Lemma \ref{unicelprop} $(iii)$ applies and $\mathtt{Tree} (\cE_b Z) \! \equiv \! 
\bcE_b (\mathtt{Tree} (Z))$. Since $h\geko b'$, 
Lemma \ref{unicelprop} $(iii)$ applies and $\mathtt{Tree} (Z)\eqo \bcE_{b'} (\mathtt{Tree} (H))$. Thus 
$\mathtt{Tree} (\cE_b Z) \! \equiv \! \bcE_b ( \bcE_{b'} (\mathtt{Tree} (H)))\eqo \bcE_{b+b'} (\mathtt{Tree} (H))$ 
by the semigroup property  (\ref{semigrtreeera}). Since $h\geko b+b'$, Lemma \ref{unicelprop} $(iii)$ again shows that 
$\bcE_{b+b'} (\mathtt{Tree} (H))) \! \equiv \!  \mathtt{Tree} (\cE_{b+b'} H)$. We have prove that 
$\mathtt{Tree} (\cE_b Z) \! \equiv \!  \mathtt{Tree} (\cE_{b+b'} H)$. Next, we observe that $(Z,\zeta_Z)$ 
satisfies (\ref{upanddown}) and that $\cE_{b+b'} H$ is a zigzag function whose terminal value is 
$H_\zeta\eqo Z (\zeta_Z)$. Then Lemma \ref{zigzagtree} $(ii)$ applies and $\cE_b Z\eqo \cE_{b+b'} H$, which is the desired result.  

We prove the general case $\smash{ (H, \zeta) \ino \bC_{\! f}}$ thanks to the unicellar case as follow. 
Let $h, x$ be such that $\smash{ x\geko h\geko b+b'+\, \overline{\! H}_\zeta \! -\!  \underline{H}_{\zeta}}$ and 
set 
$\smash{ H'\! :=\!  \phi_{x}\centerdot \phi_h \centerdot   \phi_{-h}  \centerdot  H   \centerdot \phi_h 
\centerdot   \phi_{-h}  \centerdot \phi_{-x}}$ whose lifetime is $\smash{ \zeta' \! :=\!  2x+4h+\zeta}$. 
Then $\smash{ (H',\zeta')}$ is $h$-unicellular and for all $\smash{ a\ino (0, h)}$ we get 
$\smash{ \cE_a H'\eqo \phi_{x}\centerdot } $ $\smash{ \phi_{h-a} \centerdot  }$ $\smash{  \phi_{-(h-a)}  \centerdot } $ $\smash{  \cE_a H   
\centerdot} $ $\smash{  \phi_{h-a} \centerdot } $ $\smash{   \phi_{-(h-a)}  \centerdot } $ $\smash{ \phi_{-x}}$. Since the semigroup property holds in 
$h$-the unicellular cases with $h\geko b+b'$, we get 
$ \smash{ \phi_{x}\centerdot \phi_{h-b-b'} \centerdot   \phi_{-(h-b-b')}  \centerdot  \cE_b (\cE_{b'} H)  
\centerdot \phi_{h-b-b'} \centerdot   \phi_{-(h-b-b')}  \centerdot \phi_{-x}} $ $ \eqo$ $\smash{  \cE_b (\cE_{b'} H')} $ $\eqo$ $\smash{ \cE_{b+b'}H'\eqo  \phi_{x}\centerdot \phi_{h-b-b'} \centerdot   \phi_{-(h-b-b')}  \centerdot \cE_{b+b'} H 
\centerdot \phi_{h-b-b'} \centerdot   \phi_{-(h-b-b')}  \centerdot \phi_{-x}} $. It entails 
$\smash{ \cE_b (\cE_{b'} H)\eqo \cE_{b+b'} H}$ and the proof is completed.   \cqfd

\subsection{Erasure and growth-division evolutions: deterministic results}
\label{detererasureprocsec}
We next discuss how erasure of real tree and growth of cell population are related. 
To this end we first we recall from Remarks \ref{remlifetime} $(a)$ the definition of the elementary function $\smash{ (\phi_x , |x|)\ino \bC_{\!f}}$ for all $\smash{ x\ino \bbR}$. For all $\smash{ b\ino \bbR_+}$ we denote by $\smash{ \psi_b}$ is the $\wedge$-shape zigzag function $\smash{  \psi_b(t) \eqo  \big(t \! -\! 2(t-b)_{+}\big)_+ }$, $\smash{ t\ino \bbR_+}$. Namely,   
\begin{equation}
\label{psibdef} 
\smash{ \psi_b\eqo \phi_b \centerdot \phi_{-b}, \quad \textrm{ with lifetime $2b$.}}
\end{equation} 
\begin{definition}(\emph{Growth of zigzag functions}) \label{growthzigzag}
Let $\smash{ (Z, \zeta_Z)\ino \bC_{\! f}}$ be a zigzag function. Let us denote by $\smash{((x^+_k, x^-_k))_{1\leq j\leq \tnn(Z)}\eqo \mathtt{C} (Z)}$ the minimal cell population of $\smash{(Z, \zeta_Z)}$ as in Definition \ref{cellfromzigzag}. 
For all $\smash{b\ino \bbR_+}$, we define the \emph{$b$-growth of $Z$} as the zigzag function 
$$\smash{\mathtt{G}\overline{\mathtt{ro}}_b (Z)= \phi_{x^+_1 } \! \centerdot \psi_b \centerdot \phi_{-x^-_1} \! \centerdot \ldots \centerdot \phi_{x^+_{\! \tnn(Z)} } \! \centerdot \psi_b  \centerdot \phi_{-x^-_{\! \tnn(Z)}} }$$
whose lifetime is $\smash{\zeta(\grob_b(Z) )\eqo \zeta_Z + 2b \tnn(Z)}$. See Figure \ref{reflected}.  \cq 
\end{definition}
\begin{remark}
\label{unicelerarems} $(a)$ We keep the previous notations and we first observe that 
$\smash{ \grob_b(Z)}$ starts with a rise and and ends with a descent. We next recall the function $\smash{ (a, \mathbf{c})\ino \bbR_+ \times \mathtt{Cell} \! \mapsto \! \mathtt{Gro}_a (\mathbf{c}) \ino \mathtt{Cell}}$ from Definition \ref{GroDivdef} $(a)$. 
Then note that $\smash{ \tnn(Z)\eqo \tnn (\grob_b(Z)) }$ and 
\begin{equation}
\label{grogrob}
\smash{  \mathtt{C} \big( \grob_b(Z) \big)\eqo \mathtt{Gro}_{2b} \big( \mathtt{C} (Z)\big)\eqo \big( (b+x^+_k, b+x^-_k)\big)_{1\leq k\leq \tnn(Z)}  \; , }
\end{equation}
where we recall from Definition \ref{cellfromzigzag} the notation $\smash{ \mathtt{C}(Z')}$ for the minimal cell associated with a zigzag function $\smash{ Z'}$. If $\smash{ (z_j)_{1\leq j\leq 2\mathtt{n} (Z)}}$ stands for the sign-changes sequence of $\smash{ (Z,\zeta_Z)}$, then we check that 
$\smash{ \mathscr M (\grob_b (Z))\eqo \big\{ z_{2k-1} \! +\!  (2k\! -\! 1) b \;  ; 1\leqo k\leqo \tnn(Z)\big\}}$. 

\smallskip

\noi
$(b)$ Let $\smash{ (Z, \zeta_Z)}$ and $\smash{ (Z',\zeta_{Z'}) \ino \bC_{\! f}}$ be zigzag functions that satisfy the following 
$\smash{ \tnn (Z')\eqo \tnn(Z)}$ and $\smash{ (\cE_b Z' , M_b (Z'))\eqo (Z, \zeta_Z)}$. If $\smash{ (Z, \zeta_Z)}$ furthermore satisfies (\ref{upanddown}), then we easily see that 
$\smash{ Z'\eqo \grob_b(Z)}$ and $\smash{ \zeta_{Z'}\eqo \zeta_Z + 2b\tnn (Z)}$. Moreover 
\begin{equation}
\label{grovseraexpli}
\smash{ \forall a\ino (0, b), \quad  \grob_{a} (Z) \eqo \cE_{b-a} \big( \grob_b (Z)) \; . }
\end{equation}
Let us first mention that if $Z$ is $h$-unicellular, then $\smash{\grob_a(Z)}$ is $(h\! +\!  a)$-unicellular.

\smallskip

\noi
$(c)$ Let $\smash{ (Z, \zeta_Z)\ino \bC_{\! f}}$ be $h$-unicellular and $a, b\ino (0, \infty)$ such that $b\geko a$. 
By the previous remark, $\smash{ \grob_{b} (Z)}$ is $(h\! +\! b)$-unicellular and $\smash{ \grob_{a} (Z)}$ 
is $(h\!  +\! a)$-unicellular. Then, by (\ref{grovseraexpli}) and Lemma \ref{unicelprop}, 
$\smash{ \mathtt{Tree} (\grob_{a} (Z)) \! \equiv \! \bcE_{b-a} (\mathtt{Tree} (\grob_b Z))}$. 
It yields an elementary subtree decomposition which allows to view $\smash{\grob_{a} (Z)}$ as resulting from 
$\smash{\grob_{b} (Z)}$ by a time-change which is defined as follows. 
For any compact $\smash{K \subset \bbR_+}$, we use the notation 
$\smash{\mathtt{dist} (y, K)\! :=\! \min \{ |y\! -\! x| \, ; x\ino K \}}$. Then for all $\smash{z\ino [0, \zeta_Z \! +\!  2b \tnn(Z)]}$ we set 
\begin{equation}
\label{timchaele}
\smash{I(a,z)= \int_0^z\!\! \un_{\! \big\{  \mathtt{dist} \big( y ; \mathscr M (\grob_b (Z))\big) > b-a  \big\}}\,  \mathrm d y  \quad  \textrm{and} \quad \psi_{b-a}^{\centerdot \, \tnn(Z)} \eqo \underbrace{\psi_{b-a} \centerdot \ldots \centerdot \psi_{b-a}}_{\textrm{$\tnn(Z)$ times }} }
\end{equation}
We easily see for all $\smash{z\ino [0, \zeta_{Z}\! + \! 2b\tnn (Z)]}$ that 
\begin{equation}
\label{subdecgrogro}
\smash{\grob_b (Z) (z) = \psi_{b-a}^{\centerdot \, \tnn(Z)}  \big( z\! -\! I(a,z) \big) + \grob_{a} (Z) \big( I(a,z)\big)\; .  }
\end{equation} 
Namely, $\smash{I (a, \cdot) }$ is the local time and $\smash{\psi_{b-a}^{\centerdot \, \tnn(Z)}}$ is the reflected process on the 
$\smash{(b\! -\! a)}$-erased subtree of the real tree coded by $\smash{\grob_b(Z)}$.  \cq 
\end{remark} 

Let $\smash{ (H,\zeta) \ino \bC_{\! f}}$ be $h$-unicellular. We study the \emph{erasure process} $\smash{ b\ino (0, h] \! \mapsto \big(\cE_b H, M_b(H) \big)\ino \bC_{\! f}}$. We assume that it is not trivial. Namely, we suppose that 
$\smash{ \lim_{b\downarrow 0} N_b(H)\eqo \infty}$. 
In particular we explain how it 
yields a linear growth-division evolution as introduced in Definition \ref{lingrodivevodef}. 
To this end, we first define the \emph{erasure sequence} $\smash{ (b_n, y_n)_{n\in \bbN^*}}$ of $\smash{ (H, \zeta)}$. 
Namely, $\smash{ \{ b_n\, ; n\ino \bbN^*\}}$ is the set of discontinuities of $\smash{ b\ino (0, h] \mapsto N_b(H)}$ (it is also the set of discontinuities of $\smash{ b\ino (0, h] \mapsto \tnn(\cE_bH)}$ by Lemma \ref{unicelprop} $(ii)$) and $\smash{ y_n}$ is the time at which the corresponding local maximum of the function $\smash{ z\! \mapsto (\cE_{b_n} H) (z)}$ has vanished. More precisely, we introduce the following. 
\begin{definition} (\emph{Erasure sequence})
\label{eraseqdef} Let $\smash{ h\ino (0, \infty)}$ and let $\smash{ (H, \zeta)\ino \bC_{\! f}}$ be a $h$-unicellular function such that $\smash{ \lim_{b\downarrow 0} N_b(H)\eqo \infty}$. Its \emph{erasure sequence} $\smash{ (b_n, y_n)\ino(0, h] \! \times \! \bbR_+}$, $\smash{ n\ino \bbN^*}$, is defined by the following conditions that hold for all $\smash{ n\ino \bbN^*}$. 
\begin{compactenum}

\smallskip

\item[$\mathbf{(a)}$] $\smash{ \inf_{b\in (0, b_n)} \tnn(\cE_b H) \geko \tnn (\cE_{b_n} H)} $ and $\smash{ b_{n+1} \leqo b_n}$.

\smallskip

\item[$\mathbf{(b)}$] For all $\smash{ m \ino \bbN^*}$ such that $\smash{ b_{n+m}\eqo b_n }$, if any, $\smash{ y_{n} \leqo y_{n+m}}$. 

\smallskip

\item[$\mathbf{(c)}$] $\smash{ \sum_{z\in \mathscr M(\cE_{b_n} H) } \delta_z + \sum_{m\in \bbN^*\! : \, b_{m}= b_n} \delta_{y_m} \eqo \lim_{\varepsilon \downarrow 0} \sum_{z\in \mathscr M(\cE_{b_n-\varepsilon} H) } \delta_z}$. \cq 
\end{compactenum}
\end{definition}
\noi
Erasure sequences are related to division sequences as follows. 
\begin{proposition}
\label{celevovsera} Let $\smash{ h \ino (0, \infty)}$, $\smash{ x^+\! , x^-\! \ino \bbR_+}$. Let $\smash{ (H, \zeta)\ino \bC_{\! f}}$ be a $h$-unicellular function with initial cell $\smash{ (x^+\! , x^-)}$. We assume that  $\smash{ \lim_{b\downarrow 0} N_b(H)\eqo \infty}$ and we denote by $\smash{ (b_n, y_n)_{n\in \bbN^*}}$ the erasure sequence of $\smash{ \big((\cE_b H, M_b(H)) \big)_{b\in (0, h]}}$ as in Definition \ref{eraseqdef}. We set 
\begin{equation}
\label{connecexpli}
\smash{ \forall t\ino [0, 2h), \quad \bcc(t) \eqo \mathtt{C} \big( \cE_{h-\frac{_1}{^2} t } H \big) \; .}
\end{equation}
Then $\smash{ (\bcc(t))_{t\in [0, 2h)}}$ is a linear growth-division evolution.as in Definition \ref{lingrodivevodef}, which is completely characterized by its initial cell $\smash{ \bcc(0)\eqo (x^+\! , x^-)}$ and its division sequence $\smash{ (t_n, y_n)_{n\in \bbN^*}}$ where we have set $\smash{ t_n \eqo 2(h \! -\! b_n) }$. In particular we get for all $\smash{ t\ino [0, 2h)}$ 
\begin{equation}
\label{sizemassconnec} 
\smash{ \tnn \big( \bcc(t) \big)  \eqo \tnn \big(\cE_{h-\frac{_1}{^2} t } H \big) \quad \textrm{and} \quad \tmm (\bcc(t)) \eqo M_{h-\frac{_1}{^2} t} (H) \; .}
\end{equation}   
\end{proposition}
\noi
\textbf{Proof.} Since $\smash{ (H, \zeta)}$ is $h$-unicellular with initial cell $\smash{ (x^+\! , x^-)}$, 
we see that $\smash{ \cE_hH\eqo \phi_{x^+}\!  \centerdot \phi_{-x^-}}$. Thus, 
$\smash{ \bcc(0)\eqo \mathtt{C} (\phi_{x^+}\!  \centerdot \phi_{-x^-})\eqo (x^+\! , x^-)}$ and 
$\smash{ M_h(H)\eqo x^+ \! + x^-\eqo \tmm (\bcc(0))}$. 
Let $m\geqo 0$ and $n\geqo1$ be such that $\smash{ b_{n+m+1} \leko b_{n+m} \eqo \ldots \eqo b_n \leko b_{n-1}}$, with the convention $\smash{ b_0\eqo h}$. We recall from Definition \ref{GroDivdef} $(b)$ the function $\smash{ \mathtt{Div}_y\! : \! \mathtt{Cell} \! \to \! \mathtt{Cell}}$. 
Then by Definition \ref{eraseqdef} $(c)$, we get 
\begin{equation}
\label{divera}
\smash{  \lim_{\varepsilon \downarrow 0} \mathtt{C} \big( \cE_{b_n-\varepsilon} H\big) \eqo \big( \mathtt{Div}_{y_n} \circ \ldots \circ  \mathtt{Div}_{y_{n+m}}\big) \big( \mathtt{C} (\cE_{b_n} H )\big) }
 \end{equation}
To simplify let us denote by $\smash{ \bcc'}$ the limit in (\ref{divera}) and let us set $n'\eqo n+m$. As a consequence 
of Definitions \ref{GroDivdef} and \ref{concadef} we see that 
$\smash{ \mathtt{Z} (\mathtt{Div}_y (\bcc))\eqo \mathtt{Z}(\bcc)}$ for any $\smash{ \bcc\ino   \mathtt{Cell}}$ and 
any $\smash{ y\ino \bbR_+}$. Thus $\smash{ \mathtt{Z} (\bcc') \eqo \cE_{ b_{ n'}}H\eqo \cE_{b_n} H}$. 
Moreover, by definition of the sequence $\smash{ (b_p)_{p\in \bbN^*}}$, for all $\smash{ a\ino (0, b_{n'} \! -\! b_{n'+1} ]}$, 
$\smash{ \tnn (\cE_{b_{n'}-a} H)\eqo \tnn(\bcc')}$ and thus 
$\smash{ \tnn (\cE_{b_{n'+1}} H)\eqo \tnn(\bcc')}$
since the function 
$\smash{ b\! \mapsto\!  \tnn (\cE_b H)}$ is right-continuous by Lemma \ref{unicelprop} $(ii)$. 
For all $\smash{ a\ino (0,  b_{n'} \! -\! b_{n'+1}]}$ the semigroup property stated in Lemma \ref{semigrera} 
implies $\smash{ \cE_{b_{n'}} H\eqo \cE_a(\cE_{b_{n'}-a} H)}$. 
Then Remark \ref{unicelerarems} $(b)$ applies to 
$\smash{ (Z', \zeta_{Z'})\! :=\! (\cE_{b_{n'}-a}H, M_{b_{n'}-a} (H))}$ and $\smash{ (Z, \zeta_Z) \! :=\! (  \cE_{b_{n'}} H, M_{b_{n'}} (H))}$ and we get $\smash{ \cE_{b_{n'}-a} H\eqo \grob_{a} (\cE_{b_{n'}} H)}$ and 
$\smash{ M_{b_{n'}-a }(H) } $ $\eqo$ $\smash{  M_{b_{n'}-a} (H)}$ $ +$ $\smash{   2a\tnn(\bcc')}$. 
By (\ref{grogrob}) in Remark \ref{unicelerarems} $(a)$ it entails 
$$\smash{ \forall a\ino (0, b_{n'} \! -\! b_{n'+1}], \quad 
\mathtt{C} (\cE_{b_{n'}-a} H)\eqo \mathtt{C} \big( \grob_{a} (\cE_{b_{n'}} H ) \big) \eqo \mathtt{Gro}_{2a} \big( \mathtt{C} (\cE_{b_{n'}} H) \big) = \mathtt{Gro}_{2a} (\bcc') \; .}$$ 
This implies the desired result by Definition \ref{lingrodivevodef} of linear growth-division evolution.  
\cqfd

\smallskip

To prove Theorem \ref{recovercell} we shall use Lemma \ref{divseqidentif} and to this end we need to compute in a 
tractable way the law of $\smash{ (b_{\mathtt{n} (\cE_b \mathbf H) }, y_{\mathtt{n} (\cE_b \mathbf H)})}$ knowing $\smash{ (\cE_a \mathbf H )_{ a\in [b, h]}}$, where $\smash{ (\mathbf H, \bzeta)}$ is a $h$-unicellular Brownian path. 
More specifically, let $\smash{ h \ino (0, \infty)}$, $\smash{ x^+\! , x^-\ino \bbR_+}$ and 
$\smash{ (H, \zeta)\ino \bC_{\! f}}$, a $h$-unicellular function with initial cell $\smash{ (x^+\! , x^-)}$. 
We assume that $\smash{ \lim_{b\downarrow 0} N_b (H)\eqo \infty}$ we fix $b \ino (0, h]$. To simplify, 
we often skip $H$ in various notations below.  
We also recall from (\ref{1erbtps}) and (\ref{snbdef}) the definition of $\smash{ (s^{_b}_{^k})_{1\leq k\leq N_b}}$. 
By Lemma \ref{unicelprop} $(i)$ $\smash{ (H,\zeta, (s^{_b}_{^k})_{1\leq k\leq N_b)})}$ satisfies (\ref{Hascend}) in 
Definition \ref{margtreedef}.  
We then define a cell population $\smash{ \bcc:=((x^+_k, x^-_k))_{1\leq k\leq N_b}}$ as in (\ref{celpopmargtree}) 
in Definition \ref{margtreedef}. Namely, 
\begin{equation}
\label{celpopmargtreebis}
\smash{ \forall k\ino \{ 1, \ldots, N_b\}, \qquad x^+_k \eqo H_{s^b_k}\! -\! m_H(s^b_{k-1}, s^b_k) \quad \textrm{and} \quad x^-_k \eqo H_{s^b_k}\! -\! m_H(s^b_k,s^b_{k+1} )  .}
\end{equation}
with the conventions: $\smash{ s^{b}_{0}\eqo 0}$ and $\smash{  s^{b}_{{N_b+1}}\eqo \zeta}$. Then 
$$ \smash{ \cE_bH\eqo \mathtt{Z} (\bcc) \eqo \phi_{x^+_1} \! \centerdot  \phi_{-x^-_1}\!  \centerdot \ldots \centerdot  \phi_{x^+_{N_b}} \! \centerdot  \phi_{-x^-_{N_b}} \quad \textrm{and} \quad \tmm(\bcc)=M_b =\!\!\!  \sum_{^{1\leq k\leq N_b}} \!\! x^+_k \! + x^-_k \; . }$$
Note that $\bcc$ may not be the minimal cell of $\smash{ \cE_b H}$ (see Definition \ref{cellfromzigzag}). This situation indeed occurs when $\smash{ \tnn(\cE_bH) \leko N_b}$, i.e., when $b$ is a jumptime of $\smash{ N_b}$. 
We also recall the notation $\smash{ \mathtt{Tree} (H)} $ $\eqo$ $\smash{ (\cT_{\! H}, d_H, \rho_H, <_H)}$ which is the real tree coded by 
$\smash{ (H, \zeta)}$ and we denote by $\smash{ \mathtt{p}_H\! : \! [0, \zeta] \! \to \! \cT_{\! H}}$ 
the canonical projection. We next recall from (\ref{oversbk}) the definition of the times $\smash{ (\overline{s}^{_b}_{^k})_{1\leq k\leq N_b}}$ and we set 
\begin{equation}
\label{Tbancill}
\smash{  \mathscr{T}_b\! = \!\!\!\!\! \! \!\!\!\!  \bigcup_{\quad 1\leq k\leq N_b}   \!\!\!\!\! \!  
 \lgeo \rho_{H} , \mathtt{p}_{\! H} (s^{_b}_{^k}) \rgeo   \quad \textrm{and} \quad  T_b \! = \!\!\!\!\! \! \!\!\!\! \bigcup_{\quad 1\leq k\leq N_b}  \!\!\!\!\! \!  \lgeo \rho_{H} , \mathtt{p}_{\! H} (\overline{s}^{_b}_{^k}) \rgeo \; }
\end{equation}

By Lemma \ref{unicelprop} $(i)$, $\smash{ (H,\zeta, (\overline{s}^{_b}_{^k})_{1\leq k\leq N_b})}$ satisfies (\ref{Hascend}) in Definition \ref{margtreedef} and $\smash{ T_b}$ is the $\smash{ (\overline{s}^{_b}_{^k})_{1\leq k\leq N_b}}$-marginal subtree of 
$\smash{ \mathtt{Tree} (H)}$ as 
in Definition \ref{margtreedef} $(a)$. 
Since $H$ is $h$-unicellular, (\ref{eraera}) in Lemma \ref{unicelprop} $(iii)$ asserts that 
$\smash{  \mathtt{Tree} (\cE_b H)\! \equiv \!}$  $\smash{ ( \mathscr T_b, d_H,\rho_H, <_H)}$ which is 
$\smash{ (\bcE_b \mathcal{T}_{\! H}, d_H, \rho_H , <_H) }$.

We then set $\smash{ (Z^b\! , \zeta_b)\eqo \grob_b (\cE_b H)}$. 
Namely, $\smash{  Z^b\! =\phi_{x^+_1} \! \centerdot  \psi_b \centerdot \phi_{-x^-_1}\! 
\centerdot \ldots \centerdot  \phi_{x^+_{N_b}} \! \centerdot  \psi_b \centerdot \phi_{-x^-_{N_b}} }$ 
and $\smash{ \zeta_b \eqo M_b + 2b N_b}$. 
Since the geodesic intervals $\smash{ (\, \rgeo \pcH (s^{_b}_{^k}) ,\pcH (\overline{s}^{_b}_{^k})  \lgeo \,)_{1\leq k\leq N_b}}$ 
are pairwise disjoint in $\smash{ \cT_{\! H}}$, we easily see that 
$$\smash{  \mathtt{Tree} (Z^b) \equiv \big( T_b, d_H, \rho_H, <_H) \; .}$$
Although $\smash{ \cE_bH}$ may have less than $\smash{ N_b}$ strict local maxima, $\smash{ Z^b}$ has $\smash{ N_b}$ strict local maxima, 
$$ \smash{ \tnn (Z^b)=N_b \quad \textrm{and} \quad \mathscr M (Z^b)= \big\{ z_{2k-1}\! + (2k\! -\! 1) b \, ; \, 1\leqo k\leqo N_b \big\}   \; ,} $$
where $\smash{ \overline{\bcc}\eqo (z_j)_{1\leq j\leq 2N_b}}$ stands for the cumulated version of the cell population $\bcc$ given by (\ref{celpopmargtreebis}).

We next denote by $\smash{ \cH^b_\cdot (H)}$ and by $\smash{J^b_\cdot (H)}$ respectively the reflected process and the local time on the subtree $\smash{T_b}$, as defined in Lemma \ref{loctimlem}.  Namely for all $\smash{t\ino [0, \zeta]}$, 
\begin{equation}
\label{loctiremind}
\smash{ \cH^b_t(H)\eqo d_H \big( \pcH(t), \mathtt{proj}_{T_b} (\pcH(t))\big) \quad \textrm{and} \quad H_t \eqo \cH^b_t(H)+ Z^b_{\! J^b_t(H)}\; .}
\end{equation}
Here we recall for all $\smash{\sigma\ino \mathcal{T}_{\! H}}$ that 
$\smash{\mathtt{proj}_{T_b} (\sigma)}$ is the point in $\smash{T_b}$ which stays the closest to $\sigma$.  
We recall from Lemma \ref{loctimlem} $(i)$ 
that $\smash{\cH^b_0(H)\eqo  \cH^b_\zeta(H)\eqo 0}$ and we define the excursion intervals of 
$\smash{ \cH^b_\cdot (H)}$ above $0$ as the following connected components. 
$$\smash{ \bigcup_{^{j\in \cJ_b}} (\alpha_j, \beta_j) =  \big\{ t\ino [0, \zeta] :  \cH^b_t (H) \geko 0 \big\} \; .}$$
We recall from Remark \ref{connTT} that $\smash{\pcH( (\alpha_j, \beta_j))}$, $\smash{j\ino \cJ_b}$, are the connected components of $\smash{\cT_{\! H} \backslash T_b}$. See Figure \ref{reflected}.

\begin{figure}[htb]
\hspace{-0mm}
\centering
\begin{minipage}{0.9\textwidth}
\centering 
\includegraphics[width=1\textwidth, height=0.30\textheight]{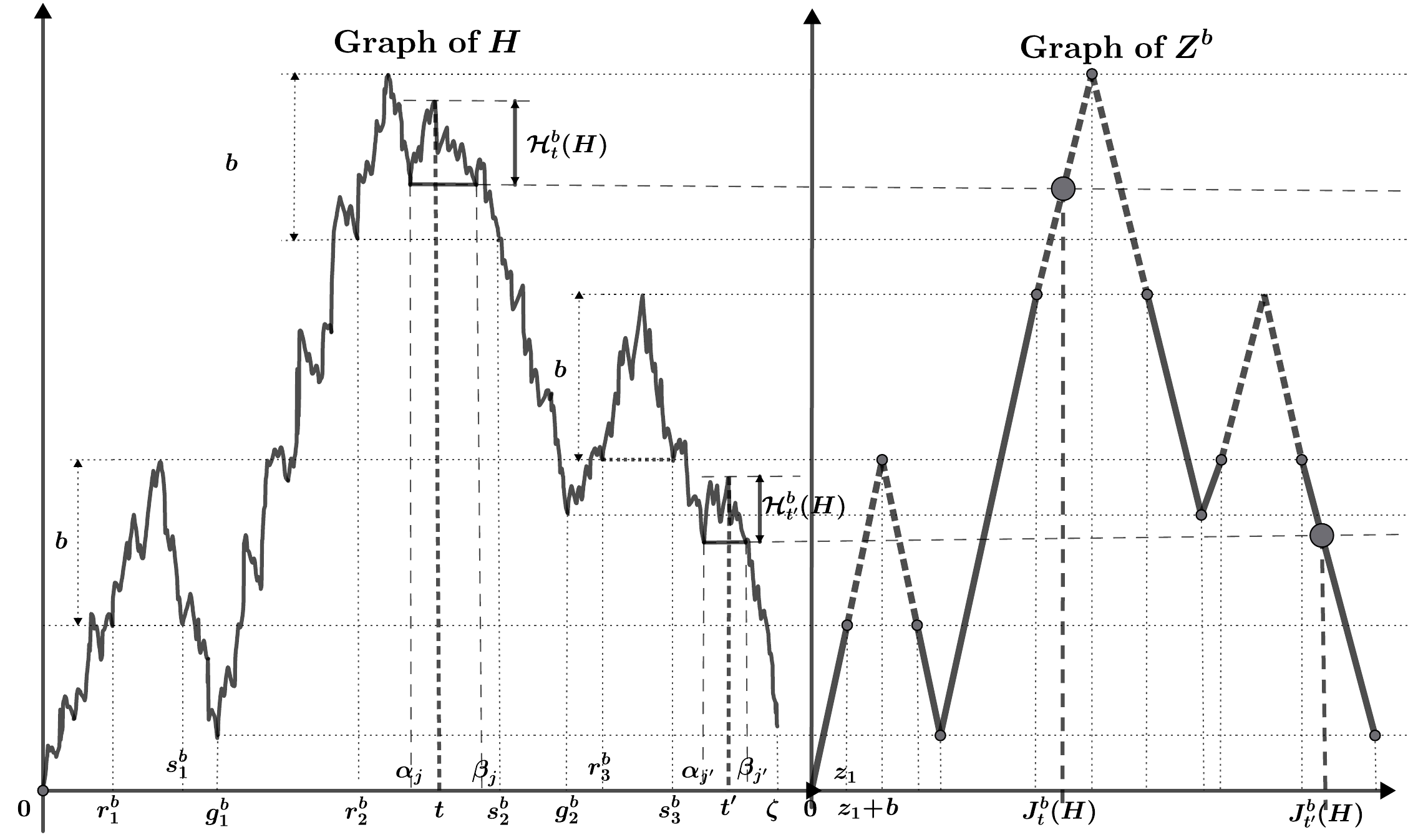}
\caption{{\footnotesize 
\emph{The function $\smash{ (H,\zeta)}$ on the left-hand side has three complete excursions of height $b$, i.e., $\smash{N_b(H) \eqo3}$. 
The process $\smash{\mathcal H^b_\cdot (H)}$ which is the reflected process on the subtree $\smash{T_b}$ as defined by (\ref{Tbancill}) we represent $\smash{\mathcal H^b_t (H)}$ and  $\smash{\mathcal H^b_{t'} (H)}$ for two times $\smash{t,t'\ino [0, \zeta]}$ such that $t\leko t'$. 
 The contour function which encodes $\smash{T_b}$ is $\smash{Z^b}$ whose graph appears on the right-hand side of the figure. The two corresponding local times $\smash{J^b_t(H)}$ and $\smash{J^b_{t'} (H)}$ are shown there. In the left-hand side of the figure, $\smash{(\alpha_j , \beta_j)}$ (resp.~$\smash{(\alpha_{j'}, \beta_{j'})}$) is the excursion interval of $\smash{\mathcal H^b_\cdot (H)}$ above $0$ which contains $t$ (resp.~$\smash{t'}$). 
}}
\label{reflected}}
\end{minipage}
\end{figure}  
\begin{proposition}
\label{nextjump} Let $\smash{h \ino (0, \infty)}$, $\smash{x^+\! , x^-\ino \bbR_+}$. Let $\smash{(H, \zeta)\ino \bC_{\! f}}$ 
be a $h$-unicellular function with initial cell $\smash{(x^+\! , x^-)}$ such that $\smash{\lim_{b\downarrow 0} N_b(H)
\eqo \infty}$. Its erasure sequence, as in Definition \ref{eraseqdef}, is denoted by 
$\smash{(b_n,y_n)_{n\in \bbN^*}}$ and it is convenient to set $\smash{b_0\eqo h}$. We fix $b\ino (0, h]$ and we 
keep the above notations for $\smash{Z^b_\cdot}  $, $\smash{T_b\, }$, $\smash{\cH^b_\cdot (H)}$, 
$\smash{J^b_\cdot (H)}$ and the 
$\smash{(\alpha_j, \beta_j)}$, $\smash{j\ino \cJ_b}$. Then there is a unique pair $\smash{(n,m)\ino (\bbN^*)^2}$ such that 
\begin{equation}
\label{bsitua}
\smash{b_{n+m} < b_{n+m-1}\! = \ldots = b_{n} < b \leq b_{n-1} \; }
\end{equation}
and such that the following holds true: $\smash{n\eqo  N_b(H)}$, $\smash{b_{n}\eqo  \max \big\{ \cH^{_b}_{^t} (H)\, ; \, t\ino [0, \zeta] \big\}}$ and there 
are distinct indices $\smash{j_1, \ldots, j_m \ino \cJ_b}$ such that $\smash{J^{b}_{{\alpha_{j_1}}} (H) \! 
\leq \ldots \leq \! J^{b}_{{\alpha_{j_m}}} (H)}$ and such that for all $\smash{l\ino \{ 1, \ldots, m\}}$ we get 
$$\smash{ \!\!\!\! \max_{^{\; \, t\in [\alpha_{j_l}, \beta_{j_l}]}}  \!\!\!\! \! \cH_{t}^b(H) = b_{n} \quad \textrm{and} \quad y_{n+l-1}\eqo  I\big( b_{n}\,  ,\,  J^{b}_{{\! \alpha_{j_l}}}\!  (H) \big)}$$
where for all $\smash{r\ino [0, b]}$ and all $\smash{z\ino [0,  M_b(H)+ 2bN_b(H)]}$ we have set 
$$ \smash{I(r,z)\eqo \int_0^z\!\! \un_{\! \big\{  \mathtt{dist} ( y \, ; \, \mathscr M (Z^b)) > b-r  \big\}}\,  \mathrm d y  .}$$
\end{proposition}
\noi
\textbf{Proof.} By definition of erasure sequences, 
(\ref{bsitua}) implies $\smash{ N_a\eqo N_b\eqo n}$, for all $\smash{ a\ino (b_{n}, b]}$ and 
Remark \ref{sbnrems} $(b)$ implies that 
$\smash{ r^{_{a}}_{^{k}}\eqo \max \{ t\ino [ r^{_{b}}_{^{k}} ,\overline{s}^{_{b}}_{^{k}} ]: H(t)
 \eqo H (\overline{s}^{_{b}}_{^{k}})\! -\! a \} }$ and $\smash{ s^{_{a}}_{^{k}}\eqo \min \{ t\ino 
 [\overline{s}^{_{b}}_{^{k}} , s^{_{b}}_{^{k}} ]}$ $\smash{ : H(t) \eqo H (\overline{s}^{_{b}}_{^{k}})\! -\! a\} }$. 
 Therefore, we get $\smash{  \overline{s}^{_{a}}_{^{k}}\eqo \overline{s}^{_{b}}_{^{k}}}$. 
 Let us suppose that there is $\smash{ j\ino \cJ_b}$ such that 
 $\smash{ \max \big\{ \cH^{_b}_{^t}\, ; \, t\ino [\alpha_j, \beta_j ] \big\} \geqo a}$ 
 (here to simplify we write $\smash{ \cH^{_b}_{^t}}$ instead of $\smash{ \cH^{_b}_{^t}(H)}$). 
 It would imply that $\smash{ \overline{s}^{_{a}}_{^{k}}\eqo \overline{s}^{_{b}}_{^{k}}\ino (\alpha_j, \beta_j)}$ 
 for some $\smash{ 1\leqo k \leqo N_b}$ which contradicts the very definition of $\smash{ \cH^{_b}_\cdot}$ which requires that 
 $\smash{  \cH^{b} (\overline{s}^{_{b}}_{^{k}} )\eqo 0}$.
Since $a$ can be chosen arbitrarily close to $b_{n}$, it proves 
\begin{equation}
\label{uppbndmaxcH}
\smash{ \max \big\{ \cH^{b}_{t} \, ; \, t\ino [0, \zeta] \big\}\leqo b_{n}\; .}
\end{equation}
By definition of erasure sequences, (\ref{bsitua}) implies $\smash{ N_{b_{n}}\eqo N_b+m}$. 
Argueing as previously entails 
$\smash{ \{\overline{s}^{_{b}}_{^{k}} \, ; \,  1\leqo k \leqo N_b \} \! \subset \! \{\overline{s}^{_{{b_{n}}}}_{^{k'}} \, ; \, 
1\leqo k' \leqo N_b +m\}} $. We then denote by 
$\smash{ k_1 \! \leq}$  $\smash{ \ldots \leq}$  $\smash{ \! k_m}$ the integers 
$\smash{ k'\ino \{ 1}$ $, \ldots$ $\smash{ , m+N_b\}}$ such that  
$\smash{ \overline{s}^{_{{b_{n}}}}_{^{k'}} \! \notin\! \{\overline{s}^{_{b}}_{^{k}} \, ; \,  1\leqo k \leqo N_b \} }$. We fix 
$\smash{ l\ino \{ 1, \ldots, m\}}$. 
Since $\smash{ \pcH (\overline{s}^{_{b_{n}}}_{^{k_l}}) \! \notin \! T_b}$, we see that 
$\smash{ \cH^{b} ({\overline{s}^{_{b_{n}}}_{^{k_l}}}) 
 \geko 0}$ and there is $\smash{ j_l\ino \cJ_b}$ such that $\smash{ \overline{s}^{_{b_{n}}}_{^{k_l}} \ino (\alpha_{j_l}, \beta_{j_l})}$.  We then prove that 
\begin{equation}
\label{alohrcoi}
\smash{ \alpha_{j_l}= r^{b_{n}}_{k_l} \quad \textrm{and} \quad \beta_{j_l}= s^{b_{n}}_{k_l} . }
\end{equation}
By definition, we get $\smash{ r^{_{b_{n}}}_{^{k_l}}\eqo \max \{ t\ino [ 0,  \overline{s}^{_{b_{n}}}_{^{k_l}} ]: 
H(t) \eqo H (\overline{s}^{_{b_{n}}}_{^{k_l}})\! -\! b_{n} \} }$ and $\smash{ s^{_{b_{n}}}_{^{k_l}}\eqo 
\min \{ t\ino [\overline{s}^{_{b_{n}}}_{^{k_l}} ,\zeta ]: H(t)} $  $\eqo$ $\smash{  H (\overline{s}^{_{b}}_{^{k_l}})\! -\! b_{n} \} }$. 
We first observe that (\ref{uppbndmaxcH}) entails $\smash{ b_{n} + H (\alpha_{j_l})  \geqo H (t) \geko H (\alpha_{j_l}) \eqo 
H (\beta_{j_l})}$ for all $\smash{ t\ino (\alpha_{j_l},\beta_{j_l})}$. We thus get 
$\smash{ r^{_{b_{n}}}_{^{k_l}} \leqo \alpha_{j_l}\leko \beta_{j_l} \leqo s^{_{b_{n}}}_{^{k_l}}}$. We 
next observe that if $\smash{ H(\alpha_{j_l})\eqo H(r^{_{b_{n}}}_{^{k_l}})}$, which is equivalent to 
$\smash{ H(\beta_{j_l})\eqo H(s^{_{b_{n}}}_{^{k_l}})}$, then (\ref{alohrcoi}) holds true. 

To complete the proof of (\ref{alohrcoi}), it therefore remains to prove that $\smash{ H(\alpha_{j_l})\eqo H(r^{_{b_{n}}}_{^{k_l}})}$. 
To this end we argue by contradiction by supposing that $\smash{ H(\alpha_{j_l}) \! \neq \! H(r^{_{b_{n}}}_{^{k_l}})}$. 
It implies $\smash{ H(r^{_{b_{n}}}_{^{k_l}}) } $ $\eqo$ $\smash{  H(s^{_{b_{n}}}_{^{k_l}})}$ 
$ \leko$ $\smash{  H(\alpha_{j_l}) } $ $\eqo$ $\smash{  H(\beta_{j_l})}$. Let $\smash{ k\ino \{ 0, \ldots, N_b\}}$ be such that 
$$\smash{  \overline{s}^b_k \leko s^b_k \leqo r^{{b_{n}}}_{{k_l}}  \leqo \alpha_{j_l}\leko 
\beta_{j_l} \leqo  s^{{b_{n}}}_{{k_l}}\leqo r^b_{k+1} \leqo \overline{s}^b_{k+1} ,} $$
with the conventions $\smash{ \overline{s}^{_b}_{^0}\eqo s^{_b}_{^0}\eqo 0}$ and 
$\smash{ \overline{s}^{_b}_{^{N_b+1}}\eqo s^{_b}_{^{N_b+1}}\eqo \zeta}$. 
Consequently, 
$\smash{ m_H(\overline{s}^b_k,  \alpha_{j_l}) \leqo H( r^{{b_{n}}}_{{k_l}}) } $ $\leko$ $\smash{ H(\alpha_{j_l})}$ and 
$\smash{ m_H( \beta_{j_l},  \overline{s}^b_{k+1})\leqo H( s^{{b_{n}}}_{{k_l}})} $  $ \leko$ $\smash{ H( \beta_{j_l})}$. 
We next observe that $\smash{ H( \alpha_{j_l})\eqo H(\beta_{j_l})}$ and similarly that $\smash{ m_H(\overline{s}^b_k,  \alpha_{j_l})} $ $\eqo$  
$\smash{m_H(\overline{s}^b_k,  \beta_{j_l})}$. Therefore,  
$$ \smash{ H (\beta_{j_l}) \geko m_H(\overline{s}^b_k,  \beta_{j_l} ) \vee 
m_H( \beta_{j_l},  \overline{s}^b_{k+1})=Z^b \big(J^b_{\alpha_{j_l}} (H)\big), }$$
by Lemma \ref{loctimlem} $(iii)$. By (\ref{loctiremind}) we get 
$\smash{ \cH^{b} (\beta_{j_l})   \geko 0}$, which contradicts the definition of $\smash{ \beta_{j_l}}$. 
Thus $\smash{ H(\alpha_{j_l})\eqo H(r^{_{b_{n}}}_{^{k_l}})}$ which entails (\ref{alohrcoi}) and also 
$\smash{ b_{n}\eqo \max \{ \cH^{_b}_{^t} (H)\, ; \,  t\ino [\alpha_{j_l}, \beta_{j_l} ] \}}$, $\smash{ 1\leqo l\leqo m}$. 

  Next note that while $\smash{ J^{b}_{{\alpha_{j_l}}} (H)}$ is the position of the $(n+l-1)$th erasure within the 
parametrization of $\smash{ Z^b\eqo\grob_b (\cE_b H)}$ which encodes $\smash{ T_b}$, 
$\smash{ y_{n+l}}$ is the position of the $(n+l-1)$th erasure within the 
parametrization of $\smash{ \cE_{b_{n }} H}$. 
Since $\smash{ \cE_{b_{n }} H\eqo \cE_{b-b_{n}} Z^b}$, by Remark \ref{unicelerarems} $(c)$ applied to $a\eqo b_n$, we get $ \smash{ y_{n+l-1}\eqo I ( b_{n}\,  ,\,  J^{b}_{{\alpha_{j_l}}}\!  (H) )}$, which completes the proof of the proposition.  \cqfd 

\section{DR branching processes driven by unicellular Brownian processes}
\label{randDRsec}
\subsection{Preliminary results}
\label{prelrandsubsec}
In this subsection, which contains no new result, we set some notations and 
we recall standard results on Brownian excursion and Bessel processes. 
Recall from the introduction that $\smash{ (B_t)_{t\in \bbR_+}}$ under $\bP$ stands for a real valued Brownian motion with initial value $\smash{ B_0\eqo 0}$. 
For all $\smash{ t\ino \bbR_+}$ we recall the notations
$\smash{ \underline{B}_t\eqo \min_{s\in [0, t]} B_s}$ and $\smash{ B^+_t \eqo B_t \! -\! \underline{B}_t}$. We also recall from (\ref{NItomeadefi}) the definition of Ito's excursion measure $\bN$ on $\bC$. 
We shall use the following lemma which computes explicitely the Laplace transform of the lifetime of a $h$-unicellular Brownian path. 
\begin{lemma}
\label{durationLapla} Let $\smash{ h\ino (0, \infty)}$ and $\smash{ x^+\! , \, x^-\ino \bbR_+}$. Let $\smash{ (\bH, \bzeta)}$ be a $\bC$-valued process whose law is equal to $\smash{ P_{\! x^+\! , \, x^- \! , \, h}}$, as in Definition \ref{unicelBrodef} $(c)$. Then, for all $\smash{ \lambda\ino \bbR_+}$
\begin{equation}
\label{explidura}
\smash{ \bE  \big[ \mathrm e^{-\lambda \bzeta} \big] = \bigg( \, \frac{h\sqrt{ 2\lambda\, }}{\, \sh (h\sqrt{ 2\lambda\, })} \bigg)^2  \exp \Big( \!\! -(x^+\! +x^-) \Big( \sqrt{2\lambda\, } \mathrm{coth} \big(  h\sqrt{2\lambda\, }\,  \big) \! -\! \tfrac{1}{h} \Big) \Big) .}
\end{equation}
\end{lemma}
\noi
\textbf{Proof.} We recall from (\ref{NItomeadefi}) the notation $\smash{ (\alpha_j, \beta_j)}$ and $ \smash{ (H_j, \zeta_j)}$, $\smash{ j\ino \mathcal J}$, for the excursions of $\smash{ B^+}$ above $0$. We recall that $\bN$ is normalized as in (\ref{normaIto}). Let $x\ino (0, \infty)$. We then set 
$\smash{ V_{h,x}\eqo \# \{ j\ino \mathcal J: } $ $\smash{ -\underline{B}_{\alpha_j} \leqo x \; \mathrm{and} \; \Gamma (H_j) \geko h \}}$. We recall from (\ref{bvarodef}) that $\smash{ \bvaro_{-x}}$ stands for the $(-x)$-hitting time of $B$. We note that $\smash{ \bvaro_{-x}\eqo \sum_{j\in \mathcal J} \zeta_j \un_{[0, x] } (-\underline{B}_{\alpha_j})}$. By 
standard P.p.p.~computations, we get 
\begin{eqnarray*}
\smash{ \bE\Big[ \mathrm{e}^{-\lambda \bvaro_{-x} }\un_{\{\max_{t\in  [0, \bvaro_{-x} ]} \! B^+_t \leq h \}} \Big] }&\eqo &\smash{ \lim_{^{y\to \infty}} \bE\big[ \mathrm{e}^{-\lambda \bvaro_{-x} -yV_{h,x}} \big]\eqo  \lim_{^{y\to \infty}} \exp \big( \! -\! x \bN \big[1\! -\! \mathrm{e}^{-\lambda \zeta -y \un_{\{ \Gamma > h\} }} \big] \big)} \\
& \eqo & \smash{ \exp \big(\!  -\! x \bN \big[  \big( 1\! -\! \mathrm{e}^{-\lambda \zeta } \big) \un_{\{ \Gamma \leq h\}}\big] - \tfrac{x}{h} \, \big),}
\end{eqnarray*}
for all $\smash{ \lambda, x\ino \bbR_+}$. Therefore 
\begin{equation}
\label{explidescentdur}
\smash{\int_{\bC} \!\!\! P^{\downarrow}_{\! x, h} (\mathrm d H) e^{-\lambda \zeta_H}
  \eqo \bE\Big[ \mathrm{e}^{-\lambda \bvaro_{-x} } \Big| \max_{t\in  [0, \bvaro_{-x} ]} \! B^+_t \leq h \Big] 
= \exp \big(\!  -\! x \bN \big[  \big( 1\! -\! \mathrm{e}^{-\lambda \zeta } \big) \un_{\{ \Gamma \leq h\}}\big]  \big).}
\end{equation}
By time-reversal, we also get $\smash{\int_{\bC} \! P^{\uparrow}_{\! x, h} (\mathrm d H) e^{-\lambda \zeta_H}
\eqo \int_{\bC} \! P^{\downarrow}_{\! x, h} (\mathrm d H) e^{-\lambda \zeta_H}}$. 
Next recall, as an easy consequence of Williams decomposition (see e.g.~Borodin \& Salminen \cite{BorSal15}, IV.17 p.~61), we get 
\begin{equation}
\label{williamcsq}
\smash{\bN \Big[ \mathrm{e}^{-\lambda \zeta}  \Big| \, \Gamma \eqo h\Big] \eqo  \bigg(  \frac{h\sqrt{ 2\lambda\, }}{\, \sh (h\sqrt{ 2\lambda\, })} \bigg)^{\! 2} \; \textrm{and} \; \bN \Big[ \big(1 \! -\! \mathrm{e}^{-\lambda \zeta} \big) \un_{\{ \Gamma \leq h\}} \Big] \eqo  \sqrt{ 2\lambda\, } \mathrm{coth} \big( h\sqrt{ 2\lambda\, }\big) \! -\! \frac{1}{h} }
\end{equation}
(the second equality is derived from the first one thanks to elementary computation). The definition of $P_{\! x^+\! , \, x^- \! , \, h}$, combined with (\ref{explidescentdur}) and (\ref{williamcsq}) implies (\ref{explidura}). \cqfd

\smallskip

  We next recall several standard results on three dimensional Bessel processes. Let $\smash{(R_t)_{t\in \bbR_+}}$ be such a process starting at $\smash{R_0\eqo 0}$, which is an entrance boundary point. The first one is Pitman's theorem (see e.g.~Revuz and Yor \cite{RevYor04}, Theorem 3.5 p.253), which asserts the following.  
 \begin{equation}
\label{PitmanTh}
\smash{\big( B^+_t , -\underline{B}_t \big)_{\! t\in \bbR_+} \overset{\textrm{law}}{=} \big( R_t \! -\! \tJJ_t, \tJJ_t \big)_{\! t\in \bbR_+} \quad \textrm{where} \quad \tJJ_t= \min_{r\in [t, \infty)} R_t \; , \quad t\ino \bbR_+.}
\end{equation}  
 Here, $\smash{(\tJJ_t)_{t\in \bbR_+}}$ is the \emph{future infimum process}, which is a continuous nondecreasing process such that a.s.~$\smash{\tJJ_0\eqo 0}$, $\smash{\tJJ_t  \geko 0}$ for all $t\ino (0, \infty)$, and $\smash{\lim_{t\to \infty} \tJJ_t\eqo \infty}$. 

For all $h\ino (0, \infty)$, we next set $\smash{\lambda_h\eqo \max \{ t\ino \bbR_+: R_t\eqo h  \}}$ which is a.s.~finite and equal to $\smash{ \max \{ t\ino \bbR_+: \tJJ_t\eqo h  \}}$. Recall from (\ref{bvarodef}) that $\smash{\bvaro_{-h}}$ stands for $\smash{\varrho_{-h} (B)}$. 
Since Brownian motion has homogeneous and independent increments Pitman's theorem (\ref{PitmanTh}) entails 
\begin{equation}
\label{Pitconsq1}
\smash{\textrm{$R_{\, \cdot \wedge \lambda_h}$ and $R_{\lambda_h +\,  \cdot}$ are independent,}  \; R_{(\lambda_h - \, \cdot)_+}\! -\! h   \overset{\textrm{law}}{=} B_{\, \cdot \wedge \bvaro_{-h}} 
\;  \textrm{and} \; R_{\lambda_h + \, \cdot}\! -h  \overset{\textrm{law}}{=} R_{\, \cdot} }
\end{equation}
\begin{remark}
\label{decBes} Let $\smash{x, b\ino (0, \infty)}$. We recall from (\ref{hittim}) the notation $\smash{\varrho_{x+b} (R)}$. To simplify the notation, we 
set $\smash{\ell \eqo} $ $\smash{ \max \{ t\ino [0, \varrho_{x+b} (R)] : R_t\eqo x\}}$. For all $\smash{y\ino \bbR}$, we recall that use the notation $\smash{\bvaro_{y}}$ for $\smash{\varrho_{y} (B)}$. Then we easily deduce from (\ref{Pitconsq1}) the following. 
\begin{compactenum}

\smallskip

\item[$(i)$] $\smash{\big(R_{\,\cdot  \wedge \ell}, \ell)}$ and $\smash{(R_{(\ell +\, \cdot )\wedge \varrho_{x+b} (R)} , \varrho_{x+b} (R) \! -\! \ell \big)}$ are independent.

\smallskip

\item[$(ii)$] $\smash{\big( R_{\, \cdot \wedge \ell}, \ell \big)  \overset{\textrm{law}}{=} \big(R_{\, \cdot \wedge \lambda_x}, \lambda_x \big)}$ under $\smash{\bP \big( \cdot  \big| \max_{t\in [0, \lambda_x] } R_t \leko x+b \big)}$.

\smallskip

\item[$(iii)$] By (\ref{Pitconsq1}), $(ii)$ translates into $\smash{(R_{(\ell -\, \cdot\, )_+} \!\!\!  -\! x\, , \ell)   \overset{\textrm{law}}{=} (B_{\cdot \wedge \bvaro_{-x}} ,  \bvaro_{-x})}$ under $\smash{\bP (\, \cdot \,  | \, \bvaro_{-x}\! \leko \bvaro_{b} )}$.

\smallskip

\item[$(iv)$] $\smash{ \big( R_{(\ell +\, \cdot )\wedge \varrho_{x+b} (R)} , \varrho_{x+b} (R) \! -\! \ell \big) \overset{\textrm{law}}{=} \big( R_{\, \cdot \wedge \varrho_{b} (R)} , \varrho_{b} (R) \big)}$. \cq

\end{compactenum}
\end{remark}
We deduce from Remark \ref{decBes} the following decomposition of $\smash{R_{\, \cdot \wedge \varrho_h(R)}}$ into excursions above its future infimum. For all $\smash{t\ino [0, \varrho_h(R)]}$ we set 
$$\smash{ \tJJ_t^h\eqo m_R (t, \varrho_h(R)) \eqo \min  \! \big\{ \! R_s\,  ; \, s\ino [ t,\varrho_h(R)] \big\} 
 \quad \textrm{and} \quad  \cH_t= R_t - \tJJ_t^h .}$$
We define the excursion intervals $\smash{(\alpha_j, \beta_j)}$, $\smash{j\ino \cJ}$, of $\cH$ above $0$ as the connected components 
$$\smash{ \bigcup_{j\in \cJ} (\alpha_j, \beta_j) \eqo \big\{ t\ino [0, \varrho_h(R)]: \cH_t \geko 0 \big\} \; }$$ 
and the corresponding excursions by
$\smash{\big(\cH_j( \cdot), \zeta_j \big)\! :=\!  \big(R_{(\alpha_j + \, \cdot)\wedge \beta_j} \! -\tJJ_{\alpha_j}^h \, , \, \beta_j \! -\! \alpha_j\big)}$, $\smash{j\ino \cJ}$. Then 
\begin{equation}
\label{Bessdec1}
\smash{\textrm{$ \Big\{ \big( \tJJ_{\alpha_j}^h \, , (\cH_j, \zeta_j)\big)\,  ; \, j\ino \mathcal J \Big\}$ is a Poisson point process on $[0, h] \! \times \! \bC$} }
\end{equation}
with intensity $\smash{\un_{[0, h]} (y) \mathrm dy\,  \bN (\mathrm d H\, ; \Gamma \leko y)}$, 
where $\smash{\int_{\bC} F(H) \bN (dH\, ; \Gamma \leko y) }$ means $\smash{\bN \big[ F(\cdot ) \un_{\{ \Gamma < y\}} \big]}$ for all bounded and measurable $\smash{F \! :\! \bC \! \to \! \bbR}$.      

We next recall William's decomposition of $\bN$ as follows. To this end we first recall from Definition \ref{condheightIto} the conditional law $\smash{\bN (\, \cdot \, | \, \Gamma \eqo h)}$. 
Let $\smash{(R', \zeta')}$ and $\smash{(R'' \! , \zeta'')}$ be two $\smash{\bC}$-valued processes such that 
$\smash{(R', \zeta')}$ and $\smash{(R''_{(\zeta''- \, \cdot)_+}\!\!\!  -  R''_{\zeta''} \, ,\zeta'')}$ are two independent copies of ($\smash{R_{\, \cdot \wedge \varrho_h(R)}, \varrho_h(R))}$. Then Williams' decomposition asserts that 
\begin{equation}
\label{Williams}
\smash{\big( R' \! \centerdot R'' , \, \zeta' \! +\zeta''\big)  \overset{\textrm{law}}{=} \bN (\, \cdot \, | \, \Gamma \eqo h). }
\end{equation}
Namely, the concatenation of $\smash{R'}$ and $\smash{R''}$ has law $\smash{\bN (\, \cdot \, | \, \Gamma \eqo h)}$. 
We shall need to view $\smash{\bN (\, \cdot \, | \, \Gamma \eqo h)}$ as the law of a specific function of $H$ under $\smash{\bN (\, dH \,  | \, \Gamma \geqo h')}$ for any 
$\smash{h'\ino [h, \infty)}$. To this end we introduce the following functional on $\smash{\bC}$. 
\begin{definition}
(\emph{Topping of functions}) 
\label{toppdef} Let $\smash{(H, \zeta)\ino \bC}$ and $h\ino (0, \infty)$ satisfy the following 
\begin{equation}
\label{conditopp}
\smash{\zeta \leko \infty \quad \textrm{and} \quad \max_{t \in [0, \zeta]} H_t \geqo h+ (H_\zeta)_+ \; .}
\end{equation}
We set $\smash{\overline{s} (H)\eqo \min \{ t\ino [0, \zeta] : H_t\eqo \max_{s\in [0, \zeta]} H_s \}}$. Then (\ref{conditopp}) allows for  defining: 
$$\smash{ \gamma_h (H) \eqo \max \big\{ t \ino [0, \overline{s} (H)] : H_t \eqo H_{\overline{s} (H)} -h \big\} \; \,  \textrm{and}  \; \,  \delta_h (H) \eqo \min \big\{ t \ino [\overline{s} (H), \zeta] : H_t \eqo H_{\overline{s} (H)} -h \big\} .}$$
We then define the \emph{$h$-topping of $H$} by $\smash{ \mathtt{Top}_h(H)\! := \!  \big( H_{( \gamma_h (H) + \, \cdot )\wedge \delta_h (H)  } -H_{\gamma_h (H) }\,  , \delta_h (H) \! -\! \gamma_h (H) \big)}$. 
If $\smash{(H, \zeta)}$ does not satisfy (\ref{conditopp}) then we simply set $\smash{\mathtt{Top}_h(H, \zeta)\eqo  \partial}$. \cq
\end{definition}
\begin{remark}
\label{eravstop} For all $b\ino (0, h)$, we easily check that $\smash{\cE_b (\mathtt{Top}_h (H))\eqo \mathtt{Top}_{h-b} (\cE_bH)}$. \cq  
\end{remark}
\begin{lemma}
\label{measuTop} Let $h\ino (0, \infty)$. Then the following holds true. 
\begin{compactenum}

\smallskip

\item[$(i)$] $\smash{(H,\zeta) \ino \bC \! \mapsto \! \big( \gamma_h(H) , \overline{s} (H), \delta_h (H) \big) \ino \bbR_+^3}$ and $\smash{(H,\zeta) \ino \bC \! \mapsto \! \mathtt{Top}_h(H)\ino \bC}$ are measurable.

\smallskip

\item[$(ii)$] Let $\smash{(H,\zeta)\ino \bC}$ satisfy (\ref{conditopp}). We assume 
for all $\varepsilon \ino (0, \infty)$ that  
\begin{equation} 
\label{critcontitop}
H_{\gamma_h} \! \geko m_H \big( 
\gamma_h  \! - \varepsilon , \gamma_h \big), \; \, H_{\delta_h} \! \geko m_H \big( \delta_h , \delta_h + \varepsilon  \big)  \; \textrm{and} \; \forall t \ino [0, \zeta] \; \,  \big( H_t\eqo H_{\overline{s}} \big) \! \Longrightarrow \!  \big( t\eqo \overline{s}  \big).
\end{equation} 
Then $\smash{(h',(H', \zeta')) \! \mapsto \mathtt{Top}_{h'} (H') }$ is continuous at $\smash{(h, (H,\zeta))}$.  
\end{compactenum}
\end{lemma}
\noi
\textbf{Proof.} We leave it to the reader. \cqfd 
\begin{lemma}
\label{Topmore} Let $\smash{x,b\ino (0, \infty)}$. Let $\smash{\gamma_b}$, $\smash{\overline{s}}$ and $\smash{\delta_b}$ be as in Definition \ref{toppdef}.  
Under $\smash{\bN (dH | \, \Gamma \eqo x+b)}$ we set 
$\smash{(H^\uparrow \! , \zeta^\uparrow) \! =\!  
(H_{(\gamma_b-\cdot)_+}\! \! -\! x, \gamma_b)}$ and $\smash{(H^\downarrow \! ,\zeta^\downarrow)\! =\!  \big(H_{(\delta_b + \cdot ) \wedge \zeta} \! -\! x, \zeta \! -\! \delta_b \big)}$. Then the following holds. 

\begin{compactenum}

\smallskip

\item[$(i)$] $\smash{(H^\uparrow, \zeta^\uparrow) }$, $\smash{\mathtt{Top}_b (H)}$ and $\smash{(H^\downarrow ,\zeta^\downarrow)}$
are independent. 

\smallskip

\item[$(ii)$] $\smash{(H^\uparrow, \zeta^\uparrow)} $ and  $\smash{(H^\downarrow ,\zeta^\downarrow)}$ are distributed as $\smash{(B_{\cdot \wedge \bvaro_{-x}}, \bvaro_{-x})}$ under $\smash{\bP ( \, \cdot  \, | \, \bvaro_{-x} \leko \bvaro_{b})}$.

\smallskip

\item[$(iii)$] The law of $\smash{\, \mathtt{Top}_b (H)}$ is 
$\smash{\bN (dH | \, \Gamma \eqo b)}$.

\smallskip

\end{compactenum}

\noi
In particular for all $h \ino [b, \infty)$ the law of $\smash{\mathtt{Top}_b (H)}$ under 
$\smash{\bN (dH | \, \Gamma \geqo h)}$ is $\smash{\bN (dH | \, \Gamma \eqo b)}$.
\end{lemma}
\noi
\textbf{Proof.} We work under $\smash{\bN (dH | \, \Gamma \eqo x+b)}$. 
To simplify the notation, we omit the lifetime in what follows. 
We set $\smash{\frak h^{ \uparrow}\!  \eqo H_{(\gamma_b + \cdot ) \wedge \overline{s}} \! -\! x}$ 
and $\smash{\frak h^{ \downarrow} \! \eqo H_{(\overline{s} + \cdot ) \wedge \delta_b} \! -\! x \! -\! b}$. 
Thus $\smash{\mathtt{Top}_b (H)\eqo \frak h^\uparrow \! \centerdot \frak h^\downarrow}$. For all 
$\smash{(H', \zeta')\ino \bC_{\! f}}$ we use the notation 
$\smash{\overleftarrow{H}'}$ for the time-reversed function $\smash{H' ((\zeta'\! - \cdot )_+) \! -\! H'(\zeta')}$. 
Williams' decomposition (\ref{Williams}) then asserts that 
$\smash{\overleftarrow{H}^\uparrow\!  \centerdot \frak h^\uparrow}$ and 
$\smash{\overleftarrow{H}^\downarrow  \! \centerdot \! \overleftarrow{\frak{h}}^{\! \downarrow} }$ 
are two independent copies of $\smash{R_{\cdot \wedge \varrho_{x+b} (H)}}$. Observe that 
$\smash{(H^\uparrow, \frak h^{ \uparrow})}$ (resp.~$\smash{(H^\downarrow,  \frak h^{ \downarrow})}$) 
is a measurable function of 
$\smash{\overleftarrow{H}^\uparrow\!  \centerdot \frak h^\uparrow}$ (resp.~of 
$\smash{\overleftarrow{H}^\downarrow  \! \centerdot \! \overleftarrow{\frak{h}}^{\! \downarrow} }$). 
Thus $\smash{(H^\uparrow, \frak h^{ \uparrow})}$ and $\smash{(H^\downarrow,  \frak h^{ \downarrow})}$ are independent.
By Remarks \ref{decBes} $(i)$ we easily see 
that $\smash{H^\uparrow}$ and $\smash{\frak h^\uparrow}$ are independent. 
Thus, on one hand that $\smash{H^\uparrow}$, $\smash{\frak h^{ \uparrow}}$, $\smash{H^\downarrow}$ and 
$\smash{\frak h^{ \downarrow}}$ are independent, which implies $(i)$. On the other hand 
$\smash{H^\uparrow}$ and $\smash{H^\downarrow}$ have the same law and  
$\smash{\frak h^{ \uparrow}}$ and $\smash{\overleftarrow{\frak{h}}^{\! \downarrow}}$ have the same law too. 
Then Remarks \ref{decBes} $(iii)$ shows that $\smash{H^\uparrow}$ 
(and thus $\smash{H^\downarrow}$ too) is distributed as $\smash{(B_{\cdot \wedge \bvaro_{-x}}, \bvaro_{-x})}$ 
under $\smash{\bP ( \, \cdot  \, | \, \bvaro_{-x} \leko \bvaro_{b})}$. This proves $(ii)$. Finally 
Remarks \ref{decBes} $(iv)$ shows that  $\smash{\frak h^\uparrow}$ (and thus 
$\smash{\overleftarrow{\frak{h}}^{\! \downarrow}} $ too) has the same law as 
$\smash{R_{\cdot \wedge \varrho_b(H)}}$. Since 
$\smash{\mathtt{Top}_b (H)\eqo \frak h^\uparrow \! \centerdot \frak h^\downarrow}$, 
Williams' decomposition (\ref{Williams}) entails $(iii)$. 
The last statement of the lemma is easily derived from $(i)$, $(ii)$ and $(iii)$. \cqfd

\smallskip

We complete this subsection by providing basic results 
on the two zigzag processes 
$\smash{(Y_{b,z})_{z\in \bbR_+}}$ and $\smash{(Z_{b,z})_{z\in \bbR_+}}$ which are respectively defined in (\ref{defYb}) and (\ref{defZb}) in the introduction.  
Recall that $\smash{Y_{b,\cdot}}$ has an infinite lifetime and that $\smash{Z_{b, \cdot}}$ has a finite lifetime denoted by $\smash{\zeta_b}$. 
The processs $\smash{Z_{b,\cdot }}$ is actually the first excursion of $\smash{Y_{b,\cdot}}$ above its infimum. 
More precisely, we 
recall that $\smash{\underline{Y\! }_{\, b,z}\eqo \min_{y\in [0, z]} Y_{b,y}}$, for all $\smash{z\ino \bbR_+}$ and we denote by 
$\smash{(\mathbf a_p, \mathbf b_p)}$, $\smash{p\ino \bbN^*}$, the excursion invervals of $\smash{Y_{b, \cdot}}$ above its infimum process $\smash{\underline{Y\! }_{\, b, \cdot}}$, which are the connected components 
$$\smash{ \bigcup_{^{p\in \bbN^*}} (\mathbf a_p, \mathbf b_p) \eqo \big\{ z\ino \bbR_+ : Y_{b, z} \geko  \underline{Y}_{b,z} \big\} .}$$
For each $\smash{p\ino \bbN^*}$, the corresponding excursion 
is $\smash{(\mathcal Z_p^b(\cdot ) , \zeta^b_p)\! :=\! \big( 
Y_{b, (\mathbf a_p + \, \cdot )\wedge \mathbf b_p}  \!\!- \! \underline{Y\! }_{\, b,\mathbf a_p} \, , \mathbf b_p \! -\! \mathbf a_p
\big)} $. 
It is convenient to set $\smash{\mathbf a_0\eqo \mathbf b_0\eqo 0}$. Then 
the renewal property of Poisson processes combined with elementary arguments entail the following. 
\begin{compactenum}

\smallskip

\item[$(a)$] The r.v.s $\smash{\underline{Y\! }_{\, b,\mathbf a_{p-1}} \!\! - \underline{Y\! }_{\, b,\mathbf a_p}}$ and $\smash{(\mathcal Z_p^b(\cdot ) , \zeta^b_p)}$, $\smash{p\ino \bbN^*}$, are independent.

\smallskip

\item[$(b)$] $\smash{\underline{Y\! }_{\, b,\mathbf a_{p-1}} \!\! - \underline{Y\! }_{\, b,\mathbf a_p}}$ is exponentially distributed with mean $b$.

\smallskip

\item[$(c)$]  $\smash{(\mathcal Z_p^b(\cdot ) , \zeta^b_p)}$ has the same law as $\smash{(Z_{b,\cdot}, \zeta_b)}$. 

\smallskip

\end{compactenum}
\begin{equation}
\label{Zigzagdec}
\smash{\textrm{In other words, $ \Big\{ \big(  \!  -\!  \underline{Y\! }_{\, b,\mathbf a_p} \, , (\mathcal Z_p^b , \zeta^b_p)\big)\,  ; \, p\ino \bbN^* \Big\}$ \emph{is a Poisson point process on} $\bbR_+ \! \times \! \bC$} }
\end{equation}
\emph{with intensity }$\smash{b\un_{\bbR_+} (y) \mathrm dy\,  Q^b (\mathrm d H)}$, \emph{where} $\smash{Q^b}$ \emph{stands for the law of} 
$\smash{(Z_{b,\cdot}, \zeta_b)}$. 

\subsection{Decomposition of Brownian motions into unicellular paths}
\label{brodecsubsec}
Let $\smash{ x, x^+\! ,\,  x^-  \ino \bbR_+}$ and let $h\ino (0, \infty)$. 
We recall from Definition \ref{unicelBrodef} that $\smash{ P^{_\downarrow}_{^{\! x,h}}}$
(resp.~$\smash{ P^{_\uparrow}_{^{\! x,h}}}$) is the law of the $x$-Brownian descent (resp.~rise) with amplitude $< \! h$. We also recall from the same definition that 
$\smash{ P_{x^+ \! ,\, x^-\! ,\, h}}$ stands for the law of the $h$-unicellular Brownian process with initial cell $\smash{ (x^+\! , x^-)}$. 
\begin{remark}
\label{contiunicelBrorem} Note that for a given $x\ino (0, \infty)$, the Brownian motion $\smash{ B_\cdot}$ satisfies the continuity criterion (\ref{hittcontcrit}) in Lemma \ref{regprophit} $(ii)$ for the hitting times $\smash{ \varrho_x(B)}$. This lemma and elementary arguments imply that $\smash{ x \ino \bbR_+ \! \mapsto \!  P^{_\downarrow}_{^{\! x,h}} \ino \cM_1 (\bC)}$ is continuous (here the space of Borel probability measures $\smash{ \cM_1 (\bC)}$ on the Polish space $\smash{ \bC}$ is equipped with the topology of weak convergence). Similarly, $\smash{ x\! \mapsto \! P^{_\uparrow}_{^{\! x,h}}}$ and $\smash{  (x^+, x^-) \! \mapsto \! P_{x^+ \! ,\, x^-\! ,\, h}}$ are continuous too (we recall Remark \ref{remconca} $(c)$ which asserts that concatenation is continuous). In particular 
we get $\smash{ \lim_{x\to 0} P^{_\downarrow}_{^{\! x,h}}\eqo \lim_{x\to 0}  P^{_\uparrow}_{^{\! x,h}}\eqo \delta_\partial}$ and 
$\smash{ \lim_{(x^+\! , \, x^-) \to (0,0)} P_{x^+ \! ,\, x^-\! ,\, h}\eqo \bN (\, \cdot \, | \, \Gamma \eqo h)}$. 
\cq 
\end{remark}

The goal of this section is to compute for any $b\ino (0, h]$ the joint law 
of $\smash{ (\cE_b\bH, \bH)}$ where $\smash{ (\bH, \zeta)}$ has law $\smash{ P_{x^+ \! ,\, x^-\! ,\, h}}$. Our first step consists in computing the joint law of $\smash{ (\cE_b B , B)}$. 
We recall from  (\ref{1erbtps}) the definition of $\smash{ s^{_b}_{^1}}$, from (\ref{r1bHdef}), (\ref{g0bdef}) and (\ref{snbdef}) the definitions of $\smash{ g^{_b}_{^k}}$, $\smash{ r^{_b}_{^k}}$ and $\smash{ s^{_b}_{^k}}$, $k\ino \bbN$, with the convention that 
$\smash{ r^{_b}_{^0}\eqo s^{_b}_{^0}\eqo 0}$. 
To simplify notations we introduce the following. 
\begin{notation}
\label{notaeraBM}
For all $\smash{ k\ino \bbN^*}$ we set the following. 
\begin{compactenum}

\smallskip

\item[$-$] $\smash{ \tgg^{_b}_{^{k-1}}\!  \eqo  g^{_b}_{^{k-1}} (B)}$,  
$\smash{ \trr^{_b}_{^{k}} \eqo  r^{_b}_{^{k}} (B)}$, and $\smash{ \tss^{_b}_{^{k}} \eqo  s^{_b}_{^{k}} (B)}$, which are a.s.~finite since a.s.~$\smash{ \limsup_{t\to \infty} B_t \eqo}$  $\smash{-\liminf_{t\to \infty} B_t\eqo \infty}$. Note that a.s.~$\smash{ 0\eqo \trr^{_b}_{^{0}}\eqo \tss^{_b}_{^{0}}\leko \tgg^{_b}_{^{0}}}$ and $\smash{ \tgg^{_b}_{^{k-1}} \leko \trr^{_b}_{^{k}}\leko \tss^{_b}_{^{k}}}$ for all $\smash{ k\ino \bbN^*}$.  

\smallskip

\item[$-$] $\smash{ \tHH^{_\downarrow}_{^{k-1}} (\cdot) \eqo B \big( (\tss^{_b}_{^{k-1}} + \, \cdot \, )\wedge \tgg^{_b}_{^{k-1}}  \big) \! -\! B(\tss^{_b}_{^{k-1}} )}$ and $\smash{ \txx^{_-}_{^{k-1}} (b)\eqo  B(\tss^{_b}_{^{k-1}} ) \! -\! B(\tgg^{_b}_{^{k-1}} )}$. 

\smallskip

\item[$-$] $\smash{  \tHH^{_\uparrow}_{^{k}} (\cdot) \eqo B \big( (\tgg^{_b}_{^{k-1}} + \, \cdot\, )\wedge \trr^{_b}_{^{k}}  \big) \! -\! B(\tgg^{_b}_{^{k-1}} )}$ and $\smash{ \txx^{_+}_{^{k}} (b)\eqo  B(\trr^{_b}_{^{k}} ) \! -\! B(\tgg^{_b}_{^{k-1}} )}$. 

\smallskip

\item[$-$] $\smash{ \tHH^{_{\mathtt{ex}}}_{^k} (\cdot)\eqo B \big( (\trr^{_b}_{^{k}} + \, \cdot\, )\wedge \tss^{_b}_{^{k}}  \big) \! -\! B(\trr^{_b}_{^{k}} )}$. 

\smallskip

\end{compactenum}
See Figure \ref{ErazBM}. \cq  
\end{notation}

\begin{figure}[htb]
\hspace{-0mm}
\centering
\begin{minipage}{0.9\textwidth}
\centering 
\includegraphics[width=1\textwidth, height=0.3\textheight]{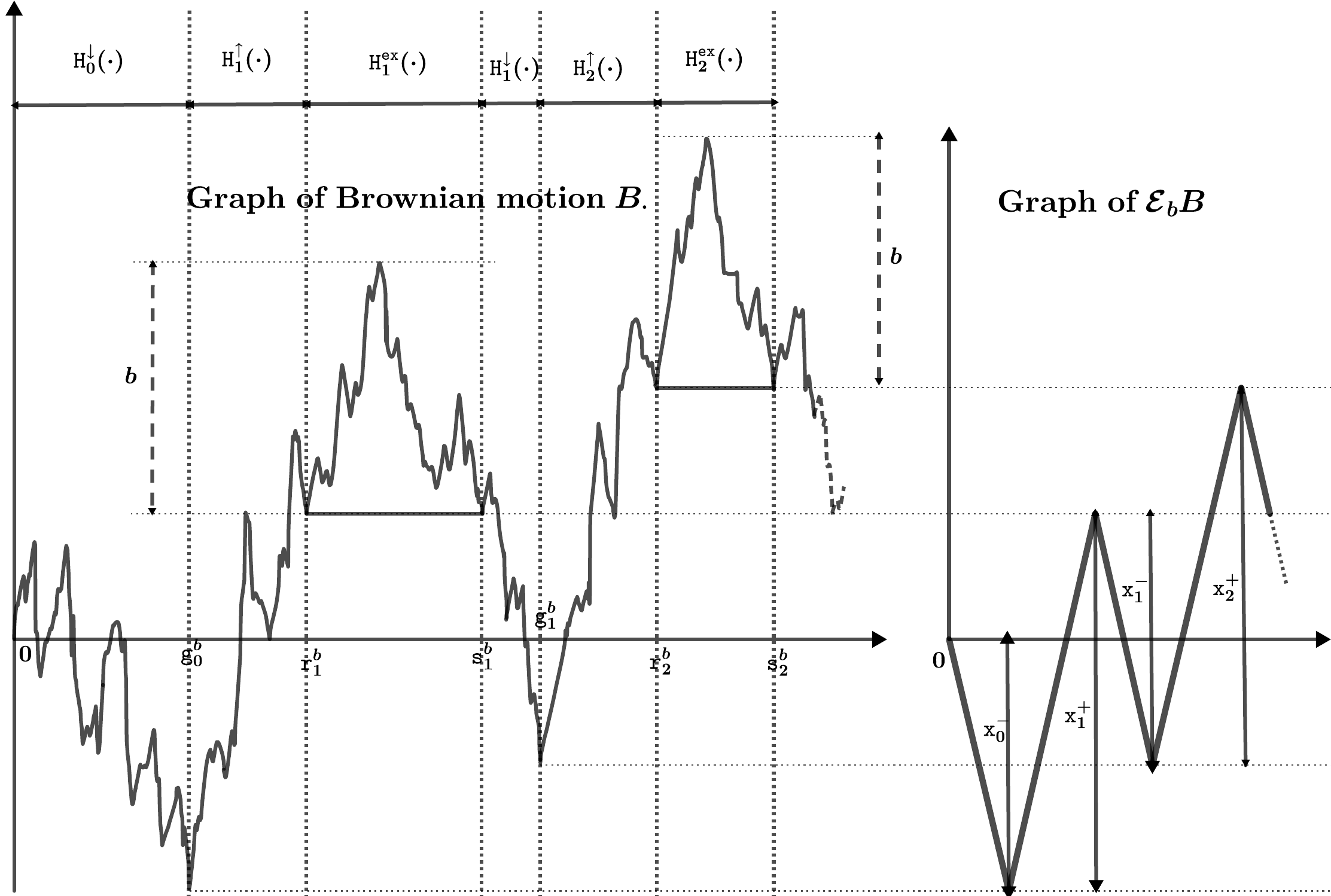}
\caption{{\footnotesize 
\emph{The left-hand side of the figure represents a linear Brownian motion with initial value $\smash{B_0\eqo 0}$ and its first two complete excursions of height $b$. The left-hand side represents the graph of the corresponding $b$-erased process 
$\smash{\mathcal E_b B}$. }\cq
}
\label{ErazBM}}
\end{minipage}
\end{figure}

\begin{theorem}
\label{Brodecunicelth} Let $b\ino (0, \infty)$. We keep the previous notations. Then the following holds true. 
\begin{compactenum}

\smallskip

\item[$(i)$] The r.v.s $\smash{\big( \tHH^{_\downarrow}_{^{k-1}}, \txx^{_-}_{^{k-1}} (b),  \tHH^{_\uparrow}_{^{k}}, \txx^{_+}_{^{k}} (b),  \tHH^{_{\mathtt{ex}}}_{^k}\big)}$, $\smash{k\ino \bbN^*}$, are i.i.d.

\smallskip

\item[$(ii)$] For $\smash{k\ino \bbN^*}$ fixed, $\smash{\tHH^{_{\mathtt{ex}}}_{^k}}$, 
$\smash{( \tHH^{_\downarrow}_{^{k-1}}, \txx^{_-}_{^{k-1}} (b))}$  and $\smash{(\tHH^{_\uparrow}_{^{k}}, \txx^{_+}_{^{k}} (b))}$ are independent. The process $\smash{\tHH^{_{\mathtt{ex}}}_{^k}}$ has law $\smash{\bN (\, \cdot \, | \, \Gamma \eqo b)}$ and for all bounded and measurable 
$\smash{F\! : \! \bC \ \! \times \! \bbR_+ \! \to \! \bbR}$,    
\end{compactenum}
\begin{eqnarray} \label{condilawBrodec0}
\bE \big[ F (\tHH^{_\downarrow}_{^{k-1}}, \txx^{_-}_{^{k-1}} (b))\big] \!\!\!\! &=&  \!\!\!\!\! \int_0^\infty \!\! \!\!  \tfrac{\mathrm d x}{b}\, \ee^{-\frac{x}{b}} \! \int_{\bC} \!\!  
P^{_\downarrow}_{^{\! x,b}} (\mathrm d H) \, F(H,x) \quad \textrm{and}  \\
 \bE \big[ F (\tHH^{_\uparrow}_{^{k}}, \txx^{_+}_{^{k}} (b))\big] \!\!\!\! &=&  \!\!\!\!\! \int_0^\infty \!\!\!\!  \tfrac{\mathrm d x}{b}\, \ee^{-\frac{x}{b}}\!  \int_{\bC} \!\!  
P^{_\uparrow}_{^{\! x,b}} (\mathrm d H) \, F(H,x) \; .
\label{condilawBrodec}
\end{eqnarray} 
\end{theorem}
\noi
\textbf{Proof.} We denote by $\smash{\mathscr F_{\! t}}$ the $\bP$-completed sigma field generated by 
$\smash{B_{\, \cdot \wedge t}}$. Observe that there is a measurable functional $\Phi$ such that 
$\smash{\big( \tHH^{_\downarrow}_{^{k-1}}, \txx^{_-}_{^{k-1}} (b),  \tHH^{_\uparrow}_{^{k}}, \txx^{_+}_{^{k}} (b),  \tHH^{_{\mathtt{ex}}}_{^k}\big)\eqo \Phi (\theta_{\! \tss^b_{k-1}} \!\!\! \!\! B)}$, where we recall from (\ref{shiftHdef}) the notation 
$\smash{\theta_t H}$. We then observe that $\smash{\tss^{_b}_{^{k-1}}}$ is a $\smash{(\mathscr F_{\! t})_{t\in \bbR_+}\!}$-stopping time. Thus $(i)$ is an immediate consequence of the strong Markov property for $\smash{B_\cdot}$ which asserts that $\smash{\theta_{\! \tss^b_{k-1}} \!\!\! \!\!B}$ is independent of $\smash{\mathscr F_{^{\!\! \tss^b_{k-1}}} \!\!  }$ and has the same law as $B$. 

Therefore, we only need to compute the law of $\smash{\big( \tHH^{_\downarrow}_{^{0}}, \txx^{_-}_{^{0}} (b),  \tHH^{_\uparrow}_{^{1}}, \txx^{_+}_{^{1}} (b),  \tHH^{_{\mathtt{ex}}}_{^1}\big)}$. To this end we denote by $\smash{(\alpha_j, \beta_j)}$, $\smash{j\ino \cJ}$, the excursion intervals of $B$ above its infimum process $\smash{\underline{B}}$. We denote by $\smash{(H_j(\cdot), \zeta_j) }$, $\smash{j\ino \cJ}$, the corresponding excursions. We recall from (\ref{NItomeadefi}) 
that $\smash{\{( -\! \underline{B}_{\alpha_j}} $ $\smash{, (H_j, \zeta_j))\, ;} $ $\smash{  j\ino \mathcal J \big\}}$ is a Poisson point process on $\smash{\bbR_+ \! \times \! \bC}$
whose intensity is equal to $\smash{\mathrm dy\,  \bN (\mathrm d H)}$. We then set $\smash{\tdd^{_b}_{^0}\eqo 
\inf \{ t \ino (\tgg^{_b}_{^0} , \infty) : B(t)\eqo B(\tgg^{_b}_{^0}) \}}$. Then 
$\smash{( \tgg^{_b}_{^0},\tdd^{_b}_{^0})}$ is the first excursion interval $\smash{(\alpha_j, \beta_j)}$ such that $\smash{\Gamma (H_j)\geqo b}$ (we recall that $\smash{\Gamma (H_j)}$ stands for the height of $\smash{H_j}$, i.e., its maximal value). We denote this excursion by 
$\smash{(\frak{h}, \frak{z})\eqo \big( B ( (\tgg^{_b}_{^0}+ \, \cdot \, ) \wedge \tdd^{_b}_{^0})) \! -\! B(\tgg^{_b}_{^0}) , \tdd^{_b}_{^0}\!-\! \tgg^{_b}_{^0} \big)}$. We recall from Definition \ref{elemfundef} $(a)$ the definition of the set 
$\smash{\bC^{_\downarrow}_{^{b}}}$ of the descents with an amplitude less than $b$. 
We recall from (\ref{bvarodef}) the definition of $\smash{\bvaro_{-x}\eqo \varrho_{-x} (B)}$.  
Elementary results on Poisson point processes imply the following for all bounded measurable functions $\smash{F,G\! : \! \bC \! \to \! \bbR}$ and $\smash{f\! : \! \bbR_+ \! \to \! \bbR}$: 
\begin{eqnarray*}
\smash[t]{ \bE \big[ F( \tHH^{_\downarrow}_{^{0}}) f(\txx^{_-}_{^{0}} (b) ) G(\frak{h}) \big]} \!\!\!\! &=&  \!\!\!\! \smash[t]{\bE \Big[ \sum_{j\in \cJ} \! F\big( B_{\, \cdot \wedge \alpha_j}\big)\un_{\{ B_{\, \cdot \wedge \alpha_j }\in\,  \bC^{_\downarrow}_{^{b}} \} } f \big(\! -\! \underline{B}_{\alpha_j} \big) G(H_j )\un_{\{ \Gamma (H_j) \geq b\}}\Big]} \\
 \!\!\!\! & =&  \!\!\!\!  \bN \big[  
 G (\cdot )\un_{\{ \Gamma \geq b \}} \big]  \int_0^\infty \!\!\!\!  \mathrm d x \, f(x) 
\,  \bE \big[ F(B_{\, \cdot \wedge \bvaro_{-x}} \big)\un_{\{ B_{\, \cdot \wedge \bvaro_{-x} }\in\,  \bC^{_\downarrow}_{^{b}} \} } \, \big]  \\
  \!\!\!\! & =&  \!\!\!\!\smash[b]{ \bN \big[ G \, \big| \, \Gamma \geqo b \big] \int_0^\infty \!\! \!\!  \tfrac{\mathrm d x}{b}\,  \ee^{-\frac{x}{b}} f(x) \! \int_{\bC} \!\!  
P^{_\downarrow}_{^{\! x,b}} (\mathrm d H) \, F(H).}
\end{eqnarray*}
Here, we use (\ref{ampliBh}) and to simplify we omit lifetimes. This proves (\ref{condilawBrodec0}) and since 
$\smash{( \tHH^{_\uparrow}_{^{1}}, \txx^{_+}_{^{1}} (b),  \tHH^{_{\mathtt{ex}}}_{^1})}$ is equal to $\smash{\big( \frak{h}( \cdot \wedge r^{_b}_{^1} (\frak{h})), \frak{h} (s^{_b}_{^1} (\frak{h})), 
\frak{h} \big(  ( r^{_b}_{^1} (\frak{h}) + \, \cdot \, )  \wedge s^{_b}_{^1} (\frak{h}) \big) \! -\! \frak{h}(r^{_b}_{^1} (\frak{h})) \big) }$, 
which is a measurable function of $\smash{\frak{h}}$, it also proves that 
$\smash{( \tHH^{_\downarrow}_{^{0}}, \txx^{_-}_{^{0}} (b))}$ and 
$\smash{(\tHH^{_\uparrow}_{^{1}}, \txx^{_+}_{^{1}} (b),  \tHH^{_{\mathtt{ex}}}_{^1})}$ 
are independent. 

  Since $\smash{(\frak h, \frak z)\overset{_{\textrm{law}}}{=}\bN (\, \cdot \,  | \,  \Gamma \geqo b)}$, it only remains to prove the following 
\begin{eqnarray}
\label{Bpluslaw}
\smash{\bN \big[ F( H_{(r^b_1 -\, \cdot )_+} ) f(H_{s^b_1})} \!\!\!\!\!\!\! & & \!\!\!\!\!\!\!\! \smash{G( H_{(r^b_1+ \cdot) \wedge s^b_1})  \, \big| \, \Gamma \geqo b \big] }
\\
\!\!\!\!\! & =&  \quad \; \smash{ \bN [ G (H) \, | \, \Gamma \eqo b ] \! \int_0^\infty \!\!\!\!  \tfrac{\mathrm d x}{b}\, \ee^{-\frac{x}{b}} f(x) \!  \int_{\bC} \!\!  \! 
P^{_\downarrow}_{^{\! x,b}} (\mathrm d H) \, F(H).} \nonumber 
\end{eqnarray}  
for all bounded and measurable $\smash{F,G\! :\! \bC \!  \to \! \bbR}$ and $\smash{f\! :\! \bbR \!  \to \! \bbR}$ (here we omit lifetime to simplify notations). 
\emph{Let us prove (\ref{Bpluslaw}).} We denote by $\smash{Q_x}$ the law of $\smash{(B_{\cdot \wedge \bvaro_{-x}} , \bvaro_{-x})}$.  
We observe that $\bN$-a.e.~$\smash{\un_{\{ N_b \geq 1\}}\eqo  \un_{ \{ \Gamma \geq b\}}\eqo \un_{\{s^b_1 <\infty \}}}$. 
Then, the strong Markov property at $\smash{s^{_b}_{^1}}$ under $\bN$ implies that 
\begin{equation}
\label{Markov}
\smash{\bN \big[  \un_{\{ \Gamma \geq b\} } F_1( H_{\! \, \cdot \wedge s^b_1} )F_2 (\theta_{\! s^b_1} H) \big]\eqo \bN \big[  \un_{\{ \Gamma \geq b\} } F_1( H_{\! \, \cdot \wedge s^b_1} ) Q_{H(s^b_1)} [F_2 ] \big]}
\end{equation}
for all bounded and measurable $\smash{F_1,F_2\! :\! \bC \!  \to \! \bbR}$. 
We then observe that $\bN$-a.e.~$\smash{N_b \eqo 1}$ is equivalent to $\smash{\Gamma \geqo b}$ and $\smash{\theta_{\! s^b_1} H \ino \bC^{_{\downarrow}}_{^b} }$. By (\ref{ampliBh}) and the definition of $\smash{P^{_{\downarrow}}_{^{x,b}}}$ we deduce from (\ref{Markov}) that 
\begin{equation}
\label{MarkovNb1}
\smash{\bN \big[  \un_{\{ N_b=1\} } F_1( H_{\! \, \cdot \wedge s^b_1} )F_2 (\theta_{\! s^b_1} H) \big]\eqo \bN \big[  \un_{\{ \Gamma \geq b\} } F_1( H_{\! \, \cdot \wedge s^b_1} ) \ee^{-H(s^b_1)/b} P^{_{\downarrow}}_{^{\! H(s^b_1),b}} [F_2 ] \big]}
\end{equation}
Next we  note that $\bN$-a.e.~on the event $\smash{\{ N_b \eqo 1\}}$, $\smash{(r^{_b}_{^1} (\overleftarrow{H}) ,  s^{_b}_{^1} (\overleftarrow{H}))\eqo (\zeta \! -\!  s^{_b}_{^1} (H) , \zeta \! -\!  r^{_b}_{^1} (H))}$. Since $\bN$ is invariant under the 
time reversion 
we get the following. 
\begin{eqnarray}
\label{revNb1}
\smash{\bN \big[ \un_{\{ N_b =1 \}} F( H_{(r^b_1 -\, \cdot )_+} ) f(H_{s^b_1})} \!\!  \!\!  \!\!  \!  & &  \!\!  \!\!  \!\!  \! \smash{G( H_{(r^b_1+ \cdot) \wedge s^b_1}) \big]} \\
 \!\!  \!\!  \!\!   &=&  \!\!  \!\!  \smash{ \bN \big[ \un_{\{ N_b =1 \}} F( \theta_{\! s^b_1} H  ) f(H_{s^b_1}) G( H_{(r^b_1+ \cdot) \wedge s^b_1}) \big]}.\nonumber 
\end{eqnarray}
This implies 
\begin{eqnarray}
\label{indeprise}
\smash{\mathrm{LHS}_{\eqref{Bpluslaw}} }&\smash{\overset{\textrm{by (\ref{MarkovNb1})}}{=}}&\smash{ \bN \big[ \un_{\{ N_b =1 \} } \, F( H_{(r^b_1 -\, \cdot )_+} ) \, f(H_{s^b_1}) \, G( H_{(r^b_1+ \cdot) \wedge s^b_1}) \, \ee^{H(s^b_1)/b}  \, \big| \, \Gamma \geqo b \big]  }
\nonumber  \\
& \smash{\overset{\textrm{by (\ref{revNb1})}}{=}}& \smash{\bN \big[  \un_{\{ N_b =1 \} }\,  F( \theta_{\! s^{b}_1} H ) \, f(H_{s^b_1}) \, G( H_{(r^b_1+ \cdot) \wedge s^b_1}) \ee^{H(s^b_1)/b}  \, \big| \, \Gamma \geqo b \big]} \nonumber \\
&\smash{ \overset{\textrm{by (\ref{MarkovNb1})}}{=}}& \smash{\bN \big[ f(H_{s^b_1}) \, G( H_{(r^b_1+ \cdot) \wedge s^b_1})\,  P^{_{\downarrow}}_{^{\! H(s^b_1),b}} [ F] \, \big| \, \Gamma \geqo b \big] }  \label{indeprise} 
\end{eqnarray}
Thus (\ref{Bpluslaw}) easily follows from (\ref{indeprise}) as soon as we have proved:
\begin{equation} 
\label{excuonly}
\smash{\bN \Big[ f(H_{s^b_1})\,  G( H_{(r^b_1+ \cdot) \wedge s^b_1})  \, \Big| \, \Gamma \geqo b \Big] \eqo \bN [ G (H) \, | \, \Gamma \eqo b ] \! \int_0^\infty \!\!\!\!  \tfrac{\mathrm d x}{b}\, \ee^{-\frac{x}{b}} f(x) .}
\end{equation}
\emph{Proof of (\ref{excuonly})}. We recall notations $\smash{\mathtt{Top}_b (\cdot )}$, $\smash{\gamma_b}$, $\smash{\overline{s}}$ and $\smash{\delta_b}$ from Definition \ref{toppdef} and notations $\smash{H^{\uparrow}}$, $\smash{H^{\downarrow}}$ from Lemma \ref{Topmore}. We observe that $\bN$-a.e.~$\smash{N_b \eqo 1}$ if and only if $\smash{H^{\uparrow}}$ and $\smash{H^{\downarrow}}$ belong to $\smash{\bC^{_\downarrow}_{^b}}$ and in this case 
we also get 
$\smash{ \gamma_b \eqo r^{_b}_{^1}}$, $\smash{ \delta_b \eqo s^{_b}_{^1}}$ and thus 
$ \smash{\mathtt{Top}_b (H)\eqo H ((r^{_b}_{^1}+ \cdot) \wedge s^{_b}_{^1})}$ and $\smash{H(s^{_b}_{^1})\eqo \Gamma (H)\! -\! b}$. Then we get the following. 
\begin{eqnarray*}
\smash[t]{\bN \Big[ f \big(H_{s^b_1} \big)\,  G\big(  H_{(r^b_1+ \cdot) \wedge s^b_1} \big) } \! \! \! \! \! \! \! \!  \! \! \! \!\! \! \! \! \!  & &  \! \! \! \!  \! \! \! \! \! \! \! \! \! \! \! \! \! \!  \smash[t]{\un_{\{  \Gamma \geq b\} } \Big] \; }
\smash[t]{\overset{\textrm{by (\ref{MarkovNb1})}}{=}} \; \smash[t]{ \bN \Big[ \un_{\{ N_b=1\}} \mathrm{e}^{H(s^b_1)/b} f\big( H_{\! s^b_1}\big) G \big( H_{(r^b_1 + \cdot)\wedge s^b_1}\big) \Big]} \\
\!\!\!\! \! \! \!  \! \! \! \!   &=& \! \! \! \!  \! \! \! \!  \!\!\!  \bN \Big[ \un_{\{ \textrm{$H^\uparrow$ and $H^\downarrow \in \bC^\downarrow_b$ } \}}  
\mathrm{e}^{\frac{\Gamma-b}{b}} f\big(\Gamma \! -\! b\big) G \big( \mathtt{Top_b(H)}\big)\un_{\{ \Gamma \geq b\}} \Big] \\ 
\!\!\!\! \! \! \!  \! \! \! \!   &=&\!\!\!\! \! \! \!  \! \! \! \!  \!  \!     \int_0^\infty \!\!\!\!  \frac{\mathrm dx}{(b+x)^2} \mathrm{e}^{\frac{x}{b}} f(x) \, \bN \Big[ \un_{\{ \textrm{$H^\uparrow$ and $H^\downarrow \in \bC^\downarrow_b$ } \} 
 G \big( \mathtt{Top_b(H)}\big) \Big| \, \Gamma \eqo b+x  \Big] }\\
\!\!\! \!\!   &\smash[b]{\overset{\textrm{by Lemma \ref{Topmore}}}{=}}& \!\! \!\!\! \smash[b]{\int_0^\infty \!\! \frac{\mathrm dx}{(b+x)^2} \mathrm{e}^{\frac{x}{b}} f(x) \, \bP \big( B_{\cdot \wedge \bvaro_{-x}} \! \ino \bC^{_\downarrow}_{^b} \big| \bvaro_{-x} \leko \bvaro_{b} \big)^2 \, \bN \big[ G(H) \, \big| \, \Gamma \eqo b \big] }, 
\end{eqnarray*}
which implies (\ref{excuonly}) since 
$ \smash{ \bP(  \bvaro_{-x} \leko \bvaro_{b} )\eqo \frac{b}{b+x}}$ and thus 
$\smash{ \bP \big( B_{\cdot \wedge \bvaro_{-x}} \!\!  \ino C^{_\downarrow}_{^b} \big| \bvaro_{-x} \! \leko \bvaro_{b} \big)\eqo \mathrm{e}^{-\frac{x}{b}}\! \cdot \! \frac{b+x}{b}}$ by (\ref{ampliBh}). This completes the proof of the theorem. \cqfd

\smallskip

We recall from (\ref{defYb}) and (\ref{defZb}) the definition of the zigzag functions $\smash{ (Y_{b,z})_{z\in \bbR_+}}$ and 
$\smash{ (Z_{b,z})_{z\in \bbR_+}}$. The process $\smash{ Z_{b, \cdot}}$ is distributed as the first excursion of $\smash{ Y_{b, \cdot}}$ above its infimum and it has 
a finite lifetime denoted by $\smash{ \zeta_b}$. The law of $\smash{ (Z_{b, \cdot}, \zeta_b)}$ is denoted by $\smash{ Q^b}$. 
We also recall from Lemma \ref{buniceldec} the definition of the $k$th $b$-unicellular part of a function which is denoted by 
$\smash{ \tU_{^k}^{_b}(\cdot)}$. We observe that $\smash{ \tU_{^k}^{_b}(B)\eqo \tHH^{_\uparrow}_{^{k}}  \centerdot \tHH^{_{\mathtt{ex}}}_{^k} \centerdot \tHH^{_\downarrow}_{^{k}}} $. 
\begin{proposition}
\label{BMzigzag}  Let $b\ino (0, \infty)$. We keep the previous notations. Then the following holds true. 
\begin{compactenum}

\smallskip

\item[$(i)$] $\smash{ \cE_b B}$ has the same law as $\smash{ Y_{b,\cdot}\, }$. 

\smallskip

\item[$(ii)$] Conditionally given $\smash{ \cE_b B}$, the processes $\smash{ \tHH^{_\downarrow}_{^{0}}} $ and $\smash{ \tU_{^k}^{_b}(B)}$, $\smash{ k\ino \bbN^*}$, are independent: the conditional law of $\smash{ \tHH^{_\downarrow}_{^{0}} }$ is $\smash{ P^{_\downarrow}_{^{\!\! \txx^{-}_{{0}} (b), b}}}$ and for all $\smash{ k\ino \bbN^*}$, the conditional law of 
$\smash{ \tU_{^k}^{_b}(B)}$ is $\smash{ P_{ \!\! \txx^{+}_{{k}} (b), \txx^{-}_{{k}} (b), b}}$.

\smallskip

\item[$(iii)$] The law of $\smash{ \cE_b H}$ under $\smash{ \bN (\, \cdot \, | \, \Gamma \geqo b)}$ is $\smash{ Q^b}$ and thus $\smash{ Q^b (\Gamma \ino \mathrm d a)\eqo b(a+b)^{-2} \mathrm da}$. 
\end{compactenum}
\end{proposition}
\noi
\textbf{Proof.} Theorem \ref{Brodecunicelth} immediatly implies $(i)$, $(ii)$. Let us prove $(iii)$. Without loss of generality 
we take $\smash{Y_{b, \cdot}}$ equal to $\smash{\cE_b B}$ and $\smash{Z_{b, \cdot}}$ equal to the first excursion of $\smash{\cE_{b} B}$ above 
its infimum. 
We recall from the proof of Theorem \ref{Brodecunicelth} the time 
$\smash{\tdd^{_b}_{^0} \eqo \inf\{ t\ino [\tgg^{_b}_{^0}, \infty) : B^{_+}_{^t}\eqo 0 \}}$ and the process $\smash{(H', \zeta')\! :=\!
(B^{+} ((\tgg^{_b}_{^0} + \cdot)\wedge \tdd^{_b}_{^0} ) , \tdd^{_b}_{^0} -\tgg^{_b}_{^0} )}$, which is the first excursion of 
$B$ above its infimum whose height is $\geqo b$: its law
is $\smash{\bN (\, \cdot \, | \Gamma \geqo b)}$. 
Then we denote by $\smash{\mathbf k }$ the largest integer $\smash{k\ino \bbN^*}$ such that $\smash{s^{_b}_{^k} \leqo\tdd^{_b}_{^0}}$. 
Then $\smash{\cE_bH'}$ is the zigzag function taking successively the values $0$, $\smash{B^+(\trr^{_b}_{^1})}$, $\smash{B^+(\tgg^{_b}_{^1})}$, 
$\smash{\ldots, B^+(\trr^{_b}_{^{\mathbf k}})}$, $0$. Namely, $\smash{\cE_b H' (\cdot) \eqo Z_{b, \cdot\, }}$, which proves the first point of $(iii)$. 
We then oberve that under 
$\smash{\bN (\, \cdot \, | \Gamma \geqo b)}$, $\smash{\Gamma (\cE_b H)\eqo \Gamma (H) \! -\! b}$. The first point of $(iii)$ thus implies that 
$\smash{\Gamma }$ under $\smash{Q^b}$ has the same law as $\smash{\Gamma  \! -\! b}$ under $\smash{\bN (\, \cdot \, | \Gamma \geqo b)}$, which completes the proof of the proposition. \cqfd 

\smallskip

As $\bN$, the law $Q^b$ admits a regular version when conditioned by its height. 
\begin{lemma}
\label{regverQbcondheight}
For all $a,b\ino (0, \infty)$ there is $\smash{Q^b (\, \cdot \, | \, \Gamma\eqo a) \ino \mathcal M_1 (\bC)}$ satisfying the following. 
\begin{compactenum}

\smallskip

\item[$(i)$] $\smash{a\ino (0, \infty) \! \mapsto \! Q^b (\, \cdot \, | \, \Gamma\eqo a) \ino \mathcal M_1 (\bC) }$ is continuous.

\smallskip

\item[$(ii)$] For all $a\ino (0, \infty)$, $\smash{Q^b (\, \cdot \, | \, \Gamma\eqo a)}$-a.s.~$\smash{\Gamma\eqo a}$.

\smallskip

\item[$(iii)$] $\smash{Q^b\eqo \int_0^\infty\frac{b \, \mathrm da}{(a+b)^2} \, Q^b (\, \cdot \, | \, \Gamma\eqo a) }$. 

\smallskip

\end{compactenum}
Moreover, for all real numbers $\smash{a'\geqo a\geko 0}$, the law of $\smash{\mathtt{Top}_a (H)}$ under $\smash{Q^b (dH \, | \, \Gamma \geqo a')}$ is equal to $\smash{Q^b (\, \cdot \, | \, \Gamma\eqo a)}$ and the law of $\smash{\cE_bH}$ under $\smash{\bN (dH \, |\, \Gamma \eqo a+b)}$ is also equal to $\smash{Q^b (\, \cdot \, | \, \Gamma\eqo a)}$. 
\end{lemma}
\noi
\textbf{Proof.} For all $\smash{a'\geqo a\geko 0}$ and all bounded and measurable $\smash{F\! :\!  \bC \! \to \! \bbR}$, the following holds true. 
\begin{eqnarray}
\smash{Q^b \big[ F( \mathtt{Top}_a ) \, \big| \, \Gamma \geqo a'  \big]} & \smash{\overset{\textrm{by Prop.~\ref{BMzigzag} $(iii)$}}{=} }& \smash{  (a'+b) \bN \big[F (\mathtt{Top}_a\circ \cE_b )  \un_{\{\Gamma \geq a'+b \}} \big] }\nonumber \\ 
&\smash{ \overset{\textrm{by Remark \ref{eravstop}}}{=} }& \smash{ (a'+b) \bN \big[F (\cE_b \circ \mathtt{Top}_{a+b} )  \un_{\{\Gamma \geq a'+b \}} \big] } \nonumber \\
 &\smash{ \overset{\textrm{ by Lemma \ref{Topmore}}}{=} }&\smash{ \bN \big[ F(\cE_b ) \, | \, \Gamma \eqo a+b \big] . }\label{Topetc}
{}\end{eqnarray}
Let us denote by $\smash{\Lambda_{a,b}}$ the law of $\smash{\cE_b H}$ under $\smash{\bN (dH\, | \, \Gamma\eqo a+b)}$. 
Since (\ref{critcontitop}) $\smash{Q^b(\, \cdot \, | \, \Gamma\geqo a')}$-a.s.~holds, Lemma \ref{measuTop} applies and, combined with (\ref{Topetc}), proves that $\smash{a\ino (0, \infty) \! \mapsto  \! \Lambda_{a,b}\ino \mathcal M_1(\bC)}$ is continuous. Moreover by definition of $\smash{\Lambda_{a, b}}$, a.s.~$\smash{\Gamma\eqo a}$. We then get the following for all bounded and measurable $\smash{f\! : \! \bbR_+ \! \to \!   \bbR}$. 
\begin{eqnarray*}
\smash[t]{\int_0^{a'}\frac{b \, \mathrm da}{(a+b)^2} f(a) \Lambda_{a,b} [F ] }\!\!\!\!\! \!\! &=& \!\!\!\!\! \!\!\!\!\!\!\!\!\! \smash[t]{\int_b^{a'+b}\frac{b \, \mathrm d \alpha}{\alpha^2} f(\alpha \! -b) \bN \big[ F(\cE_b ) \, | \, \Gamma \eqo \alpha\big]} \\
\!\!\!\!\! \!\!  & =& \!\!\!\!\! \!\!  \smash{b \bN\big[ F(\cE_b ) f(\Gamma \! \circ \! \cE_b) \un_{\{b\leq \Gamma \leq a'+b\}} \big]} \\
\!\!\!\!\! \!\! & \smash{ \overset{\textrm{by Prop.~\ref{BMzigzag} $(iii)$}}{=} }& \!\!\!\! \!\! \smash[b]{\int_\bC\!\! Q^b (dH) \,  F(H) f(\Gamma (H)) \un_{\{\Gamma (H) \leq a'\}} ,}
\end{eqnarray*}
which easily entails that $\smash{Q^b\eqo \int_0^\infty\frac{b \, \mathrm da}{(a+b)^2}\Lambda_{a,b}}$. We then take 
$\smash{Q^b(\, \cdot \, | \, \Gamma\eqo a)}$ equal to $\smash{\Lambda_{a, b}}$: it therefore satisfies $(i)$, $(ii)$ and $(iii)$, and 
(\ref{Topetc}) completes the proof of the lemma. \cqfd

\smallskip

We recall the notation $\smash{ \bvaro_{-x} \eqo \varrho_{-x} (B)}$ and to simplify we also set $\smash{\bvaro_{-x}^b\eqo \varrho_{-x} (Y_{b, \cdot})}$. We also recall  the following notations: $\smash{\underline{Y}_{b, t}\eqo \min_{s\in [0, t]} Y_{b,s}}$ and $\smash{Y^{+}_{b, t}\eqo Y_{b, t}  -  \underline{Y}_{b, t}}$. We easily check that $\smash{\cE_b (B_{\cdot \wedge \bvaro_{-x} })}$ has the same law as $\smash{Y_{b\, , \, \cdot \, \wedge \bvaro_{-x}^b}}$. This implies that $\smash{\max_{t\in [0, \bvaro_{-x} ]} B^+_t -b}$ has the same law as $\smash{\max_{t\in [0, \bvaro^b_{-x} ]} Y^+_{b,t} }$. By (\ref{ampliBh}) we thus get for all $\smash{a\ino (0, \infty)}$ 
\begin{equation}
\label{ampliBhbera}
\smash{\bP \Big( \!\! \!\! \!\! \! \max_{\quad \;  t\in [0, \bvaro^b_{-x} ]} \!\! \!\!   Y^+_{b,t}  \leq a \Big) = \exp \Big( \!\!-\frac{x}{a+b} \Big).}
\end{equation}
As in Definition \ref{unicelBrodef}, we introduce the following. 
\begin{definition}
(\emph{Unicellular zigzag processes}) $\,$
\label{unicelZigdef} Let $\smash{x, x^+\! , \,  x^- \! ,\,  a , b\ino \bbR_+}$, with $b\geko 0$. 
\begin{compactenum}

\smallskip

\item[$(a)$] If $x \geko 0$ we denote by $\smash{Q^{_{b\downarrow}}_{^{\! x,a}}}$ the law of 
$\smash{( Y_{b\, , \, \cdot \, \wedge \bvaro^b_{-x}} , \bvaro^b_{-x})}$ under $\smash{\bP ( \, \cdot \, | \, \max_{ t\in [0, \bvaro^b_{-x} ]} \! Y^{_+}_{b, t } \leko  a\, )}$. If $x\eqo 0$, we set 
$\smash{Q^{_{b,\downarrow}}_{^{\! 0,a}}\eqo \delta_{\partial}}$, where $\partial$ stands for the null lifetime function. We call 
$\smash{Q^{_{b\downarrow}}_{^{\! x,a}}}$ the law of the \emph{$b$-zigzag descent of size $x$ with amplitude $<\! a$}.

\smallskip

\item[$(b)$] We denote by $\smash{Q^{_{b\uparrow}}_{^{\! x,a}}}$ the law of $\smash{\overleftarrow{H}}$ under $\smash{Q^{_{b\downarrow}}_{^{\! x,a}} (dH)}$ and we call it the law of the \emph{$b$-zigzag rise of size $x$ with amplitude $<\! a$}.

\smallskip

\item[$(c)$] The law $\smash{Q^b_{x^+ \! ,\, x^-\! ,\, a}}$ of the \emph{$a$-unicellular $b$-zigzag process with initial cell $\smash{(x^+\! , x^-)}$} 
is the law of $\smash{Z^1\!\!  \centerdot Z^2\!  \centerdot Z^3}$, where $\smash{Z^1}$, $\smash{Z^2}$ and $\smash{Z^3}$ are independent zigzag processes that are distributed as follows: $\smash{Z^1}$ has law $\smash{Q^{{b\uparrow}}_{{\! x^+,a}}}$, $\smash{Z^2}$ has law $\smash{Q^b (\, \cdot \, | \, \Gamma\eqo a)}$ and $\smash{Z^3}$ has law $\smash{Q^{{b\downarrow}}_{{\! x^-,a}}}$.\cq 
\end{compactenum}
\end{definition}

\begin{remark}
\label{contiQbx} $(a)$ For a given $x\ino (0, \infty)$, the zigzag process $\smash{Y_{b, \cdot}}$ satisfies the continuity criterion (\ref{hittcontcrit}) in Lemma \ref{regprophit} $(ii)$. This implies 
that $\smash{x \! \mapsto \! Q^{_{b\downarrow}}_{^{x,a}}}$ and also  
$\smash{x \! \mapsto \! Q^{_{b\uparrow}}_{^{x,a}}}$ are weakly continuous in $\smash{\mathcal M_1(\bC)}$, and therefore argueing as in Remark \ref{contiunicelBrorem}, we see that $\smash{(x^+\! , \,  x^-) \! \mapsto \! Q^{_b}_{^{x^+\! ,\, x^-\! ,\, a}}}$ is continuous too and we also get $\smash{Q^{_b}_{^{x^+\! ,\, x^-\! ,\, a}}\!\!\!  \longrightarrow \! Q^b (\, \cdot \, | \, \Gamma \eqo a)}$ as $\smash{(x^+\! ,\, x^-)}$ goes to $(0,0)$. 

\smallskip

\noi
$(b)$ As already mentioned, $\smash{\cE_b (B_{\cdot \wedge \bvaro_{-x} })}$ has the same law as $\smash{Y_{b\, , \, \cdot \, \wedge \bvaro_{-x}^b}}$. Thus, $\smash{\cE_b H}$ under $\smash{P^{_{\downarrow}}_{^{x, a+b}}(dH)}$ has law $\smash{Q^{_{b\downarrow}}_{^{x,a}}}$. By time reversal, the law of $\smash{\cE_b H}$ under $\smash{P^{_{\uparrow}}_{^{x, a+b}}(dH)}$ is equal to $\smash{Q^{_{b\uparrow}}_{^{x,a}}}$.\cq
\end{remark}
\begin{lemma}
\label{beraazig} For all $\smash{x^+\! ,\, x^-\ino \bbR_+}$, $a,b\ino (0, \infty)$, $\smash{\cE_bH}$ under $\smash{P_{x^+\! ,\, x^-\! ,\, a+b}(dH)}$ has law $\smash{Q^{_b}_{^{x^+\! ,\, x^-\! ,\, a}}}$. 
\end{lemma}
\noi
\textbf{Proof.} Let $\smash{(\mathtt H^i\! , \zeta_i)}$, $i\ino \{ 1,2,3\}$ be three $\bC$-valued independent processes such that 
$\smash{(\mathtt H^1\! ,\zeta_1)}$ has law $\smash{P^{\uparrow}_{x^+\! ,\,  a+b}}$, $\smash{(\mathtt H^2 \! , \zeta_2)}$ has law $\smash{\bN (\, \cdot \, | \, \Gamma\eqo a+b)}$ and 
$\smash{(\mathtt H^3\! ,\zeta_3)}$ has law $\smash{P^{\downarrow}_{x^-\!  ,\,  a+b}}$. We denote by $\smash{(\mathtt H,\zeta)}$ their concatenation: $\smash{\mathtt H\eqo \mathtt H^1\! \centerdot \mathtt H^2\! \centerdot \mathtt H^3}$, which is distributed according to $\smash{P_{x^+\! ,\, x^-\! , \, a+b}}$ by definition. It is easy to check that $\smash{\cE_b \mathtt H\eqo \cE_b \mathtt H^1\! \centerdot \cE_b \mathtt H^2\! \centerdot \cE_b\mathtt H^3}$, which implies the desired result by Remark \ref{contiQbx} $(b)$ and by Lemma \ref{regverQbcondheight}.  
\cqfd 

\smallskip

The main goal of this section is to compute the joint law of a $(a+b)$-unicellular Brownian process and its $b$-erased process. To this end we need to introduce the following notation. 
\begin{definition}
\label{Pcellunicel} Let $b\ino (0, \infty)$ et $\smash{\mathbf c \eqo ((x^+_k,x^-_k))_{1\leq k\leq n} \ino \mathtt{Cell}}$. 
We denote by $\smash{P_{\! \mathbf c, b}}$ the law of the concatenated process $\smash{H^1\! \centerdot H^2 \! \centerdot \ldots \centerdot H^n}$, where the processes $\smash{(H^k \! , \zeta_k)_{1\leq k\leq n}}$ are independent and where 
$\smash{(H^k \! ,\zeta_k)}$ has law $\smash{P_{\! {x^+_k \! ,\,  x^-_k\! ,\,  b}}\, }$. \cq
\end{definition}
\begin{remark}
\label{contiencPcb} Since concatenation is continuous (see Remark \ref{remconca} $(c)$) and since $\smash{(x^+\! , \, x^-) \! \mapsto P_{x^+\! ,\,  x^-\! ,\,  b}}$ is continuous (see Remark \ref{contiunicelBrorem}), we get that $\smash{\mathbf c\! \mapsto P_{\! \mathbf c, b}}$ is continuous from 
$\smash{\mathtt{Cell}}$ to $\smash{\mathcal M_1(\bC)}$ equipped with the topology of weak convergence. 
\cq
\end{remark}
\begin{theorem}
\label{condi} For any zigzag function $Z$, we denote by $\mathtt C(Z)$  its associated minimal cell population as in Definition \ref{cellfromzigzag}.  Let $\smash{x^+\! ,\, x^-\ino \bbR_+}$ and let $a,b\ino (0, \infty)$. For all bounded and continuous functions $\smash{F,G\! : \! \bC \! \to \! \bbR}$, we get the following. 
\begin{equation}
\label{contierauni}
\smash{\int_{\bC} \!\! P_{\! x^+\! ,\, x^-\! ,\,  a+b} (\mathrm dH) \, F( \cE_b H) G(H) =\int_{\bC} \!\! Q^b_{\! x^+\! ,\, x^-\! ,\,  a} (\mathrm d Z) \, F(Z) 
\int_{\bC} \!\! P_{\! \mathtt C(Z), b} (\mathrm d H) \, G(H) \; .}
\end{equation}
\end{theorem}
\noi
\textbf{Proof.} Let $\smash{h\ino (0, \infty)}$, $\smash{k\ino \bbN^*}$ and $\smash{(H,\zeta)\ino \bC}$. 
We recall from (\ref{1erbtps}) the definition of $\smash{s^{h}_{1}(H)}$, from (\ref{r1bHdef}), (\ref{g0bdef}) and (\ref{snbdef}) the definitions of $\smash{g^{h}_{k}(H)}$, $\smash{r^{h}_{k}(H)}$ and $\smash{s^{h}_{k}(H)}$, $k\ino \bbN$, with the convention that 
$\smash{r^{h}_{0} (H)\eqo s^{h}_{0} (H)\eqo 0}$. 
We also recall from Lemma \ref{buniceldec} that the $k$-th $h$-unicellular part of $H$ is 
$\smash{H ((g^h_{k-1}(H)\,  + \, \cdot \, )\wedge g^h_k(H) ) \! -\! H (g^h_{k-1}(H))}$, which is denoted by 
$\smash{\mathtt{Uni}^{h}_{k}(H)}$. 
We then recall that $\smash{B_\cdot}$ under $\bP$ is a real valued Brownian motion with initial value $\smash{B_0\eqo 0}$. 
From Notation \ref{notaeraBM} we recall the following shorthands: 
$\smash{\tgg^{h}_{{k}} \! :=\!   g^{h}_{{k}} (B)}$,  
$\smash{\trr^{h}_{{k}}  \! :=\! r^{h}_{{k}} (B)}$, $\smash{\tss^{h}_{{k}}  \! :=\! s^{h}_{{k}} (B)}$, 
$\smash{\txx^{+}_{{k}} (h)  \! :=\!  B(\trr^{h}_{{k}} ) \! -\! B(\tgg^{h}_{{k-1}} )}$ and  $\smash{\txx^{-}_{{k}} (h) \! :=\! B(\tss^{h}_{{k}} ) \! -\! B(\tgg^{h}_{{k}} )}$. In particular $\smash{\mathtt{Uni}^{h}_{1} (B)\eqo B^+ ((\tgg^{{h}}_{0} + \cdot) \wedge \tgg^{{h}}_{1} )}$. 

We set $\smash{(\mathtt Z , \zeta_{\mathtt Z})\eqo \cE_b  (\mathtt{Uni}^{a+b}_{1} (B))}$. By Theorem \ref{Brodecunicelth}, conditionally given $\smash{(\txx^{_+}_{^{1}} (a+b) , \txx^{_-}_{^{1}} (a+b))}$, 
$\smash{\mathtt{Uni}^{a+b}_{1} (B)}$ has law $\smash{P_{\txx^{_+}_{^{1}} (a+b) , \txx^{_-}_{^{1}} (a+b), a+b}}$ and for all bounded and measurable $\smash{g\! : \! \bbR_+^2 \! \to \! \bbR}$, we get 
\begin{eqnarray}
\smash[t]{\bE \Big[ g\big(\txx^{_+}_{^{1}} (a+b) , \txx^{_-}_{^{1}} (a+b) \big) F(\mathtt Z) G\big( \mathtt{Uni}^{a+b}_{1} (B) \big)\Big] }&=& \nonumber  \\ 
\smash[b]{\int_{\bbR^2_+} \!\! \frac{\mathrm d x^+ \mathrm d x^-\! }{(a+b)^2}\, \mathrm{e}^{-\frac{x^+\! +\, x^-\! }{a+b}} g(x^+\! ,\, x^-)}\! \! \! \! \!\!\!\! & &\!\!\! \!  \! \! \! \! \smash[b]{ \int_{\bC} \!\! P_{\! x^+\! ,\, x^-\! ,\,  a+b} (\mathrm dH) \, F( \cE_b H) G(H).    }\label{law1step}
\end{eqnarray}
We next observe that $\smash{Z\eqo \mathtt{Uni}^a_1(\cE_b B) }$ and we also easily check that $\smash{(\txx^{_+}_{^{1}} (a+b) , \txx^{_-}_{^{1}} (a+b))}$ is a measurable function of $\smash{\cE_b B}$. Thus 
\begin{eqnarray}
\smash{\bE \big[ g\big(\txx^{_+}_{^{1}} (a+b) , \txx^{_-}_{^{1}} (a+b) \big) F(\mathtt Z) G\big( \mathtt{Uni}^{a+b}_{1} (B) \big) \, \big| \, \cE_b B\big]} &=& \nonumber  \\ 
\smash{g \big( \txx^{_+}_{^{1}} (a+b) , \txx^{_-}_{^{1}} (a+b) \big) F(\mathtt Z)   }
\! \! \! \! \!\!\! & &\!\!\! \!  \! \! \! \smash{ \bE \big[ G\big( \mathtt{Uni}^{a+b}_{1} (B) \big) \, \big| \, \cE_b B\big]}
 \label{law2step}
\end{eqnarray}
To compute the conditional expectation in the right hand-side of (\ref{law2step}), we first observe that there are two random integers $\mathbf k$ and $\mathbf l$ such that $\smash{\tgg^{_b}_{^{\mathbf k}} \eqo \tgg^{_{a+b}}_{^{0}}}$ and 
$\smash{\tgg^{_b}_{^{\mathbf l}} \eqo  \tgg^{_{a+b}}_{^{1}}}$. Consequently the cell population of $\smash{\mathtt Z}$ is $\smash{\mathtt C(\mathtt Z) \! :=\!  
\big( (\txx^{_+}_{^{\mathbf k + j}} (b) , \txx^{_-}_{^{\mathbf k + j}} (b)) \big)_{1\leq j\leq \mathbf l -\mathbf k}} $ and 
\begin{equation}
\label{concaUniuni}
\smash{\mathtt{Uni}^{a+b}_{1} (B)\eqo \mathtt{Uni}^{b}_{\mathbf k+1} (B)\! \centerdot  \mathtt{Uni}^{b}_{\mathbf k+2} (B) \! \centerdot  \ldots  \centerdot \mathtt{Uni}^{b}_{\mathbf l} (B) .}
\end{equation}
Next note that $(\mathbf k, \mathbf l)$ is a measurable function of $\cE_b B$.
By Proposition \ref{BMzigzag} $(ii)$, (\ref{concaUniuni}) implies that a regular version of the conditional law of $\mathtt{Uni}^{a+b}_{1} (B)$ given $\cE_b B$ is $P_{\mathtt C (\mathtt Z), b}$. Since Lemma \ref{beraazig} asserts that the conditional law of $\mathtt Z$ given 
$ (\txx^{_+}_{^{1}} (a+b) , \txx^{_-}_{^{1}} (a+b))$ has law $Q^{b}_{{\mathtt x^+_1(a+b) ,\mathtt x^-_1 (a+b) ,a}}$, the previous arguments combined with (\ref{law1step}) and (\ref{law2step}) imply 
\begin{eqnarray*}
\mathrm{LHS}_{\eqref{law1step}} \!\!\! &= &\!\!\!  \bE \Big[g\big(\txx^{_+}_{^{1}} (a+b) , \txx^{_-}_{^{1}} (a+b) \big) \! \int_{\bC} \!\!  Q^b_{\! \mathtt x^+_1(a+b)  ,\, \mathtt x^-_1 (a+b) ,\,  a} (\mathrm d Z) \, F(Z) 
\int_{\bC} \!\! P_{\! \mathtt C(Z), b} (\mathrm d H) \, G(H)   \Big] \\
\!\!\! &=& \!\!\!  \int_{\bbR^2_+} \!\! \frac{\mathrm d x^+ \mathrm d x^-\! }{(a+b)^2}\, \mathrm{e}^{-\frac{x^+\! +\, x^-\! }{a+b}} g(x^+\! ,\, x^-) \int_{\bC} \!\!  Q^b_{x^+\! , \, x^-\! , \, , a} (\mathrm d Z) \, F(Z) 
\int_{\bC} \!\! P_{\! \mathtt C(Z), b} (\mathrm d H) \, G(H)  
\end{eqnarray*}
which entails the desired result by continuity of $(x^+\! , \, x^-) \! \mapsto Q^b_{x^+\! , \, x^-\! , \, , a}$ and by the $Q^b_{x^+\! , \, x^-\! , \, , a}(\mathrm dZ)$ a.s.~continuity of 
$Z\mapsto P_{\mathtt{C} (Z), b}$ (see Remark \ref{contiencPcb}).  \cqfd

\subsection{Erasure processes of unicellular Brownian paths and DR branching processes}
\label{lawDRsubsec}

In this section we prove Theorem \ref{mainth3} which constitutes the main part of the proof of Theorems \ref{cvzigzag} and  \ref{recovercell}. We fix $\smash{ x^+\! , \, x^- \ino \bbR_+}$ and $h\ino (0, \infty)$ and we first study the erasure process $\smash{b\ino (0, h] \! \mapsto \! \cE_b H}$ when $H$ is random with law $\smash{P_{x^+\! , \, x^-\! ,\,  h}(\mathrm d H)}$, i.e., when it is distributed as a $h$-unicellular Brownian path with initial cell $\smash{(x^+\! , \, x^-)}$. 
More precisely, we fix $b\ino (0, h]$ and conditionally given $\smash{\cE_bH}$ we compute the next erasure time and location to come in the erasure process of $H$.  
To this end we need to describe the law of the excursion of the reflected process on the subtree spanned by the $b$-erasure times as explained in Section \ref{detererasureprocsec}.

Let $\smash{\mathbf c\eqo ((x^+_k\! , \, x^-_k))_{1\leq k\leq n} \ino \mathtt{Cell}}$, a cell population such that $n\geqo 1$ and 
$\smash{ x^+_k x^-_k\geko 0}$ for all $1\leqo k\leqo n$. We denote by $\smash{\overline{\mathbf c}\eqo (z_j)_{1\leq j\leq 2n}}$ its cumulated version (see Definition \ref{GroDivdef}) and by $\smash{Z\eqo \mathtt{Z} (\mathbf c)}$ its corresponding zigzag function (see Definition \ref{concadef}). Namely, 
$\smash{Z\eqo \phi_{x^+_1} \centerdot \phi_{-x^-_1} \centerdot \ldots \centerdot \phi_{x^+_n} \centerdot \phi_{-x^-_n}}$, where as in Remark \ref{remlifetime} for all $\smash{x\ino \bbR}$, $\smash{\phi_x(t) \eqo \varepsilon (t \wedge |x|)}$, $\smash{t\ino \bbR_+}$, $\varepsilon$ being the sign of $x$ (if $x\eqo 0$, then $(\phi_0, 0)\eqo \partial$). The lifetime of $Z$ is thus $\smash{\zeta_Z\! :=\! \mathtt m (\mathbf c)\eqo \sum_{1\leq k\leq n} (x^+_k \! + x^-_k)}$. 

We fix $\smash{ b\ino (0, h]}$ and we denote by $\smash{ (\mathtt H, \zeta)}$ a 
$\bC$-valued random process with law $\smash{ P_{\mathbf c, b}}$ 
(see Definition \ref{Pcellunicel}). Thus a.s.~$\smash{ \cE_b \mathtt H\eqo Z}$ and $\smash{ N_b (\mathtt H)\eqo n}$. For all $k\ino \{ 1, \ldots, n\}$, we also recall the  notation $\smash{ \overline{s}^{_b}_{^k} 
\eqo \min \{ t\ino [r^{_b}_{^k}, s^{_b}_{^k}]: \mathtt H_t\eqo 
\max_{s\in [r^{_b}_{^k}, s^{_b}_{^k}] }\mathtt H_s \}}$. We then set 
$$\smash{  \big(\mathcal T,d,\rho, < \big)\! :=\! \mathtt{Tree} (\mathtt H) \quad \textrm{and} \quad T_b \! :=\! \bigcup_{^{1\leq k\leq n}} \lgeo \rho, \mathtt{p}_{\mathtt H} (\overline{s}^{_b}_{^k} )\rgeo.}$$
Namely, $\smash{ T_b}$ is the $\smash{ (\overline{s}^{_b}_{^k})_{1\leq k\leq N_b}}$-marginal subtree of $\smash{ \mathcal T}$ as in Definition \ref{margtreedef}. 
We recall from (\ref{eraera}) Lemma \ref{unicelprop} that $\smash{ \mathtt{Tree} (Z) \equiv \bcE_b \mathcal T}$ and that 
$\smash{ \mathtt{Tree} (Z^b) \equiv T_b}$, where 
we recall from Definition \ref{growthzigzag} that $\smash{ Z^b\! :=\! \mathtt{G}\overline{\mathtt{ro}}_b(Z)}$, i.e., $\smash{ Z^b \eqo  \phi_{x^+_1 } \! \centerdot \psi_b \centerdot \phi_{-x^-_1} \! \centerdot \ldots \centerdot \phi_{x^+_{n} } \! \centerdot \psi_b  \centerdot \phi_{-x^-_{n}} }$, where $\smash{  \psi_b\eqo \phi_b \centerdot \phi_{-b}}$. The lifetime of $\smash{ Z^b}$ 
is $\smash{ \zeta_{Z^b}\! :=\! \zeta_Z+2nb}$. The sign-changes sequence of $\smash{ Z^b}$ is $\smash{ (z'_j)_{1\leq j \leq 2n}\! :=\! (z_j+jb)_{1\leq j\leq 2n}}$. We note that $\smash{ Z^b}$ has $\smash{ \mathtt n (Z^b)\eqo n}$ local maxima which are reached at the following set of times:  
\begin{equation}
\label{locmaxbisbis}
\smash{ \mathscr M (Z^b)\eqo \big\{ z'_{2k-1}\, ;\,  1\leqo k\leq n\big\} = \big\{ z_{2k-1}\! + (2k\! -\! 1) b\, ; \, 1\leq k\leqo n\big\}\; .}
\end{equation}
We then denote by $(\smash{ \mathcal H_t(\mathtt H))_{t\in [0, \zeta]})}$ and by $\smash{ (J_t(\mathtt H))_{t\in [0, \zeta]}}$ respectively the reflected process on $\smash{ T_b}$ and the local time at $\smash{ T_b}$, as defined in Lemma \ref{loctimlem}. Namely, a.s.~for all $t\ino [0, \zeta]$ 
\begin{equation}
\label{loctiremindbis}
\smash{  \cH_t(\mathtt H)\eqo d \big( \mathtt p_{\mathtt H} (t), \mathtt{proj}_{T_b} (\mathtt p_{\mathtt H} (t) ) \big) \quad \textrm{and} \quad  \mathtt H_t \eqo \cH_t(\mathtt H)+ Z^b_{\! J_t( \mathtt H)}.}
\end{equation}
We also recall from Lemma \ref{loctimlem} $(i)$ 
that $\smash{ \cH_0(\mathtt H)\eqo  \cH_\zeta(\mathtt H)\eqo 0}$ and we define the excursion intervals of 
$ \smash{ \cH_\cdot (\mathtt H)}$ above $0$ as the intervals $\smash{ (\alpha_j, \beta_j)}$, $\smash{ j\ino \mathcal J}$, which are the connected components of the open subset $\smash{ \{ t\ino (0, \zeta): \mathcal H_t (\mathtt H) \geko 0\}}$. We recall from Remark \ref{connTT} that $\smash{  \mathtt p_{\mathtt H} ( (\alpha_j, \beta_j))}$, $\smash{ j\ino \cJ_b}$, are the connected components of $\smash{ \cT \backslash T_b}$.  
We then define the corresponding excursions $\smash{ ((\mathscr  H_j (\cdot), \zeta^j))_{j\in \mathcal J}}$ as $\smash{ \zeta^j \eqo \beta_j \! -\! \alpha_j}$ and 
$\smash{ \mathscr H_j (t) \eqo \cH_{ (\alpha_j +t) \wedge \beta_j}(\mathtt H)}$, $\smash{ t \ino \bbR_+}$. In the next lemma we compute the law of the following random point process on $\smash{ [0, \zeta_{Z^b}] \! \times \! \bC}$: 
\begin{equation}
\label{needPiH}
\smash{  \Pi (\mathtt H) \eqo \big\{  \big( J_{\! \alpha_j} (\mathtt H) , (\mathscr  H_j (\cdot), \zeta^j)\big) \, ; \, j\ino \mathcal J\big\}. }
 \end{equation}
Recall for all $\smash{ y \ino (0, \infty)}$ that $\smash{ \bN (\mathrm d H\, ; \, \Gamma \leko y)}$ is Ito's measure restricted to the event $\smash{ \{ \Gamma \leko y\}}$. 
\begin{lemma}
\label{PittHcomput} We keep the same assumptions and the same notations as above. Then $\smash{ \Pi (\mathtt H) }$ is a Poisson point process on $\smash{ [0, \zeta_{Z^b}] \! \times \! \bC}$ with intensity 
\begin{equation}
\label{bmudef}
\smash{ \bmu (\mathrm d y,  \mathrm d H) = \un_{[0, \zeta_Z+ 2nb]} (y) \, \mathrm d y \, \bN \big( \mathrm d H \, ;\, \Gamma \leko 
b\wedge  \mathrm{dist} (y, \mathscr M (Z^b) \big) . }
\end{equation}
\end{lemma} 
\noi
\textbf{Proof.} We first explain how to concatenate independent Poisson point processes (P.p.p.~for short) 
on $\smash{ \bbR_+ \! \times \! \bC}$. 
More precisely, we say that a P.p.p.~$\smash{ \Pi}$ on $\smash{ \bbR_+ \! \times \! \bC}$ is \emph{regular} if there is $x\ino (0, \infty)$ 
such that a.s.~$\smash{ \# \Pi \cap ((x, \infty) \! \times \! \bC)\eqo 0}$ but $\smash{ \# \Pi \cap ((x\! -\! \varepsilon, \infty) \! \times \! \bC)
\eqo \infty}$ for all $\varepsilon \ino (0, x)$.Therefore $x$ is measurable with respect to $\smash{ \Pi}$: we call it the 
\emph{horizon} of the P.p.p.~$\smash{ \Pi}$. For all $z\ino (0, \infty)$, we first define $\smash{ \mathtt{Shift}_z \!: \! (y,(H,\zeta))\ino 
 \bbR_+\! \times \! \bC \! \mapsto \! (z+y, (H, \zeta)) 
 \ino  \bbR_+\! \times \! \bC }$. Let $\smash{ \Pi}$ and $\smash{ \Pi'}$ be two independent 
 regular P.p.p.~on $\smash{ \bbR_+ \! \times \! \bC}$ with respective horizons $x$ and $\smash{ x'}$ and with respective intensity 
 $\mu$ and $\smash{ \mu'}$. The concatenation 
 of $\smash{ \Pi}$ and $\smash{ \Pi'}$ is then given by $\smash{ \Pi\centerdot \Pi'\! =\! \Pi \cup \mathtt{Shift}_x (\Pi')}$, which is a regular
P.p.p.~whose horizon is $\smash{ x+x'}$ and whose intensity is $\smash{ \mu (\mathrm dy , \mathrm dH) +\mu' (x+\mathrm dy , \mathrm dH)}$.   

We first prove the lemma when $n\eqo 1$. Namely the law of $\smash{ \mathtt H}$ is $\smash{ P_{\! {x^+_1\! , \, x^-_1\! , \, b}}}$. By 
Definition (\ref{unicelBrodef}) there are three independent $\smash{ \bC}$-valued processes $\smash{ (H^{[i]}, \zeta^{[i]})}$, $\smash{ i\ino \{ 1,2,3\}}$ such that $\smash{ \mathtt H \eqo H^{[1]} \!  \centerdot   H^{[2]}  \!  \centerdot   H^{[3]}}$: the law of 
$\smash{ (H^{[1]}, \zeta^{[1]})}$ is $\smash{ P^{\uparrow}_{^{\! x^+_1\! , \, b}}}$, the law of $\smash{ (H^{[2]}, \zeta^{[2]})}$ is 
$\smash{ \bN ( \, \cdot \, |\,  \Gamma \eqo b)}$ and the law of $\smash{ (H^{[3]}, \zeta^{[3]})}$ is $\smash{ P^{_\downarrow}_{^{\! {x^-_1\! , \, b}}}}$.
For all $\smash{ i\ino \{ 1,2,3\}} $ and for all $\smash{ t\ino [0, \zeta^{[i]}]}$ we introduce the following notations. 
\begin{compactenum}
\item[$-$] We first set 
$\smash{ J_t (H^{[1]}) \eqo m_{H^{[1]} } (t, \zeta^{[1]})}$ and $\smash{ \mathcal H_t (H^{[1]}) \eqo H^{[1]}_t \! -\! J_t (H^{[1]}) }$. 
\item[$-$] We next set $\smash{ \overline{s} \eqo \min \big\{  s \ino [0, \zeta^{[2]}] \, :\, H^{[2]}_s \eqo 
\max_{r\in [0, \zeta^{[2]} ] } H^{[2]}_r \big\}}$. Then we define $\smash{ J_\cdot (H^{[2]})}$ as follows: 
if $\smash{ t\ino [0, \overline{s}]}$, we set $\smash{ J_t (H^{[2]}) \eqo m_{H^{[2]} } (t, \overline{s})}$ and 
if $\smash{ t\ino [\overline{s}, \zeta^{[2]}]}$, we set $\smash{ J_t (H^{[2]}) \eqo }$ $\smash{ 2b }$ $\! -\!$ $\smash{  m_{H^{[2]} } (t, \overline{s})}$. 
Then we set $\smash{ \mathcal H_t (H^{[2]})\eqo H^{[2]}_t\! -\! \psi_b \big(  J_t (H^{[2]})\big)}$, where we recall $\smash{ \psi_b}$ from (\ref{psibdef}). 
\item[$-$] Finally we set $\smash{ J_t (H^{[3]}) \eqo -m_{H^{[3]} } (0,t)}$ and $\smash{ \mathcal H_t (H^{[3]}) \eqo H^{[3]}_t \! -\! J_t (H^{[3]}) }$. 
\end{compactenum}  
We note that $\smash{ H^{[1]}_t\!  \eqo \mathcal H_t (H^{[1]})+ \phi_{x^+_1} \big(J_t (H^{[1]}))}$, $\smash{ H^{[2]}_t\! \eqo \mathcal H_t (H^{[2]})+ \psi_b \big(J_t (H^{[2]}))}$ and $\smash{ H^{[3]}_t\! \eqo \mathcal H_t (H^{[3]})} $ $+$ $\smash{  \phi_{-x^-_1} \big(J_t (H^{[3]}))}$. Then, we easily check  that 
$$\smash{ J  (\mathtt{H})\eqo  J(H^{[1]}) \! \centerdot J(H^{[2]})  \! \centerdot  J(H^{[3]})  \quad \textrm{and} \quad \mathcal H (\mathtt{H})\eqo \mathcal H (H^{[1]}) \! \centerdot  \mathcal H (H^{[2]})  \! \centerdot  \mathcal H (H^{[3]}) \; .}$$
For all $i\ino \{ 1,2,3\}$ we next denote by $\smash{ (\alpha^{[i]}_j, \beta^{[i]}_j)}$, $\smash{ j\ino \cJ^{[i]}}$, the connected components of the open subset $\smash{ \big\{ t\ino [0, \zeta^{[i]}]: \mathcal H_t(H^{[i]}) \geko 0 \big\}}$. For all $\smash{ j\ino \cJ^{[i]}}$ we denote the corresponding excursion by $\smash{ \mathscr H_j^{[i]} (\cdot )\eqo  \mathcal H_{(\alpha^{[i]}_j + \cdot)\wedge \beta_j^{[i]}}(H^{[i]})}$, whose lifetime is $\smash{ \zeta^{[i]}_j\eqo \beta_j^{[i]} \! -\! \alpha_j^{[i]}}$. We then set 
$$\smash{  \Pi (H^{[i]}) \eqo \Big\{ \big( J_{\!^{\alpha_j^{[i]}} } (H^{[i]}), (\mathscr H_{^j}^{_{[i]}} (\cdot ) , \zeta_{^j}^{_{[i]}})\big)\, ; \, j\ino \mathcal J^{[i]} \Big\}\; .}$$
Since $\smash{ \mathcal H_\cdot  (H^{[1]})}$ is distributed as $\smash{ B^+ (\, \cdot \! \wedge \! \bvaro_{-x^+_1})}$ under 
$\smash{ \bP ( \, \cdot \, | \,  \max \{ B^+_t  ; \, t\ino [0, \bvaro_{-x^+_1} ] \} \leko b)}$, we easily see that $\smash{ \Pi (H^{[1]})}$ is a P.p.p.~on $\smash{ \bbR_+ \! \times \! \bC}$ 
with intensity $\smash{ \mu_1 (\mathrm dy  ,  \mathrm dH)  \eqo}$  $\smash{ \un_{[0, x^+_1]} (y)}$  $\smash{ \mathrm d y}$  
$\smash{ \bN (\mathrm dH ;  \Gamma \leko b)}$. It is therefore regular with horizon $\smash{ x^+_1}$. 
Similarly by time-reversal $\smash{ \Pi (H^{[3]})}$ is a regular P.p.p.~on $\smash{ \bbR_+ \! \times \! \bC}$ 
with intensity $\smash{ \mu_3 (\mathrm dy , \mathrm dH)  \eqo}$  $\smash{ \un_{[0, x^-_1]} (y)}$  $\smash{ \mathrm d y}$  
$\smash{ \bN (\mathrm dH;  \Gamma \leko b)}$. Its horizon is $\smash{ x^-_1}$. By (\ref{Bessdec1}) and by Williams' Decomposition (\ref{Williams}), we see that $\smash{ \Pi (H^{[2]})}$ is a regular P.p.p.~on $\smash{ \bbR_+ \! \times \! \bC}$ 
with intensity $\smash{ \mu_2 (\mathrm dy , \mathrm dH)  \eqo}$  $\smash{ \un_{[0, 2b]} (y)}$  $\smash{ \mathrm d y}$  
$\smash{ \bN ( \mathrm dH ;  \Gamma \leko |b\! -\! y| )}$. Its horizon is $2b$. Then we check that 
$\smash{ \Pi (\mathtt H)}$ as defined in (\ref{needPiH}) is equal to $\smash{ \Pi (H^{[1]}) \centerdot  \Pi (H^{[2]})  \centerdot  \Pi (H^{[3]})}$. It is therefore a P.p.p.~on $\smash{ \bbR_+ \! \times \! \bC}$ 
with intensity   $\smash{ \un_{[0\, , \, x^+_1 \! + x^-_1+ 2b]} (y)}$  $\smash{ \mathrm d y}$  
$\smash{ \bN (\mathrm dH ;  \Gamma \leko b \wedge |b+ x^+_1\! -\! y| )}$, which is (\ref{bmudef}) when $n\eqo 1$. This proves the lemma in this case. 
 
The general case proceeds simply from the $n\eqo 1$ case. Indeed, by definition $\smash{ \mathtt H\eqo H^1\! \centerdot \ldots \centerdot H^n}$, where the processes $\smash{ (H^k\! , \zeta^k)_{1\leq k \leq n}}$ are independent and $\smash{ (H^k\! ,\zeta^k)}$ ihas law $\smash{ P_{^{\!\! x^+_k\! , \, x^-_k\! , \, b}}}$. We then see that a.s.~$\smash{ J(\mathtt H) \eqo J(H^1) \centerdot \ldots \centerdot J(H^n)}$, 
$\smash{ \mathcal H (\mathtt H)\eqo \mathcal H  (H^1) \centerdot \ldots \centerdot \mathcal H (H^n)}$ and $\smash{ \Pi(\mathtt H) \eqo \Pi (H^1) \centerdot \ldots \centerdot \Pi(H^n)}$. The $\smash{ n\eqo 1}$ case applies to each $\smash{ \Pi (H^k)}$ and a direct computation shows that $\smash{ \Pi(\mathtt H)}$ is a P.p.p.~with intensity $\smash{ \bmu}$ as defined in (\ref{bmudef}). This completes the proof of the lemma. \cqfd 
  
\smallskip

We next compute the law of the erasure sequence of $\smash{ (\cE_b \mathtt H)_{b\in (0, h] }}$. To this end, for all $x, b\ino (0, \infty)$ and for all $\smash{ n \ino \bbN^*}$ we introduce the following measure on $\smash{ [0, b] \! \times \! \bbR_+}$: 
\begin{equation}
\label{condlawexpliRev}  
\smash{ \, \overline{\! M}_{x,n,b} (\mathrm d r,  \mathrm dy) = \un_{[0, b]} (r) \, \un_{[0\, , \, x+ 2n(b-r)]} (y) \frac{b^{2n}}{r^{2n+2}} \, \exp \! \Big(\!\!  -\! (x+2nb) \big(\tfrac{1}{r}\! -\! \tfrac{1}{b} \big) \Big) \mathrm d r\,  \mathrm dy\, , }
 \end{equation}
which is a probability measure, as a consequence of the following lemma. 
\begin{lemma}
\label{eraseqcomput} Let $\smash{ h,b,x^+\!, \, x^-\!  \, \ino \bbR}$. We assume that $b\ino (0,h]$. Let 
$\smash{ \mathbf c\eqo ((x^+_k\! , \, x^-_k))_{1\leq k\leq n} \ino \mathtt{Cell}}$, a cell population such that $n\geqo 1$ and 
$\smash{ x^+_k x^-_k\geko 0}$ for all $\smash{ k\ino \{ 1, \ldots, n\}}$. We set $\smash{ (Z, \zeta_Z) \eqo (\mathtt{Z} (\mathbf c), \mathtt{m} (\mathbf{c}))}$, the zigzag function associated with $\smash{ \mathbf{c}}$ (see Definition \ref{concadef}). Let 
$\smash{ (\mathtt H, \zeta)}$ be a $\bC$-valued random process with law $\smash{ P_{\mathbf c, b}}$ 
(see Definition \ref{Pcellunicel}). We assume that $\smash{ Z}$ is $\smash{ (h\! -\! b)}$-unicellular with inital cell $\smash{ (x^+\!, \, x^-)}$, i.e., $\smash{ (Z, \zeta_Z) \ino \bC_{x^+\!, \, x^-\!, \, h-b}}$. Then the following holds true. 
\begin{compactenum}

\smallskip

\item[$(i)$] A.s.~$\smash{ \lim_{b'\downarrow 0}N_{b'} (\mathtt H) \eqo \infty}$, $\smash{ N_b(\mathtt H)\eqo n}$ and $\smash{ \mathtt H}$ is $h$-unicellular with inital cell $\smash{ (x^+\!, \, x^-)}$.

\smallskip

\item[$(ii)$] Let $\smash{ (\mathbf b_{n'}, \mathbf y_{n'})_{n'\in \bbN^*}}$ be the erasure sequence of $\smash{ (\cE_{b'} \mathtt H)_{b'\in (0, h] }}$ as in Definition \ref{eraseqdef}. By convenience we set $\smash{ \mathbf b_0\eqo 0}$. Then a.s.~$\smash{ \mathbf b_{n+1} \leko \mathbf b_{n} \leko b\leqo \mathbf b_{n-1}}$ and the law of $\smash{ (\mathbf b_{n}, \mathbf y_{n})}$ is equal to $\smash{ \, \overline{\! M}_{\zeta_Z, n,b} (\mathrm dr, \mathrm dy)}$, which is defined in (\ref{condlawexpliRev}).
\end{compactenum}
\end{lemma}
\noi
\textbf{Proof.} Elementary arguments on Brownian motion imply that 
a.s.~$\smash{ \lim_{b'\downarrow 0}N_{b'} (\mathtt H) \eqo \infty}$.  
We next observe that a.s.~$\smash{ \cE_b \mathtt H\eqo Z}$. Since $\smash{  x^+_k x^-_k\geko 0}$ 
for all $k\ino \{ 1, \ldots, n\}$, we see that the number $\smash{ \mathtt n (Z)\eqo n}$ of local maxima of 
$\smash{ Z}$ is here equal to $\smash{ N_b(\mathtt H)}$. By the semigroup property of erasure (see 
Lemma \ref{semigrera}), $\smash{ \cE_{h}\mathtt H\eqo \cE_{h-b} (\cE_b \mathtt H)\eqo  
\cE_{h-b}Z\eqo \phi_{x^+} \centerdot \phi_{-x^-}}$ since $\smash{ (Z, \zeta_Z) \ino \bC_{x^+\!, \, x^-\!, \, h-b}}$. 
Thus, a.s.~$\smash{ (\mathtt H, \zeta) \ino \bC_{x^+\!, \, x^-\!, \, h}}$, which completes the proof of $(i)$.

  Let us prove $(ii)$. We first note that the erasure sequence $\smash{ (\mathbf b_{n'}, \mathbf y_{n'})_{n'\in \bbN^*}}$ of 
$\smash{ (\cE_{b'} \mathtt H)_{b'\in (0, h] }}$ is well-defined by $(i)$. We keep the notations $\smash{ Z^b}$, 
$\smash{ \mathcal H_\cdot (\mathtt H)}$, $\smash{ J_\cdot (\mathtt H)}$ which satisfy (\ref{loctiremindbis}) and we recall from (\ref{needPiH}) the definition of $\smash{ \Pi (\mathtt H)}$. 
By Proposition \ref{nextjump}, 
$\smash{ \mathbf b_{n} \eqo \max_{t\in [0, \zeta]} \mathcal H_t(\mathtt H) \eqo 
\max_{j\in \mathcal J} \Gamma \big( \mathscr H_j \big)}$ 
(for all $\smash{ (H, \zeta)\ino \bC_{\!  f}}$ we recall that $\smash{ \Gamma (H)\eqo \max_{t\in \bbR_+} \! H_t }$). By Lemma \ref{PittHcomput}, and since the 'law' of $\smash{ \Gamma}$ under $\smash{ \bN}$ is diffuse, 
there exists a unique index $\smash{ \mathbf j_1 \ino \mathcal J}$ 
such that $\smash{ \mathbf b_{n} \eqo \Gamma \big( \mathscr H_{\mathbf j_1} \big)}$. The same argument entails that 
a.s.~$\smash{ \mathbf b_{n+1} \leko \mathbf b_{n} \leko b\leqo \mathbf b_{n-1}}$. We next recall the notations 
$\smash{ Z^b}$ and $\smash{ \mathscr M (Z^b)}$ from (\ref{locmaxbisbis}). For all $\smash{ z, r\ino \bbR_+}$ we set 
$\smash{ I(r,z) \eqo \int_0^z \un_{\{ \mathrm{dist} (y, \mathscr M (Z^b)) >b-r \}} \, \mathrm d y}$. 
By Proposition \ref{nextjump} again, $\smash{ \mathbf y_{n} \eqo I (\mathbf b_{n}, J_{\alpha_{\mathbf j_1}}(\mathtt H) )}$. 
For all $\smash{ r\ino (0, b]}$ we set 
$$\smash{ A(r)\eqo \big\{ (y, (H, \zeta)) \ino [0, \zeta_Z +2nb] \! \times \! \bC_f: \Gamma (H) \geko r \big\} }$$
and for all $\smash{ j\ino \mathcal J }$, we set $\smash{ \Pi_j \eqo \Pi (\mathtt H) \backslash  \big\{  \big( J_{\! \alpha_j} (\mathtt H) , (\mathscr  H_j (\cdot), \zeta^j)\big) \big\}}$. 
Then for all bounded and measurable functions $\smash{ f, g \! : \! \bbR_+\! \to \! \bbR}$, 
Mecke's formula for P.p.p.~and elementary computations imply the following. 
\begin{eqnarray*}
\smash[t]{ \bE \big[ f(\mathbf b_{n}) }\! \! \! \! \! \! \! & & \! \! \! \! \! \!  \! \smash[t]{ g(\mathbf y_{n})\big] = \bE \Big[\sum_{^{j\in \mathcal J}}  
f \big( \Gamma (\mathscr H_j) \big) g\big( I \big( \Gamma (\mathscr H_j), J_{\! \alpha_j}(\mathtt H) \big) \big) \un_{\{\# (\Pi_j \cap A(\Gamma (\mathscr H_j)))= 0 \}} \Big] }\\
\! \! \!  &=&\! \!  \! \smash[t]{  \int \!\! \int_{\bbR_+ \times \bC} \!\!\!\! \!\!\!\! \!\!\!\! \bmu (\mathrm d y , \mathrm d H )\,  f \big( \Gamma (H)\big) \, g\big( I(\Gamma (H), y )\big)
\mathrm{e}^{- \bmu (A(\Gamma (H)) ) }} \\
& =&  \int_0^{\zeta_Z +2nb} \!\!\!\! \! \! \!  \! \! \! \! \!  \mathrm d y \; \int_{\bC} \bN (\mathrm d H) \,  f \big( \Gamma (H)\big) \, g\big( I(\Gamma (H), y )\big) 
\mathrm{e}^{- \bmu (A(\Gamma (H)) ) } \un_{\{ \Gamma (H) \leq b\wedge \mathrm{dist} (y, \mathscr M(Z^b)) \}} 
 \end{eqnarray*}  
$$ \smash{\!\!\!\!\!\!\! \!\!\!\!\!\!\! \!\!\!\!\!\!\! \!\!\!\!\!\!\! = \int_0^b \frac{\mathrm d r}{r^2}\,  f(r) \, \mathrm{e}^{- \bmu (A(r) ) } \! \int_0^{\zeta_Z +2nb} \!\!\!\! \! \! \!  \! \! \! \! \!  
\mathrm d y \,  g( I(r, y )) \un_{\{ r\leq   \mathrm{dist} (y, \mathscr M(Z^b)) \}} }.$$
Since $ \smash{\partial_y I(r,y)\eqo \un_{\{ r\leq   \mathrm{dist} (y, \mathscr M(Z^b)) \} }}$ and since $ \smash{I(r, \zeta_Z+2nb)\eqo 
\zeta_Z+2n(b\! -\! r)}$, a simple change of variable entails the following. 
\begin{equation}
\label{Meckeconseq}
 \smash{ \bE \big[ f(\mathbf b_{n}) \, g(\mathbf y_{n})\big]=  \int_0^b \frac{\mathrm d r}{r^2}\,  f(r) \, \mathrm{e}^{- \bmu (A(r) ) } \int_0^{\zeta_Z +2n(b-r)} \!\!\!\! \! \! \!  \! \! \! \! \!  \! \! \! \!  \! \! \! \!  \! \! \! \!  
\mathrm d y \,  g( y).}
\end{equation}

Next, we need to compute $ \smash{\bmu (A(r))}$. To this end we recall that $ \smash{\overline{\mathbf{c}}\eqo (z_j)_{1\leq j\leq 2n}}$ stands for the cumulated version of $ \smash{\mathbf c}$ and we also recall $ \smash{\mathscr M(Z^b)}$ from (\ref{locmaxbisbis}). Then the following holds true. 
\begin{eqnarray*}
\smash{\!\!\!\!\!\!\!\! \bmu (A(r))} \!\!\! &=&\!\!\! \!\!\!\! \smash{\int_0^{ \zeta_Z +2nb} \!\!\!\! \! \! \!  \! \! \! \! \!   \! \! \!   
\mathrm d y \,  \bN\big(r\leko \Gamma \leko b\! \wedge \! \mathrm{dist} (y, \mathscr M(Z^b)) \big) }\\
\!\!\! &=& \!\! \! \smash[b]{ \bN (r\leko \Gamma \leko b )\int_0^{\zeta_Z +2nb} \!\!\!\! \! \! \!  \! \! \! \! \!   \! \! \!   \mathrm d y \,  
\un_{ \{  b < \mathrm{dist} (y, \mathscr M(Z^b)) \}} }\\
\!\!\! & & \!\! \!  \qquad \quad + \sum_{^{1\leq k\leq n}} 
   \int_{z_{2k-1}' \! -b}^{z_{2k-1}'\!  -r} \!\!\!\! \! \! \!  \! \! \! \! \!   \! \! \!  \mathrm d y \,
\bN\big(r\leko \Gamma \leko z_{2k-1}' \! -\! y \big) +
\int_{z_{2k-1}' +r}^{z_{2k-1}' +b} \!\!\!\! \! \! \!  \! \! \! \! \!   \! \! \!  \mathrm d y \,
\bN\big(r\leko \Gamma \leko y  \! -\! z_{2k-1}'  \big) \\
\!\!\! &=& \!\! \! \smash{ \; \, \bN (r\leko \Gamma \leko b )\, \zeta_Z \; + \; 2n \! \int_r^b\! \mathrm d z \, \bN (r\leko \Gamma \leko z )} \\
\!\!\! &=& \!\! \! \smash[b]{ \big( \tfrac{1}{r} \! -\! \tfrac{1}{b} \big) \, \zeta_Z+ 2n  \int_r^b\! \mathrm d z \, \big( \tfrac{1}{r} \! -\! \tfrac{1}{z} \big) =  (\zeta_Z+ 2nb)\big( \tfrac{1}{r} \! -\! \tfrac{1}{b} \big) + 2n \log \tfrac{r}{b},} 
\end{eqnarray*}  
which entails $(ii)$ by (\ref{Meckeconseq}).  \cqfd 

\smallskip

We now state and prove the main result of this article which implies a large part of Theorems \ref{cvzigzag} and  \ref{recovercell}. Recall that $\mathtt C(Z)$ is the minimal cell population associated with a zigzag function $Z$. 
\begin{theorem}
\label{mainth3} Let $ \smash{ x^+\! , \, x^-\ino \bbR_+}$ and $h\ino (0, \infty)$. Let $ \smash{ (\bH, \bzeta)}$ be 
distributed according to $ \smash{ P_{^{\! x^+\! , \, x^-\! , \, h}}}$, i.e., as a Brownian $h$-unicellular process with initial 
cell $ \smash{ \bcc(0)\eqo (x^+\! ,\, x^-)}$. For all $ \smash{ t\ino [0, 2h)}$
we set $ \smash{ \mathbf c (t) \eqo \mathtt C (\cE_{\! \frac{_1}{^2} (2h-t)} \bH)}$. Then, 
$ \smash{ (\mathbf c(t))_{t\in [0, 2h)}}$ is a D-R branching process with initial cell 
$ \smash{ \bcc(0)\eqo (x^+\! , x^-)}$ and time-horizon $ \smash{ 2h}$. In particular, a.s.~for all 
$ \smash{ t\ino [0, 2h)}$, $ \smash{ \mathtt n(\mathbf{c}(t))}$ is the number of local maxima of 
$ \smash{ \cE_{\! \frac{_1}{^2} (2h-t)} \bH}$ and $ \smash{ \mathtt m (\mathbf c(t))\eqo M_{\frac{_1}{^2} (2h-t)} (\bH)}$, 
which is the lifetime of 
$ \smash{ \cE_{\! \frac{_1}{^2} (2h-t)} \bH}$.  
\end{theorem}
\noi
\textbf{Proof.} To simplify notation, for all $ \smash{ t\ino [0, 2h)}$, we set $ \smash{ \mathtt n_t \! :=\! \mathtt n(\mathbf c(t))}$ and 
$ \smash{ \mathtt m_t \! :=\! \mathtt m(\mathbf c(t))}$. 
We easily check that a.s.~$ \smash{ \lim_{b\downarrow 0} N_b (\bH)\eqo \infty}$. The erasure sequence of 
$ \smash{ (\cE_b \bH)_{b\in (0, h]}}$ is thus well-defined and we denote it by $ \smash{ (\mathbf b_{n}, \mathbf y_{n})_{n\in \bbN^*}}$ 
(see Definition \ref{eraseqdef}). 
By Proposition \ref{celevovsera}, almost surely $ \smash{ (\mathbf c(t))_{t\in [0, 2h)}}$ is a growth-division evolution 
which characterized by its initial cell $ \smash{ \bcc(0)\eqo (x^+\! , x^-)}$ and its division sequence is equal to 
$ \smash{ (\mathbf t_n, \mathbf y_n)_{n\in \bbN^*}}$ where $ \smash{ \mathbf t_n \eqo 2(h \! -\! \mathbf b_n) }$. 
Note that (\ref{sizemassconnec}) in Proposition \ref{celevovsera} entails the last statement of the theorem. 
To complete the proof, we show that $ \smash{ (\mathbf t_n, \mathbf y_n)_{n\in \bbN^*}}$ is distributed 
as the division sequence of a D-R branching process and to this end, we use Lemmas \ref{divseqidentif} 
and \ref{condlawDRcomput}: 
namely, we first have to prove that $ \smash{ M_{x^+ + x^-\! ,\,  1, 0}}$ is the law of $ \smash{ (\mathbf t_1, \mathbf y_1)}$, 
where $M_{y,n,t}$ is defined in (\ref{condlawexpli}), and next to prove that $ \smash{ M_{\mathtt m_t  , \mathtt n_t,t }}$ 
is the conditional law of 
$ \smash{ (\mathbf t_{\mathtt n_t }, \mathbf y_{\mathtt n_t })}$ given $ \smash{ \mathscr F_{\! t}}$, which is the sigma field generated by $ \smash{ (\mathbf c(s))_{s\in [0, t]}}$ or equivalently as the sigma field generated by 
$ \smash{ (\cE_{\! \frac{_1}{^2} (2h-s)} \bH )_{s\in [0, t]}}$. 

Since $ \smash{ (\bH, \bzeta)}$ has law $ \smash{ P_{\! x^+\! , \, x^-\! , \, h}}$, we can apply Lemma \ref{eraseqcomput} with $ \smash{ b\eqo h}$, $ \smash{ Z\eqo \phi_{x^+}\centerdot \phi_{-x^-} \eqo \cE_h \bH}$, $ \smash{ \zeta_Z\eqo \mathtt m_0\eqo x^+\! + x^-}$ and $ \smash{ n \eqo 1}$. It shows that 
$ \smash{ (\mathbf b_1, \mathbf y_1)}$ has law $ \smash{ \, \overline{\! M}_{ x^+ + x^-\! ,\,  1, h} }$ as defined in (\ref{condlawexpliRev}). A simple change of variable then entails that $ \smash{ (\mathbf t_1,  \mathbf y_1)\eqo (2(h \! -\! \mathbf b_1),  \mathbf y_1)}$ has law 
$ \smash{ M_{x^++ x^-\! ,\,  1, 0}}\, $. 

  We now compute the conditonal law of $ \smash{ (\mathbf t_{\mathtt n_t }, \mathbf y_{\mathtt n_t })}$ given $ \smash{ \mathscr F_{\! t}}$. 
By Theorem \ref{condi}, a regular version of the conditional law of $ \smash{ (\bH, \bzeta)}$ given $ \smash{ \mathscr F_{\! t}}$ is $ \smash{ P_{\! \mathbf c (t), \frac{_1}{^2}(2h-t)}}$. We then apply Lemma \ref{eraseqcomput} to $ \smash{ b\eqo  \frac{_1}{^2}(2h-t)}$, $ \smash{ (Z, \zeta_Z)\eqo ( \cE_{\! \frac{_1}{^2} (2h-t)} \bH, \mathtt m_t)}$: thus $ \smash{ \mathbf b_{\mathtt n_t} \leko b\eqo \frac{_1}{^2} (2h-t) \leqo  b_{\mathtt n_t-1}}$ and the conditional law of 
$ \smash{ (\mathbf b_{\mathtt n_t }, \mathbf y_{\mathtt n_t })}$ given $ \smash{ \mathscr F_{\! t}}$ is thus $ \smash{ \, \overline{\! M}_{\mathtt m_t, \mathtt n_t, \frac{_1}{^2} (2h-t)}}$. A simple change of variable then entails that $ \smash{ (\mathbf t_{\mathtt n_t},  \mathbf y_{\mathtt n_t})\eqo (2(h \! -\! \mathbf b_{\mathtt n_t}),  \mathbf y_{\mathtt n_t})}$ 
has law $ \smash{ M_{\mathtt m_t, \mathtt n_t, t}}$. By Lemma \ref{divseqidentif} and Lemma \ref{condlawDRcomput}, 
$ \smash{ (\mathbf t_n, \mathbf y_n)_{n\in \bbN^*}}$ is distributed as the division sequence of a D-R branching process with time horizon $2h$ and initial cell 
$ \smash{ \bcc(0)\eqo (x^+\! ,\, x^-)}$, which completes the proof of the theorem. \cqfd 

\subsection{Proofs of Theorems \ref{cvzigzag} and  \ref{recovercell}}
\label{LLNDRsubsec}
Let $\smash{ x^+\! , \, x^-\ino \bbR_+}$ and let $h\ino (0, \infty)$. Let $\smash{(\bH , \bzeta)}$ be distributed according to $\smash{P_{\! x^+\! , \, x^- \! , \, h}}$. 
By Theorem \ref{mainth3}, $\smash{t\ino [0, 2h) \! \mapsto \! \mathbf c (t) \eqo \mathtt C (\cE_{\frac{1}{2} (2h -t)} \bH)}$ is a D-R branching process with initial cell $\smash{(x^+\! , \, x^-)}$ (and time-horizon $2h$). 
To prove Theorems \ref{cvzigzag} and \ref{recovercell}, we thus only need to prove that 
\begin{equation}
\label{cvproba1}
\smash{ \Big(  \big( \cE_{b} \bH \big( \tfrac{t}{b}) \big)_{t\in \bbR_+} \, , \,  bM_b (\bH)  \Big) \xrightarrow[b\downarrow 0]{\; }
\big( \bH, \bzeta \big) \quad \textrm{in $\bP$-probability on $\bC$}}
\end{equation}
We first claim that (\ref{cvproba1}) follows from the similar convergence for 
the Brownian motion $B$:  
\begin{equation}
\label{cvproba2}
\smash{ \big( \cE_{b} B \big( \tfrac{t}{b}) \big)_{\! t\in \bbR_+}  \xrightarrow[b\downarrow 0]{\; } ( B_t )_{t\in \bbR_+} \quad \textrm{in $\bP$-probability on $(\bC^0 (\bbR_+, \bbR), \delta)$.}}
\end{equation}
\noi
\emph{Proof of (\ref{cvproba2}) $\! \Rightarrow \! $ (\ref{cvproba1})}. 
To simplify we denote by $\smash{B_{b} (\cdot)}$ the 
process in the left hand side of (\ref{cvproba2}) and we recall the notations 
$\smash{B^+_b (t) \eqo B_b(t)\! -\! \min_{s\in [0 ,t]} B_b(s)}$ and $\smash{B^+ (t) \eqo B(t)\! -\! \min_{s\in [0 ,t]} B(s)}$. Let $\smash{x\ino (0, \infty)}$. Recall from (\ref{hittim}) the notation $\smash{\varrho_{-x} (H)}$, for all $\smash{H \ino \bC^0 (\bbR_+, \bbR)}$ and that 
$\smash{\bvaro_{-x}\eqo \varrho_{-x} (B)}$. 
Since $\smash{B_\cdot}$ a.s.~satisfies (\ref{hittcontcrit}), Lemma \ref{regprophit} $(ii)$ applies and by (\ref{cvproba2}), we get 
\begin{equation}
\label{cvprobhitting}
\smash{\lim_{b\downarrow 0}\big( B_{b} (\cdot \wedge \varrho_{-x} (B_b)),  \varrho_{-x} (B_b) \big)\eqo (B_{\cdot \wedge \bvaro_{-x} }, \bvaro_{-x}) \; \textrm{ in $\bP$-probability on $\bC$.}}
\end{equation} 
Since the law of $\smash{\max \{ B^+_t ;\,  t\ino [0, \bvaro_{-x}]\}}$ is diffuse, a similar convergence holds in probability with respect to $\smash{\bP (\, \cdot \, | \, \max \{ B^+_t ;\,  t\ino [0, \bvaro_{-x}]\}  \leko h)}$-probability. By definition of $\smash{P^{\downarrow}_{\! x,h}}$ and $\smash{P^{\uparrow}_{\! x,h}}$ and by continuity of time-reversal, it entails: 
\begin{equation}
\label{risedesproba}
\smash{ \Big(  \big( \cE_{b} H \big( \tfrac{t}{b}) \big)_{t\in \bbR_+} \!  ,   bM_b (H)  \Big) \underset{b\to 0^+}{  -\!\!\! -\!\!\!  \longrightarrow} \big( H, \zeta \big) \; \textrm{in $P^{\downarrow}_{\! x,h} (\mathrm dH)$- and $P^{\uparrow}_{\! x,h} (\mathrm dH)$-probability on $\bC$.}}
\end{equation}

We next want to derive from (\ref{cvproba2}) a similar result firstunder $\smash{ \bN (\mathrm dH \, | \, \Gamma \geqo h)}$ and next under $\smash{ \bN (\mathrm dH \, | \, \Gamma \eqo h)}$. 
For all $b\ino (0, h)$, we set $\smash{ \mathtt t (b)\eqo \varrho_{h-b} (B^+_b)}$ and we also set $\smash{ \mathtt t(0)\eqo \bvaro_h}$. We easily see that a.s.~$-B$ satisfies the continuity assumption (\ref{hittcontcrit}): Lemma \ref{regprophit} $(ii)$ implies that a.s.~$\smash{ \lim_{b\downarrow  0} \mathtt t (b)\eqo \mathtt t(0)}$.    
We next introduce the following times: 
$\smash{ \mathtt g(b)\eqo \max \{ t\ino [0, \mathtt t(b)] : B^+_b(t) \eqo 0\}}$ and $\smash{ \mathtt d(b)\eqo \min \{ t\ino [\mathtt t(b), \infty) : B^+_b(t) \eqo 0\}}$, with the convention that $\smash{ B_0 (\cdot) \eqo B}$. We 
also introduce the event $\smash{ A\! :=\! \{\mathtt t (b)\! \to \!  \mathtt t(0) \, ; \,  \mathtt g(0) \leko \mathtt t(0) \leko \mathtt d(0) \, ; \,  \forall \varepsilon \ino (0,1),\,  \underline{B}_{\mathtt d(0)+ \varepsilon } \leko   \underline{B}_{\mathtt t(0)} \leko  \underline{B}_{\mathtt g(0) -\varepsilon} \}}$. Observe that $\bP (A)\eqo 1$. On $A$ we first prove that 
$\smash{ \lim_{b\downarrow 0} (\mathtt g(b), \mathtt d(b))\eqo (\mathtt g(0), \mathtt d(0))}$. 

\emph{Indeed}, 
the $\delta$-continuity of $\smash{ H\mapsto \underline{H}}$ and (\ref{cvproba2}) imply for all $\varepsilon\ino (0, 1)$ 
that there is $\smash{ b_\varepsilon\ino (0, \infty)}$ such that for all $\smash{ b\ino (0, b_\varepsilon)}$, 
$\smash{ \underline{B}_b (\mathtt d(0)+ \varepsilon) \leko   \underline{B}_b (\mathtt t(b))  
\leko  \underline{B}_b (\mathtt g(0) \! -\! \varepsilon)}$. It entails $\smash{ \mathtt g(0)
\! -\! \varepsilon \leko \mathtt g(b)} $ $ \leko$ $ \smash{ \mathtt d(b)} $ $ \leko$ $\smash{  \mathtt d (0) } $ $+$ $\smash{  \varepsilon}$. Let 
$\smash{ \varepsilon_0 \ino (0, 1)}$ be such that $\smash{ \mathtt g(0) + \varepsilon_0 \leko \mathtt t(0) \leko \mathtt d(0)
\! -\! \varepsilon_0}$. We argue on the event $A$ and we fix $\smash{ \varepsilon \ino (0, \varepsilon_0)}$. 
Then $\smash{ \min \{ B^+_t \, ; \, t\ino [\mathtt g(0) + \varepsilon , \mathtt d(0) \! -\! \varepsilon] \}\geko 0}$. 
By the $\delta$-continuity of $\smash{ H\! \mapsto \! H^+}$ and (\ref{cvproba2}), 
there is $\smash{ b'_\varepsilon\ino (0, \infty)}$ such that 
for all $\smash{ b\ino (0, b_\varepsilon)}$, $\smash{ \mathtt t(b) \ino  [\mathtt g(0) + \varepsilon , \mathtt d(0) \! -\! \varepsilon] }$ 
and $\smash{ \min \{ B^+_b(t) \, ; \, t\ino [\mathtt g(0) + \varepsilon , \mathtt d(0) \! -\! \varepsilon] \}\geko 0}$, 
which implies $\smash{ \mathtt g(b) \leko \mathtt g(0) + \varepsilon}$ and $\smash{ \mathtt d(0) \! -\! \varepsilon \leko \mathtt d(b)}$. 
This proves that a.s.~$\smash{ \lim_{b\downarrow 0} (\mathtt g(b), \mathtt d(b))\eqo (\mathtt g(0), \mathtt d(0))}$. 

  To simplify the notation we set $\smash{ \, \mathtt H (\cdot)  \eqo B^+ \big( \mathtt g(0)+ \cdot )\wedge \mathtt d(0) \big)}$: 
$\smash{ \mathtt H (\cdot)}$ is the first excursion of $\smash{ B^+}$ above $0$ whose height is $\smash{ \geqo h}$, its law is equal to $\smash{ \bN ( \, \cdot \, | \, \Gamma \geqo h)}$ and  
$\smash{ \cE_b \mathtt H (\cdot/b)\eqo  B^+_b (\mathtt g(b)+ \cdot )\wedge \mathtt d(b) )}$. Thus (\ref{cvproba2}) 
and the previous arguments imply that 
\begin{equation}
\label{cvprobexcugeqoh}
\smash{ \lim_{b\to 0}\big( \cE_b \mathtt H (\cdot/b), \mathtt g(b), \mathtt d(b)\big)\eqo \big(  \mathtt H (\cdot), \mathtt g(0), \mathtt d(0)\big) \; \, \textrm{in $\bP$-probability on $\bC \! \times\!  \bbR_+^2$.}}
\end{equation}
For all $b\ino [0, h)$, we define 
$\smash{ \overline{s}(b)\eqo \min \big\{ t\ino [\mathtt g(b), \mathtt d(b)]\! : \! B^+_b(t)\eqo \max \{ B^+_b(s)  ;  t\ino [\mathtt g(b) , \mathtt d(b)] \} \big\}}$, $\smash{ \gamma (b) \eqo }$ $\smash{\max \big\{ t\ino [\mathtt g(b), \overline{s}(b)]: B^+_b(t)\eqo B^+_b(\overline{s}(b))\! -\! h+b \}}$ and $\smash{ \delta (b)\eqo \min \big\{ t\ino [ \overline{s}(b),\mathtt d(b)]: B^+_b(t)} $ 
$\eqo$  $\smash{B^+_b(\overline{s}(b))\! -\! h+b \}}$ (with the conventions that $B_0 (\cdot)\! :=\!  B_\cdot$ and $B^+_0 (\cdot) \! :=\!  B^+_\cdot$). By Definition 
of topping we obtain $\smash{ \mathtt{Top}_{h-b} (\cE_b \mathtt H ) (\cdot/b)\eqo B_b ((\gamma (b) + \cdot )\wedge \delta (b))\! -\! B_b (\gamma (b))}$. By Remark \ref{eravstop}, we get $\smash{ \mathtt{Top}_{h-b} (\cE_b \mathtt H )\eqo \cE_b (\mathtt{Top}_{h} (\mathtt H))}$. 
Now we easily see 
that $\smash{ \mathtt H(\cdot)}$ a.s.~satisfies the continuity assumption (\ref{critcontitop}). Then Lemma \ref{measuTop} $(ii)$ 
applies and combined with (\ref{cvprobexcugeqoh}) it shows that 
\begin{equation}
\label{cvprobexcueqoh}
\smash{ \lim_{b\downarrow 0}\big(  \cE_b (\mathtt{Top}_{h} (\mathtt H))(\cdot/b), \gamma (b),\delta(b)\big)\eqo \big( \mathtt{Top}_{h} (\mathtt H), \gamma (0), \delta(0)\big) \; \, \textrm{in $\bP$-probability on $\bC\! \times\!  \bbR_+^2$.}}
\end{equation}
By Lemma \ref{Topmore} $(iii)$, the law of $\smash{ \mathtt{Top}_{h} (\mathtt H)}$ is equal to $\smash{ \bN ( \, \cdot \, | \, \Gamma \eqo h)}$. Therefore, we have proved that 
\begin{equation}
\label{condiexcproba}
\smash{  \Big(  \big( \cE_{b} H \big( \tfrac{t}{b}) \big)_{t\in \bbR_+} \!  ,   bM_b (H)  \Big) \underset{b\to 0^+}{  -\!\!\! -\!\!\!  \longrightarrow} \big( H, \zeta \big) \; \, \textrm{in $\bN (\mathrm dH \, | \, \Gamma\eqo h)$-probability on $\bC$.}}
\end{equation}
By definition, $(\bH, \bzeta)$ is the concatenation of three independent $\bC_{\! f}$-valued processes whose respective laws are $\smash{ P^{_\uparrow}_{\!\! ^{^{x^+\! , \, h}}}}$, $\smash{ \bN (\, \cdot \, | \, \Gamma\eqo h)}$ and $\smash{ P^{_\downarrow}_{\! \! ^{^{x^-\! , \, h}}}}$. Since concatenation is continuous on $\bC_{\! f}$, (\ref{risedesproba}) and (\ref{condiexcproba}) prove that (\ref{cvproba2}) implies (\ref{cvproba1}). \cq
  
\smallskip

\noi  
\emph{Proof of (\ref{cvproba2}).} We recall $\smash{ \mathtt x^+_k(b)}$ and $\smash{ \mathtt x^-_k(b)}$ from Notation \ref{notaeraBM}. 
Then the $b$-erased Brownian motion $B$ is the zigzag function given by 
$\smash{ \cE_b B (\cdot)  \eqo \phi_{^{-\mathtt x^{-}_{0} (b)}} \centerdot (\phi_{^{\mathtt x^{+}_{1} (b)}}\centerdot \phi_{^{-\mathtt x^-_1 (b)}}) \centerdot \ldots \centerdot (\phi_{^{\mathtt x^+_k (b)}}\centerdot } $ $\smash{ \phi_{{-\mathtt x^-_k (b)}}) \centerdot \ldots \, }$ which makes sense even if its lifetime is infinite. We recall that the $\smash{ \mathtt x^{_\pm}_{^k}(b)}$ are i.i.d.~$\smash{ \mathtt{expo}(\frac{_1}{^b})}$ r.v.s. Thus, $\smash{ \cE_b B (\cdot)}$ has the same law as $\smash{ Y_{b, \cdot}}$ (see Theorem \ref{Brodecunicelth} and Proposition \ref{BMzigzag}). Its starts with a descent and its sign-changes sequence 
$\smash{ (\mathtt{z}_j)_{j\in \bbN}}$ is given by 
$\smash{ \mathtt{z}_0\eqo \mathtt{x}^-_0 (b)}$, $\smash{ \mathtt{z}_{2k+1}\eqo  \mathtt{z}_{2k}+ \mathtt x^+_{k+1} (b)}$ 
and $\smash{ \mathtt{z}_{2k+2}\eqo  \mathtt{z}_{2k+1}+ \mathtt x^-_{k+1} (b)}$ for all $\smash{ k\in \bbN}$. The sequence of positive times at which $\smash{ \cE_b B_\cdot}$ reaches a local maximum is $\smash{ (\mathtt z_{2k+1})_{k\in \bbN}}$. The construction given in  (\ref{loctiremind}) extends to $\smash{ B_\cdot}$: there exist a reflected process $\smash{ (\mathcal H_t )_{t\in \bbR_+}}$ and a local time $\smash{ (\mathtt J_t)_{t\in \bbR_+}}$ at the real tree spanned by the times $\smash{ \overline{s}^{_b}_{^k} \eqo \min \big\{ t\ino [\mathtt r^{_b}_{^k} ,\mathtt s^{_b}_{^k}]
: B_t\eqo \max \{ B_s; s\ino  [\mathtt r^{_b}_{^k} ,\mathtt s^{_b}_{^k}]\} \big\}}$, $\smash{ k\ino \bbN^*}$, which satisfy 
$\smash{ B_t \eqo \mathcal H_t + Z^b \big( \mathtt J_t\big)}$, for all $\smash{ t\ino \bbR_+}$. 
Here, $Z^b$ is derived from $\smash{ \cE_bB_\cdot}$ by $b$-growth, i.e., at each positive maximum time 
$\smash{ \mathtt z_{2k+1}}$ we inserts a $\wedge$-shaped zigzag function $\smash{ \psi_b}$ whose lifetimes is equal to $2b$. Namely, 
$\smash{ Z^b \eqo \phi_{-\mathtt x^-_0 (b)} \centerdot } $ $\smash{ (\phi_{\mathtt x^+_1 (b)}\centerdot \psi_b \centerdot  \phi_{-\mathtt x^-_1 (b)} )\centerdot \ldots \centerdot (\phi_{\mathtt x^+_k (b)}\centerdot \psi_b \centerdot  \phi_{-\mathtt x^-_k (b)}) \centerdot \ldots \, }$.  

We next introduce the right continuous inverse $\smash{ (\mathtt J^{-1}_y)_{y\in \bbR_+}}$ of $\smash{ \mathtt J_\cdot }$, which is defined for all $\smash{ y\ino \bbR_+}$ by $\smash{ \mathtt J^{-1}_y\eqo \inf \{ t\ino \bbR_+: \mathtt J_t \geko y\}}$. We see that $\smash{ \mathcal H (\mathtt J^{-1}_y)\eqo 0}$. Thus, $\smash{ B_{\mathtt J^{-1}_y} \eqo Z^b ( y)}$, for 
all $\smash{y\ino \bbR_+}$.

We next specify $y$ in order to get $\cE_b B$ instead of $Z^b$ in the previous equality. 
To this end, we fix $z\ino \bbR_+$ and we denote by $\bnu(b,z)$ the number of positive times $\leqo z$ at which $\cE_b B$ reaches a local maximum: namely $\bnu(b,z)\eqo \min \{ k\ino \bbN : \mathtt{z}_{2k+1} \geko z\}$. Thus, $z$ in the parametrization of $\cE_b B_\cdot$ corresponds to $y\eqo z+2b \bnu(b, z)$ in the parametrization of $Z^b$, i.e., $Z^b \big( z+2b\bnu (b,z)\big)\eqo \cE_bB(z)$ for all $z\ino \bbR_+$. We set $K(b,z) \eqo \mathtt J^{-1}_{z+2b\bnu(b,z)}$. Since $\smash{ B_{\mathtt J^{-1}_y} \eqo Z^b ( y)}$, for 
all $\smash{y\ino \bbR_+}$, we get  
\begin{equation}
\label{BMtrifft2}
\smash{ \forall z\ino \bbR_+, \qquad B_{K(b,z)} \eqo \cE_b B ( z) \, .}
\end{equation}
We then prove for all $\smash{ t\ino \bbR_+}$ that
\begin{equation}
\label{BMcvprob1}
\smash{ \lim_{b\downarrow 0} K(b, t/b) = t \; \textrm{in $\bP$-probability.}}
\end{equation}
\emph{Indeed}, since $\smash{ \{ \mathtt z_j; j\ino \bbN\}}$ is a P.p.p.~on $\smash{ \bbR_+}$ whose intesity is $\smash{ b^{-1} \mathrm dt}$, standard arguments show that for all $\smash{ t\ino \bbR_+}$
\begin{equation}
\label{BMapproxnu}
\smash{ \lim_{b\to 0^+} 2b^2 \bnu(b, t/b) = t \; \textrm{in $\bP$-probability.}}
\end{equation}
We next recall $\smash{ \mathtt g^{_b}_{^k}}$, $\smash{ \mathtt r^{_b}_{^k}}$ and $\smash{ \mathtt s^{_b}_{^k}}$ from Notation \ref{notaeraBM}, with the convention that $\smash{ \mathtt s^{_b}_{^0}\eqo 0}$. For all $\smash{ k\ino \bbN^*}$, we set 
$\smash{ \zeta^{-}_{k-1}\! :=\!  \mathtt g^{_b}_{^{k-1}} \! -\! \mathtt s^{_b}_{^{k-1}}}$,  
$\smash{ \zeta^{+}_{k}\! :=\!  \mathtt r^{_b}_{^{k}} \! -\!  \mathtt g^{_b}_{^{k-1}} }$ and 
$\smash{ \zeta^{\mathrm{exc}}_k\! :=\!    \mathtt s^{_b}_{^{k}} \! -\!   \mathtt r^{_b}_{^{k}}} $. 
We also set $\smash{ V(b,0)\! :=\! 0}$ and $\smash{ V(b,k) \! :=\!  }$ $\smash{ V(b,k\! -\! 1)+ \zeta^{-}_{k-1}+ \zeta^{+}_{k}+ \zeta^{\mathrm{exc}}_k}$. We observe that 
\begin{equation}
\label{BMapproxiK}
\smash{ \forall z\ino \bbR_+, \quad V\big(b,\bnu(b,z)\big) \leq K(b,z)  < V\big(b,\bnu(b,z)+1\big) \; .}
\end{equation}
Then Theorem \ref{Brodecunicelth} and Lemma \ref{durationLapla} imply the following for all $\smash{ k\ino \bbN^*}$ and all $\smash{ \lambda \ino (0, \infty)}$, 
\begin{equation}
\label{BMapproxiKK}
\smash{ \bE \Big[ \mathtt{e}^{-\lambda V(b,k) } \Big] \eqo \bE \Big[ \mathtt{e}^{-\lambda V(b,1)} \Big]^k\eqo \Big( \mathrm{cosh} \big(b\sqrt{2\lambda \, }\big) \Big)^{-2k}= \Big(1+\lambda b^2 + \mathcal O_\lambda (b^4) \Big)^{-2k}. }
\end{equation}
Since $\smash{k\mapsto V(b,k)}$ increases, (\ref{BMapproxnu}) and (\ref{BMapproxiKK}) imply that 
$\smash{\lim_{b\downarrow 0} \bE [ \mathrm e^{-\lambda V(b,\bnu(b, t/b))} \big] \eqo \mathrm e^{-\lambda t}}$, which entails (\ref{BMcvprob1}) by (\ref{BMapproxiK}). \cq 

\smallskip

To  complete the proof of (\ref{cvproba2}) we use the following result. 
\begin{lemma}
\label{Dinicvproba} Let $\smash{((R^n_t)_{t\in \bbR_+})_{n\in \bbN}}$ be a sequence of $\smash{\bbR_+}$-valued nondecreasing processes. For all $t\ino \bbR_+$ we assume that $\smash{\lim_{n\to \infty} R^n_t\eqo t}$ in probability. Then, 
$\smash{\lim_{n\to \infty} \max_{t\in [0, t_0]} |R^n_t \! -\! t|\eqo 0 }$ in probability for all $\smash{t_0\ino\bbR_+}$.
\end{lemma}
\noi
\textbf{Proof.} It is a straightforward consequence of Dini's theorem. \cqfd 

\smallskip
 
By Lemma \ref{Dinicvproba} and (\ref{BMcvprob1}) for all $t_0\ino\bbR_+$ we thus get 
$\smash{\lim_{b\downarrow 0} \max_{t\in [0, t_0]}  |K(b,t/b) \! -\! t| \eqo 0 }$ in $\smash{\bP}$-probability which implies (\ref{cvproba2}) by  (\ref{BMtrifft2}) and easy arguments (see e.g.~Theorem 3.1 Whitt \cite{Wh80}). This completes the proof of (\ref{cvproba2}) and and also the proofs of Theorems \ref{cvzigzag} and  \ref{recovercell}.\cqfd 

\subsection{Proof of Proposition \ref{LLNnuc}}
\label{ProofLLNnucsec}
Let $\smash{x^+\! , \, x^-\ino \bbR_+}$, let $h\ino (0, \infty)$ and let $\smash{(\bH , \bzeta)}$ be distributed according to $\smash{P_{\! x^+\! , \, x^- \! , \, h}}$. 
For all $b\ino (0, h]$, we denote by $\smash{\mathtt C(\cE_b \bH)\eqo \big( (x^+_k(b, \bH), (x^+_k(b, \bH))\, ;\,  1\leqo k\leqo N_b(\bH)\big)}$, the cell population associated with the zigzag function $\smash{\cE_b\bH}$. 
For all continuous and bounded $\smash{F\! : \! \bbR_+^2 \! \to \! \bbR}$ we then set 
$\smash{\mathbf S_F(b)\eqo \sum_{1\leq k\leq N_b(\bH)} F \big( \frac{1}{b}x^+_k (b,\bH),\frac{1}{b}x^-_k (b,\bH) \big)} $. Thanks to Theorem \ref{mainth3} we see that Proposition \ref{LLNnuc} is a consequence from the following limit 
\begin{equation}
\label{LLNcvproba1}
\smash{2b^2\mathbf S_{\! F} (b)  \xrightarrow[b\downarrow 0]{\; } \langle F\rangle \bzeta  \quad \textrm{in $\bP$-probability,}}
\end{equation}
where $\smash{\langle F\rangle }$ stands for $\smash{\int_0^\infty \! \int_0^\infty \! F(y^+\! ,\, y^-) \mathrm{e}^{-(y^+\! + y^-)}\mathrm dy^+  \mathrm dy^- }$. 

To prove (\ref{LLNcvproba1}) we proceed as in the proofs of Theorems \ref{cvzigzag} and  \ref{recovercell} in the previous section and we first prove a similar result for the Brownian motion $\smash{B_\cdot}$. We recall $\smash{\mathtt x^+_k(b)}$ and $\smash{\mathtt x^-_k(b)}$ 
from Notation \ref{notaeraBM} and we recall that $\smash{(\mathtt{z}_j)_{j\in \bbN}}$ is defined as follows: 
$\smash{\mathtt{z}_0\eqo \mathtt{x}^-_0 (b)}$, $\smash{\mathtt{z}_{2k+1}\eqo  \mathtt{z}_{2k}+ \mathtt x^+_{k+1} (b)}$ and 
$\smash{\mathtt{z}_{2k+2}\eqo  \mathtt{z}_{2k+1}+ \mathtt x^-_{k+1} (b)}$ for all $\smash{k\in \bbN}$; 
$\smash{(\mathtt z_{2k+1})_{k\in \bbN}}$ is the sequence of positive times at which $\smash{\cE_b B_\cdot}$ 
reaches a local maximum. 
For all $\smash{z\ino \bbR_+}$, we also recall that $\smash{\bnu(b,z) \eqo \min \{ k\ino \bbN : \mathtt{z}_{2k+1} \geko z\}}$  is the number of positive times $\smash{\leqo z}$ at which $\smash{\cE_b B}$ 
reaches a local maximum. For all $\smash{z\ino \bbR_+}$, we set 
$$\smash{ S_{\! F} (b,z) \eqo \!\! \!\! \!\! \!\!  \sum_{^{\quad 1\leq k\leq  \bnu(b,z)}} \!\! \!\! \!\! \!\!  F \big( \tfrac{1}{b}\mathtt x^+_k (b),\tfrac{1}{b} \mathtt x^-_k (b) \big) .}$$
Then the law of large numbers combined with (\ref{BMapproxnu}) and Lemma \ref{Dinicvproba} imply for all $\smash{t_0\ino \bbR_+}$ that 
\begin{equation}
\label{LLNBrown}
\smash{\lim_{b\downarrow 0} \max_{t\in [0,t_0]} \big| 2b^2S_{\! F} (b,t/b) \! -\!  \langle F\rangle t \big|=0 \; \textrm{in $\bP$-probability.}}
\end{equation}
We then recall that $\smash{B_{b} (\cdot)}$ stands for the 
process in the left hand side of (\ref{cvproba2}). We recall from (\ref{hittim}) the notation $\smash{\varrho_{-x} (B_b)}$, which is equal to $\smash{b\varrho_{-x} (\cE_b B)}$, and also that $\smash{\bvaro_{-x}\eqo \varrho_{-x} (B)}$. Since the law of 
$\smash{\max \{ B^+_t ;\,  t\ino [0, \bvaro_{-x}]\}}$ is diffuse, (\ref{LLNBrown}) combined with (\ref{cvprobhitting}) imply for all $\smash{x, h \ino (0, \infty)}$ 
\begin{equation}
\label{LLNBrownhitting}
\smash{\lim_{b\downarrow 0} 2b^2S_{\! F} \big( b,\varrho_{-x} (\cE_b B) \big) = \langle F\rangle \bvaro_{-x}\;\,  \textrm{in $\bP \Big(\, \cdot \, \Big| \!\!\!\!\! \max_{\quad t\in [0, \bvaro_{-x}]}  \!\!\!\!\! B^+_t   \leko h \Big)$-probability.}}
\end{equation}

  From the proof in the previous section, we next recall the definition of the times $\gamma (b)$ and $\delta (b)$ for all $b\ino [0, h)$ and we set $\smash{S^{\mathrm{exc}}_{\! F} \big( b) \eqo S_{\! F} \big( b,\frac{1}{b} \delta (b) \big) \! -\! S_{\! F} \big( b,\frac{1}{b} \gamma (b) \big)}$. By (\ref{LLNBrown}) and (\ref{cvprobexcueqoh}), we get  
\begin{equation}
\label{LLNBrownexcu}
\smash{\lim_{b\downarrow 0^+} 2b^2S^{\mathrm{exc}}_{\! F} \big( b)  = \langle F\rangle \big(\delta (0)\! -\! \gamma (0) \big)\;\,  \textrm{in $\bP$-probability.}}
\end{equation}

Let $\smash{(H^i\! , \zeta_i)_{i\in \{ 1,2,3\}}}$ be independent $\smash{\bC_{\! f}}$-valued random paths such that 
$\smash{\bH\eqo H^1\! \centerdot H^2\! \centerdot H^3}$ and whose respective laws are given as follows: $\smash{(H^1\! , \zeta_1)}$ has law $\smash{P^{\uparrow}_{{\! x^+\! , \, h}}}$, $\smash{(H^2\! , \zeta_2)}$ has law $\smash{\bN (\, \cdot \, | \, \Gamma \eqo h)}$ and 
$\smash{(H^3\! , \zeta_3)}$ has law $\smash{P^{\downarrow}_{{\! x^-\! , \, h}}}$. It is easy to see that $\smash{\cE_b \bH\eqo \cE_bH^1\! \centerdot \cE_b H^2\! \centerdot  \cE_b H^3}$ and up to boundary errors $\smash{\mathbf S_{\! F} (b)}$ also splits into three corresponding terms 
$\smash{S_i(b)}$, $i\ino \{ 1,2,3\}$ which satisfy the following. 
\begin{compactenum}

\smallskip

\item[$(i)$] $\smash{\big|\mathbf S_{\! F} (b) -\! (S_1(b)+S_2(b)+S_3(b))  \big| \leq 4 \sup_{y\in \bbR_+} |F(y) |}$.

\smallskip

\item[$(ii)$] $\smash{\big(\overleftarrow{H}^1\! , \zeta_1, (S_1(b))_{b\in (0,h)} \big)\!  \overset{_{\textrm{(law)}}}{=}\!
 \big( B(\cdot \wedge \bvaro_{-x^+}), \bvaro_{-x^+} , \big( S_{\! F} ( b,\varrho_{-x^+} (\cE_b B) ) \big)_{b\in (0,h)} \big)}$ under the conditional probability $\smash{\bP \big(\, \cdot \, \big| \max \{ B^+_t  ;  t\in [0, \bvaro_{-x^+}]\}   \leko h \big)}$. 
 
\smallskip

\item[$(iii)$]  $\smash{\! \! \big(H^2\! , \zeta_2, (S_2(b))_{b\in (0,h)} \big)\!  \overset{_{\textrm{(law)}}}{=}\!\big( B \big( (\gamma (0) + \cdot) \! \wedge \! \delta(0) \big) \! -\! B(\gamma(0)),  \delta(0) \! -\! \gamma(0) , (S^{\mathrm{exc}}_{\! F} \big( b) )_{b\in (0,h)}\big) }$. 

\smallskip

\item[$(iv)$] $\smash{\big(H^3\! , \zeta_3, (S_3(b))_{b\in (0,h)} \big) \!  \overset{_{\textrm{(law)}}}{=}\!
 \big( B(\cdot \wedge \bvaro_{-x^-}), \bvaro_{-x^-} , \big( S_{\! F} ( b,\varrho_{-x^-} (\cE_b B) ) \big)_{b\in (0,h)} \big)}$ under the conditional probability $\smash{\bP \big(\, \cdot \, \big| \max \{ B^+_t  ;  t\in [0, \bvaro_{-x^-}]\}   \leko h \big)}$. 
 
\smallskip

\end{compactenum}

Then $(ii)$, $(iv)$ and (\ref{LLNBrownhitting}) imply that $\smash{\lim_{b\downarrow 0} 2b^2S_i(b) \eqo  \langle F\rangle\zeta_i}$ in $\bP$-probability for 
$i\ino \{ 1,3\}$. Similarly, $(iii)$ and (\ref{LLNBrownexcu}) imply that $\smash{\lim_{b\downarrow  0} 2b^2S_2(b) \eqo  \langle F\rangle\zeta_2}$ in $\bP$-probability.
By $(i)$ we therefore get 
$\smash{\lim_{b\downarrow 0} 2b^2\mathbf S_{\! F} (b) \eqo  \langle F\rangle (\zeta_1+ \zeta_2+  \zeta_3)\eqo  \langle F\rangle\bzeta }$ in $\bP$-probability. This proves (\ref{LLNcvproba1}) and the proof of Proposition \ref{LLNnuc} follows. \cqfd

\end{document}